\documentclass[12pt,openany]{amsart}
\counterwithin{equation}{section}

\allowdisplaybreaks

\usepackage{bbm}
\usepackage{amssymb}
\usepackage{amsmath}
\usepackage[all]{xy}
\usepackage{amsfonts}
\usepackage{mathrsfs}
\usepackage{latexsym}
\usepackage{tikz}
\usepackage{tikz-cd}
\tikzcdset{scale cd/.style={every label/.append style={scale=#1},
    cells={nodes={scale=#1}}}}

\newcommand{\diagramglobalscale}{1}

\tikzset{
  every picture/.append style={
    scale=\diagramglobalscale,
    transform shape
  }
}
    
\usetikzlibrary{patterns}
\usepackage{graphicx}
\usepackage{enumerate}
\usepackage{amscd,amssymb,amsmath,amsbsy,amsthm}
\definecolor{linkred}{rgb}{0.7,0.2,0.2}
	\definecolor{linkblue}{rgb}{0,0.2,0.6}
	\definecolor{linkgreen}{rgb}{0,0.6,0.2}
\usepackage[colorlinks,plainpages,
    linkcolor=linkblue,
    citecolor=linkgreen,
    urlcolor=linkred]{hyperref}
\usepackage{zref-clever}

\zcsetup{
  cap,                      
  nameinlink=false,         
  lastsep = {, and }        
}
\zcRefTypeSetup{assumption}{
    Name-sg = Assumption ,
    name-sg = assumption ,
    Name-pl = Assumptions ,
    name-pl = assumptions ,
}
\zcRefTypeSetup{corollary}{
    Name-sg = Corollary ,
    name-sg = corollary ,
    Name-pl = Corollaries ,
    name-pl = corollaries ,
}
\zcRefTypeSetup{definition}{
    Name-sg = Definition ,
    name-sg = definition ,
    Name-pl = Definitions ,
    name-pl = definitions ,
}
\zcRefTypeSetup{enumi}{
    Name-sg = {} ,
    name-sg = {} ,
    Name-pl = {} ,
    name-pl = {} ,
}
\zcRefTypeSetup{equation}{
    Name-sg = {} ,
    name-sg = {} ,
    Name-pl = {} ,
    name-pl = {} ,
}
\zcRefTypeSetup{example}{
    Name-sg = Example ,
    name-sg = example ,
    Name-pl = Examples ,
    name-pl = examples ,
}
\zcRefTypeSetup{lemma}{
    Name-sg = Lemma ,
    name-sg = lemma ,
    Name-pl = Lemmas ,
    name-pl = lemmas ,
}
\zcRefTypeSetup{prop}{
    Name-sg = Proposition ,
    name-sg = proposition ,
    Name-pl = Propositions ,
    name-pl = propositions ,
}
\zcRefTypeSetup{remark}{
    Name-sg = Remark ,
    name-sg = remark ,
    Name-pl = Remarks ,
    name-pl = remarks ,
}
\zcRefTypeSetup{theorem}{
    Name-sg = Theorem ,
    name-sg = theorem ,
    Name-pl = Theorems ,
    name-pl = theorems ,
}

\zcsetup{
  countertype = {
    enumi=item, enumii=item, enumiii=item, enumiv=item
  },
  counterresetby = {
    enumii=enumi, enumiii=enumii, enumiv=enumiii
  }
}
\zcRefTypeSetup{item}{
    Name-sg = {} ,
    name-sg = {} ,
    Name-pl = {} ,
    name-pl = {} ,
}

\AddToHook{env/assumption/begin}{\zcsetup{countertype={theorem=assumption}}}
\AddToHook{env/corollary/begin}{\zcsetup{countertype={theorem=corollary}}}
\AddToHook{env/definition/begin}{\zcsetup{countertype={theorem=definition}}}
\AddToHook{env/example/begin}{\zcsetup{countertype={theorem=example}}}
\AddToHook{env/lemma/begin}{\zcsetup{countertype={theorem=lemma}}}
\AddToHook{env/prop/begin}{\zcsetup{countertype={theorem=proposition}}}
\AddToHook{env/remark/begin}{\zcsetup{countertype={theorem=remark}}}

\usepackage{dsfont}

\newcommand{\cref}[1]{\zcref{#1}}
\newcommand{\Cref}[1]{\zcref[S]{#1}}

\newtheorem{theorem}{Theorem}[section]
\newtheorem{definition}[theorem]{Definition}
\newtheorem{lemma}[theorem]{Lemma}
\newtheorem{corollary}[theorem]{Corollary}
\newtheorem{prop}[theorem]{Proposition}

\newtheorem{example}[theorem]{Example}
\newtheorem{remark}[theorem]{Remark}

\newcommand{\red}[1]{{\color{red}#1}}

\def\Hom{\operatorname{Hom}}

\newcommand\C{\mathbb{C}}

\newcommand\N{\mathbb{N}}
\newcommand\Q{\mathbb{Q}}
\newcommand\R{\mathbb{R}}
\newcommand\Z{\mathbb{Z}}

\newcommand\kk{\Bbbk}
\newcommand\one{\mathds{1}}

\newcommand\ba{\mathbf{a}}
\newcommand\bA{\mathbf{A}}
\newcommand\bb{\mathbf{b}}

\newcommand\bi{\mathbf{i}}
\newcommand\bj{\mathbf{j}}

\newcommand\bT{\mathbf{T}}

\newcommand\bS{\mathbf{S}}

\newcommand{\bnu}{\mathbf{\nu}}

\newcommand{\fywtmod}{\ftrust\text{-wtmod}}

\newcommand\cC{\mathcal{C}}
\newcommand\cD{\mathcal{D}}

\newcommand\cF{\mathcal{F}}

\newcommand{\cR}{\mathcal{R}}
\newcommand\cS{\mathcal{S}}

\newcommand\cM{\mathcal{M}}
\newcommand\cN{\mathcal{N}}
\newcommand\cP{\mathcal{P}}

\newcommand\op{\textup{op}}
\newcommand\rev{\textup{rev}}

\newcommand{\obj}[1]{\operatorname{obj}(#1)}

\newcommand{\height}{\text{ht}}

\DeclareFontFamily{OT1}{pzc}{}
\DeclareFontShape{OT1}{pzc}{m}{it}{<-> s * [1.10] pzcmi7t}{}
\DeclareMathAlphabet{\mathpzc}{OT1}{pzc}{m}{it}

\newcommand{\Word}{\langle Q_0\rangle}
\newcommand{\Qui}{\mathcal{Q}}

\newcommand{\cRB}{{}^\tau_\mu\cR}
\newcommand{\objRB}[1]{{}^\tau\Word_{#1}}
\newcommand{\talpha}{{}^\tau{\alpha}}

\newcommand{\klrw}{{\red\cR}}
\newcommand{\intklrw}{{{}^\tau_{\mathbf{0}}\red\cR^{\lamseq}_{\text{int}}}}
\newcommand{\lamseq}{{\red{\underline{\lambda}}}}
\newcommand{\rlam}{\red{\lambda}}
\newcommand{\rmu}{\red{\mu}}
\newcommand{\dom}{P^{\text{dom}}}

\newcommand{\oklrwmu}{{}^\tau_\mu{\klrw}}
\newcommand{\klrwalg}{\klrw^{\lamseq}}
\newcommand{\oklrwalg}{\oklrwmu^{\lamseq}}
\newcommand{\ctP}{\tilde{\mathcal{P}}}
\newcommand{\bzero}{{\mathbf{0}}}
\newcommand{\oklrwzero}{{}^\tau_{\bzero}{\klrw}}

\newcommand{\trust}{{}^\tau Y_\mu^\lambda}

\newcommand{\Pol}{\mathsf{Pol}}
\newcommand{\redkappa}{\red{\kappa}}
\newcommand{\Polint}{{\Pol^{\lamseq}_{\Omega}}}

\newcommand\cEnd{\mathpzc{End}}     
\newcommand\cVect{\mathpzc{Vect}}     

\newcommand\bH{\mathbf{H}}
\newcommand\bR{\mathbf{R}}

\newcommand{\word}{\langle Q_0 \rangle}

\DeclareMathOperator{\End}{End}
\DeclareMathOperator{\ev}{ev}

\DeclareMathOperator{\id}{id}

\DeclareMathOperator{\sgn}{sgn}

\DeclareFontFamily{U}{wncy}{}
\DeclareFontShape{U}{wncy}{m}{n}{<->wncyi10}{}
\DeclareSymbolFont{mcy}{U}{wncy}{m}{n}
\DeclareMathSymbol{\Sha}{\mathord}{mcy}{"58} 
\DeclareMathSymbol{\Ize}{\mathord}{mcy}{"49} 
\DeclareMathSymbol{\Ya}{\mathord}{mcy}{"17} 
\newcommand{\tR}{\Ize}

\newcommand{\intmodnil}{\intklrw\text{-mod}_{\text{nil}}}

\usetikzlibrary{arrows.meta,patterns}
\usetikzlibrary{decorations.markings,decorations.pathreplacing}
\usepackage{tikz-cd}

\tikzset{multi/.style={very thick}}
\tikzset{
    anchorbase/.style={
        >=To,
        line cap = round, line join = round,
        baseline={([yshift=-0.5ex]current bounding box.center)},
    }
}

\tikzset{wei/.style={draw=red,double=red!40!white,double distance=1pt,thick}}

\newcommand{\ftrust}{{}^\tau F Y_\mu^\lambda}
\newcommand{\abtrust}{{}^\tau F Y^{\text{ab}}}

\newcommand{\strandlabel}[1]{$\scriptstyle{#1}$}
\newcommand{\botlabel}[1]{node[anchor=north] {\strandlabel{#1}}}
\newcommand{\toplabel}[1]{node[anchor=south] {\strandlabel{#1}}}
\newcommand{\rbotlabel}[1]{node[anchor=north,xscale=-1] {\strandlabel{#1}}}
\newcommand{\rtoplabel}[1]{node[anchor=south,xscale=-1] {\strandlabel{#1}}}
\newcommand{\braidup}{to[out=up,in=down]}

\newcommand{\coupon}[2]{ 
    \draw (#1) node[inner sep=2pt,draw,rounded corners,fill=black!20!white] {$\scriptstyle{#2}$}
}

\newcommand{\bdot}[1]{
   \filldraw[fill=black,draw=black] (#1) circle (1pt);
}

\definecolor{mirrorcolor}{RGB}{195,230,235}
\newcommand{\mirror}[3]{%
\ifdim#3pc<#2pc%
    \fill[mirrorcolor] (#1,#2) rectangle ({#1-0.04},#3);
    \fill[gray!60] (#1,#2) rectangle ({#1+0.04},#3);
    \fill[pattern=north east lines, pattern color=gray!60]
      ({#1+0.00},{#2+0.01}) rectangle ({#1+0.09},#3);
\else%
    \fill[mirrorcolor] (#1,#2) rectangle ({#1+0.04},#3);
    \fill[gray!60] (#1,#2) rectangle ({#1-0.04},#3);
    \fill[pattern=north east lines, pattern color=gray!60] ({#1-0.0},{#2-0.00}) 
      rectangle ({#1-0.1},{#3+0.00});    
\fi}

\newcommand{\rmirror}[3]{%
\ifdim#3pc<#2pc%
    \fill[mirrorcolor] (#1,#2) rectangle ({#1+0.04},#3);
    \fill[gray!60]    (#1,#2) rectangle ({#1-0.04},#3);
    \fill[pattern=north east lines, pattern color=gray!60]
      ({#1-0.09},{#2+0.01}) rectangle ({#1-0.00},#3);
\else%
    \fill[mirrorcolor] (#1,#2) rectangle ({#1-0.04},#3);
    \fill[gray!60]    (#1,#2) rectangle ({#1+0.04},#3);
    \fill[pattern=north east lines, pattern color=gray!60]
      ({#1+0.00},{#2-0.00}) rectangle ({#1+0.10},{#3+0.00});
\fi}

\title{On weight modules over truncated shifted iYangians}
\author{Yaolong Shen, Linliang Song and Rui Xiong}

\begin{document}

\begin{abstract}
Truncated shifted iYangians are a family of algebras expected to
quantize certain components of affine Grassmannian islices.
We introduce orientifold KLRW (oKLRW) algebras associated with quivers
with involution and establish their faithful polynomial representations
and diagrammatic bases. We also define KLR iYangians using double reflective KLR diagrams and construct diagrammatic realizations of the iGKLO homomorphisms. For integral parameters, we introduce interval oKLRW algebras and prove an equivalence between integral weight modules over truncated shifted iYangians and nilpotent modules over the corresponding interval oKLRW algebras.
\end{abstract}

\maketitle

\tableofcontents 

\section{Introduction}
\subsection{Background}
Quiver Hecke algebras, also known as KLR algebras, provide a powerful diagrammatic framework for the categorification of quantum groups \cite{KL09,Rou}.  Webster \cite{Web17a} later introduced KLRW algebras, which categorify tensor products of integrable highest-weight modules.  Diagrammatically, KLRW algebras are obtained by adding red strands to the usual black-strand KLR diagrams; these red strands are labeled by dominant weights and record the highest-weight data of the tensor product factors.

On the other hand, Braverman, Finkelberg, and Nakajima constructed Coulomb branches of three-dimensional $\mathcal N=4$ gauge theories \cite{BFN18}. For framed quiver gauge theories of simply-laced type, these Coulomb branches are closely related to slices in the affine Grassmannian, and their quantizations admit explicit realizations by difference operators. In this setting, Kamnitzer--Webster--Weekes--Yacobi \cite{KWWY14} introduced truncated
shifted Yangians, using the difference-operator representations originating in the work of Gerasimov--Kharchev--Lebedev--Oblezin \cite{GKLO05}.
These algebras provide quantizations of affine Grassmannian slices.

Kamnitzer--Tingley--Webster--Weekes--Yacobi \cite{KTWWY19} subsequently established a relationship between Coulomb branch algebras and Webster's diagrammatic tensor product categorifications. More precisely, they studied category $\mathcal O$ for truncated shifted Yangians and showed that, for generic integral parameters, this category is equivalent to a weight space in Webster's categorification of a tensor product of fundamental representations using KLRW algebras. 

The argument in \cite{KTWWY19} proceeds through a sequence of intermediate diagrammatic and Yangian-type algebras. On the diagrammatic side, they introduced the parity KLRW algebra, together with metric and coarse metric variants in which the strands carry additional integral ``longitude'' data determined by the parameters.  On the Yangian side, they replace the truncated shifted Yangian by a Morita equivalent flag version and relate it to the KLR Yangian, which is a diagrammatic algebra drawn on a cylinder. The desired equivalence of categories is then obtained by passing through these intermediate algebras; see \cite[Sect. 1.7]{KTWWY19} for details.

A further development of this picture was given in \cite{KWWY24}, where
parabolic induction and restriction functors for Coulomb branch algebras were
used to construct categorical Kac--Moody actions on categories of modules over
truncated shifted Yangians. Under the equivalence with flavoured KLRW algebras,
these functors correspond to the usual categorical action on KLRW module
categories.

More recently, a new class of gauge-theoretic constructions arising from quivers with involution has attracted much interest.
In these constructions, Yangians are replaced by twisted Yangians (also known as iYangians); see \cite{LZ24,LWZ25} for Drinfeld-type presentations.
Shifted iYangians associated to quasi-split Satake diagrams of ADE type are developed in \cite{LWW}, where PBW bases and a uniform family of iGKLO representations are constructed.
The images of these representations are referred to as truncated shifted iYangians.
In type AI, certain truncated shifted iYangians are further identified in \cite{LWW2} with the truncated shifted iYangians introduced in \cite{LPTTW25}, which are isomorphic to finite $W$-algebras of classical types.

In \cite{SSX25}, the authors study the Coulomb branch of the
three-dimensional $\mathcal N=4$ involution-fixed part of a quiver gauge
theory and construct an iGKLO-type homomorphism from the corresponding
shifted iYangian to the quantized Coulomb branch algebra. A related
construction in the second symmetric-power case is given in \cite{Wang25}.
It is expected that these constructions admit K-theoretic analogues,
related to the shifted affine iquantum groups developed in
\cite{LP26,LWW26}.

The main goal of this paper is to establish an orientifold analogue of the equivalence constructed in \cite{KTWWY19}. More precisely, we relate integral weight modules over
truncated shifted iYangians to nilpotent modules over interval
oKLRW algebras. The comparison proceeds through a flag version of the
truncated shifted iYangian and a diagrammatic algebra, namely the
KLR iYangian.

\subsection{Main results}
\label{sec:qwi}
Assume throughout the paper that there is an involution $\tau$ on $\Qui=(Q_0,Q_1,s,t)$ satisfying
\begin{itemize}
    \item 
    $s(\tau(h))=\tau(t(h))$ and $ t(\tau(h))=\tau(s(h))$ for all $h \in Q_1$; 
    \item $\tau(s(h))=t(h)$ if and only if $\tau(h)=h$. 
\end{itemize}
We remark that the pair $(\Qui,\tau)$ is called a \emph{quiver with involution} or an \emph{iquiver} in the literature; see \cite{EK07} and \cite{VV11}. Through the rest of the paper, we assume that $Q_0$ has no fixed point under $\tau$. Note that the same assumption is also made in \cite{En09,VV11,SSX25}. We let $Q_1^{\tau}$ denote the set of arrows which are $\tau$-invariant.

The diagrammatic starting point for our construction is the orientifold-KLR (oKLR) algebra. In fact, the oKLR algebra is a variant of the KLR algebra in which the usual black-strand diagrams are allowed to interact with a reflecting boundary, which we draw as a mirror. When a black strand meets the mirror, its label is changed from $i$ to \(\tau i\). Such algebras arose from the study of the affine Hecke algebra of type B by Varagnolo--Vasserot \cite{VV11}; see also \cite{Prz23}.

Our first step is to introduce orientifold KLRW (oKLRW) algebras, defined in \cref{eq:oklrwalg}. These algebras combine the mirror formalism of oKLR algebras with Webster's red-strand diagrammatics. The mirror encodes the involution on the quiver, while the red strands encode the dominant-weight data. We construct polynomial representations for these algebras in
\cref{prop:polyrep}, prove a diagrammatic basis theorem in
\cref{oklrwbasis}, and deduce the faithfulness of the polynomial
representations in \cref{cor:faithfulpolyrep}. The construction is compatible with both the oKLR and KLRW polynomial representations, and fits into the following diagram:
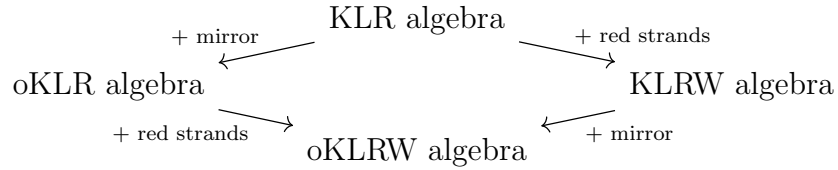
\[
\begin{tikzcd}[row sep=tiny]
&\text{KLR algebra}  \arrow[rd,"\text{+ red strands}"] \arrow[ld,"\text{+ mirror}"']& \\
\text{oKLR algebra} \arrow[rd,"\text{+ red strands}"'] & & \text{KLRW algebra}  \arrow[ld,"\text{+ mirror}"] \\
& \text{oKLRW algebra} &
\end{tikzcd}
\]
We also show that, in the diagonal case where the quiver is the disjoint union of two copies interchanged by the involution, the oKLRW algebra is Morita equivalent to the ordinary KLRW algebra for one copy; see \cref{thm:diagonal}.

On the Yangian side, we introduce KLR iYangians. They are diagrammatic algebras, defined in \cref{def:Ya} using the
double reflective KLR diagrams of \cref{def:twomirrordiagram}, in which
black strands move between two facing mirrors.
Here is an example of a double reflective KLR diagram:
$$
\begin{tikzpicture}
        [anchorbase,scale=1.2]
\mirror{0}{0}{2};
\mirror{5}{2}{0};
\draw (0.5,0) \braidup (1.5,2);
\draw (1.0,0) \braidup (2.5,2);
\draw (1.5,0) \braidup (2,0.75) \braidup (0.5,2);
\draw (2,0) to [out=50,in=-180+35] (5,1) to [out=180-35, in=down] (4,2); 
\draw (2.5,0) to [out = 180-50, in = -15] (0,1.2) to [out = 15, in = down] (1.0,2); 
\draw (3,0) \braidup (3.75,0.6) \braidup (3.0,1.3)\braidup (3.5,2); 
\draw (3.5,0) \braidup (2.0,2);
\draw (4,0) to [out=50,in=-180+35] (5,1.6) to [out=180-35, in=down] (4.5,2); 
\draw (4.5,0) \braidup (3.0,2); 
\bdot{2.0,0.75};
\bdot{2.65,0.45};
\bdot{2.95,0.85};
\bdot{0.56,1.75};
\bdot{1.45,1.65};
\bdot{1.50,1.85};
\bdot{3.53,1.15};
\bdot{4.50,1.30};
\end{tikzpicture}
$$
In \cref{thm:tRPol}, we construct a faithful polynomial representation
in which the two mirror generators act by reflection operators, with
the left reflection carrying a factor determined by the framing
parameters.

We then use these algebras to formulate a diagrammatic version of the iGKLO construction appearing in \cite{SSX25,LWW}. More specifically, we choose an idempotent $e(\aleph)$ in $\tR$ and,
using the iGKLO formulas, embed the truncated shifted iYangian
$\trust$ into the corner $e(\aleph)\tR e(\aleph)$; see \cref{lambsh}.  This gives an inclusion
\(
    \trust \hookrightarrow e(\aleph)\tR e(\aleph).
\)

To study the weight modules over $\trust$, we also introduce a flag version of the truncated shifted iYangian, following the strategy of \cite{KTWWY19}.  This flag algebra $\ftrust$ is generated by the image of the truncated shifted iYangian together with polynomial multiplication operators and the appropriate nilHecke operators.  We prove in \cref{thm:morita} that $\ftrust$ is Morita equivalent to $\trust$.  Moreover, in \cref{thm:flag=ab}, we also obtain explicit diagrammatic generators of $\ftrust$ for later use. 

Then we consider integral weight modules over $\ftrust$. The integrality condition allows us to encode generalized weight spaces by configurations of black strands placed in intervals determined by the parameters. This leads to the definition of the interval oKLRW algebra $\intklrw$, which is an idempotent truncation of the oKLRW algebra whose objects are precisely the interval configurations relevant to integral weights.

The interval oKLRW algebra is generated by local nilHecke operators inside each interval, together with long crossings and long reflections which move black strands between different intervals and across the mirror; see \cref{prop:intgen}. In particular, the polynomial actions of generators of $\intklrw$ match the action of the explicit diagrammatic generators of $\ftrust$ obtained in \cref{thm:flag=ab}.  In this way, the interval oKLRW algebra provides the diagrammatic algebra which controls the integral weight spaces of $\ftrust$.

We write $\fywtmod$ to denote the category of finitely generated $\ftrust$-weight modules with integral weights. Let $\intmodnil$ denote the category of finitely generated $\intklrw$-modules on which all dot generators act nilpotently. The main result of the paper is the following:
\begin{theorem}[\cref{thm:equivalence}] 
    We have the following equivalence of categories\begin{equation*}
   \fywtmod
    \cong
   \intmodnil .
\end{equation*}
\end{theorem}
More precisely, we construct functors (see \cref{prop:Theta,prop:Gamma})
\[
    \Theta:
   \intmodnil
    \longrightarrow
   \fywtmod,
    \qquad
    \Gamma:
   \fywtmod
    \longrightarrow
   \intmodnil,
\]
and prove that they are mutually inverse equivalences. The functor $\Theta$ reconstructs a weight module by assigning the
summand $e(\nu_{\mathbf a})M$ to each integral weight $\mathbf a$,
while $\Gamma$ takes the direct sum of the weakly increasing integral
generalized weight spaces and equips it with the interval oKLRW action.
Combining this equivalence with the Morita equivalence between $\ftrust$ and $\trust$, we obtain a diagrammatic description of integral weight modules over truncated shifted iYangians.

\subsection{Further directions}

We end the introduction by mentioning several directions for future work. First, as mentioned in \cref{rem:extend}, it would be natural to extend the construction of oKLRW algebras to iquivers with fixed vertices.  The corresponding oKLR algebras and their polynomial representations have already been studied in \cite{Prz23}, and one expects that the oKLRW construction admits an analogous extension by adding Webster's red strands to that setting.  This should lead to diagrammatic iGKLO homomorphisms compatible with the shifted iYangians attached to more general quasi-split Satake diagrams.

Second, one can ask for a categorification-theoretic interpretation of oKLRW algebras.  Ordinary KLRW algebras categorify tensor products of integrable highest-weight modules; see \cite{Web17a}.  Since oKLR algebras are closely related to quantum symmetric pairs and iquivers, it is natural to ask whether oKLRW algebras categorify tensor product representations in quantum-symmetric-pair setting. 

Third, it would be desirable to develop an intrinsic analogue of category $\mathcal O$ for shifted iYangians.
Webster \cite{Web16} constructs categories $\mathcal O$ for various quantized Coulomb branches of cotangent type and relates them, via Koszul duality, to the corresponding Higgs-side categories $\mathcal O$.
In the fixed-point-free setting considered in this paper, the Coulomb branches associated to quivers with involution are still of cotangent type; see \cite{SSX25}.
It would therefore be interesting to understand how Webster's construction applies in our setting and, in particular, how the resulting category $\mathcal O$ is reflected on the oKLRW side.

\subsection*{Acknowledgment}
We would like to thank Kang Lu, Changjian Su, Ben Webster and Weiqiang Wang for useful discussions. YS is partially supported by the Fields Institute.  LS is partially supported by NSFC (Grant No. 12071346).

\section{Orientifold KLRW algebra}
In this section, we review the definition and basic properties of KLR algebras and orientifold KLR (oKLR) algebras following \cite{KL09,Rou,VV11,Prz23}, with particular emphasis on their polynomial representations. We then introduce an extension of the oKLR algebra by adding red strands, analogous to the passage from KLR algebras to KLRW algebras; see \cite{KTWWY19}.

\subsection{KLR algebras}
\label{sec:KLR}
Fix once and for all an algebraically closed ground field $\kk$.
Let $\Qui = (Q_0, Q_1, s, t)$ be a loop-free quiver, where $Q_0$ denotes the set of vertices and $Q_1$ the set of arrows.  
For any $h \in Q_1$, we write $s(h)$ (resp.\ $t(h)$) for the source (resp.\ target) of the arrow $h$.  
For $i, j \in Q_0$, let $m_{i,j}$ denote the number of arrows from $i$ to $j$.

The associated (symmetric) \emph{Cartan matrix} $C = (c_{i,j})_{i,j \in Q_0}$ is defined by
\[
c_{i,i} := 2, 
\qquad 
c_{i,j} := - m_{i,j} - m_{j,i} \quad \text{for } i \neq j .
\]
The matrix $C$ determines a symmetric Kac--Moody algebra $\mathfrak{g}$ with Cartan subalgebra $\mathfrak{h}$. 

We fix a choice of root datum for $\mathfrak{g}$.  
This determines a \emph{weight lattice} $P$, which is a finitely generated free abelian group equipped with a symmetric bilinear form
\[
P \times P \to \Q, 
\qquad 
(\lambda, \mu) \mapsto \lambda \cdot \mu ,
\]
and containing \emph{simple roots} $(\alpha_i)_{i \in Q_0}$ and \emph{fundamental weights} $(\varpi_i)_{i \in Q_0}$ satisfying
\[
\alpha_i \cdot \alpha_j = c_{i,j},
\qquad
\alpha_i \cdot \varpi_j = \delta_{i,j},
\quad \text{for all } i,j \in Q_0 .
\]
A weight $\lambda\in P$ is said to be \emph{dominant} if $\lambda\in \displaystyle\bigoplus_{i\in Q_0}\N \varpi_i$. Let $\dom$ denote the set of all dominant weights. 


The \emph{root lattice} is
\[
X := \bigoplus_{i \in Q_0} \Z \alpha_i \subset P ,
\]
and we set
\[
X^+ := \bigoplus_{i \in Q_0} \N \alpha_i \subset X .
\]
For $\alpha = \sum_{i \in Q_0} c_i \alpha_i \in X^+$, the \emph{height} of $\alpha$ is defined by
\(
\height(\alpha) := \sum_{i \in Q_0} c_i .
\)

Finally, let $\Word$ denote the set of all words in the alphabet $Q_0$.  
For $\alpha \in X^+$ of height $n$, let $\Word_\alpha \subset \Word$ consist of all words
\(
\bi = i_1 \cdots i_n
\)
such that
\[
\alpha_{\bi}:=\alpha_{i_1} + \cdots + \alpha_{i_n} = \alpha .
\]
The symmetric group $S_n$ acts on $\Word_\alpha$ by permuting the letters in the natural way. We denote by $s_i$ for $i=1,2,\ldots,n-1$ the simple reflections of $S_n$.

It is common and convenient to describe KLR algebras diagrammatically. Throughout, we assume that the underlying quiver is \emph{simply laced}, that is,
\[
m_{i,j} + m_{j,i} \leq 1 \qquad \text{for all } i \neq j .
\]
Let $\cR$ be the graded $\kk$-linear strict monoidal category generated by objects
$i$ for $i \in Q_0$, together with homogeneous morphisms
\[x=
\begin{tikzpicture}
    [anchorbase]
    \draw (0,0) \botlabel{i} -- (0,.4)\toplabel{i};
     \bdot{0,.2};
\end{tikzpicture} \colon i \to i \quad \text{of degree } 2,
\qquad \sigma=
\begin{tikzpicture}
    [anchorbase]
    \draw (0,0) \botlabel{i} \braidup (.4,.4) \toplabel{i};
    \draw (.4,0) \botlabel{j} \braidup (0,.4) \toplabel{j};
\end{tikzpicture}  \colon ij \to ji \quad \text{of degree } - \alpha_i \cdot \alpha_j,
\]
for all $i,j \in Q_0$ subject to the following generating relations:
\begin{gather}
\label{dotcross}
    \begin{tikzpicture}
        [baseline=.6]
         \draw (0,0) \botlabel{i} \braidup (.4,.4);
    \draw (.4,0) \botlabel{j} \braidup (0,.4);
    \bdot{.04,.3};
    \end{tikzpicture}
    =
     \begin{tikzpicture}
        [baseline=.6]
         \draw (0,0) \botlabel{i} \braidup (.4,.4);
    \draw (.4,0) \botlabel{j} \braidup (0,.4);
    \bdot{.36,.1};
    \end{tikzpicture}
    +\delta_{i,j}
    \begin{tikzpicture}
         [baseline=.6]
         \draw (0,0) \botlabel{i} \braidup (0,.4);
    \draw (.4,0) \botlabel{j} \braidup (.4,.4);
    \end{tikzpicture},
    \qquad
        \begin{tikzpicture}
        [baseline=.6]
         \draw (0,0) \botlabel{i} \braidup (.4,.4);
    \draw (.4,0) \botlabel{j} \braidup (0,.4);
    \bdot{.04,.1};
    \end{tikzpicture}
    =
     \begin{tikzpicture}
        [baseline=.6]
         \draw (0,0) \botlabel{i} \braidup (.4,.4);
    \draw (.4,0) \botlabel{j} \braidup (0,.4);
    \bdot{.36,.3};
    \end{tikzpicture}
    +\delta_{i,j}
    \begin{tikzpicture}
         [baseline=.2]
         \draw (0,0) \botlabel{i} \braidup (0,.4);
    \draw (.4,0) \botlabel{j} \braidup (.4,.4);
    \end{tikzpicture}, \\
    \label{doublecross}
    \begin{tikzpicture}
        [baseline=.8em]
         \draw (0,0) \botlabel{i} \braidup (.4,.4) \braidup (0,.8);
    \draw (.4,0) \botlabel{j} \braidup (0,.4) \braidup (.4,.8);
    \end{tikzpicture}
    =\delta_{i\neq j}
    \begin{tikzpicture}
        [baseline=1em]
         \draw (0,0) \botlabel{i} \braidup (0,.8);
    \draw (.4,0) \botlabel{j} \braidup (.4,.8);
    \coupon{.2,.4}{f_{ij}(y_1,y_2)};
    \end{tikzpicture}, \\
    \label{braidA}
    \begin{tikzpicture}
        [baseline=1em]
        \draw (0,0) \botlabel{i} \braidup (1,1);
    \draw (1,0) \botlabel{k} \braidup (0,1);
    \draw (.5,0) \botlabel{j} \braidup (0,.5) \braidup (.5,1);
    \end{tikzpicture}
    =
        \begin{tikzpicture}
        [baseline=1em]
        \draw (0,0) \botlabel{i} \braidup (1,1);
    \draw (1,0) \botlabel{k} \braidup (0,1);
    \draw (.5,0) \botlabel{j} \braidup (1,.5) \braidup (.5,1);
    \end{tikzpicture}
    +\delta_{i=k\neq j}
    \begin{tikzpicture}
        [baseline=1em]
        \draw (0,0) \botlabel{i} \braidup (0,1);
    \draw (.8,0) \botlabel{k} \braidup (.8,1);
    \draw (.4,0) \botlabel{j} \braidup (.4,.5) \braidup (.4,1);
    \coupon{.4,.5}{\frac{f_{ij}(y_3,y_2)-f_{ij}(y_1,y_2)}{y_3-y_1}};
    \end{tikzpicture},
\end{gather}
where we let $y_k$ denote the dot on the $k$th black strand read from
left to right and
\[X_{ij}(u,v)= \begin{cases}
  1 & i \not \leftarrow j\\
 u-v & i\leftarrow j\\
\end{cases} \qquad f_{ij}(u,v)=X_{ij}(u,v)X_{ji}(v,u)=
\begin{cases}
  1 & i \not \leftrightarrow j\\
 u-v & i\leftarrow j\\
v-u & i\to j
\end{cases}
\]
We refer to $\cR$ as the quiver Hecke category. Then we define the {\em KLR algebra} associated with $\alpha\in X^+$ to be 
\[
\cR(\alpha):=\bigoplus_{\bi,\bj\in \Word_{\alpha}} \Hom_{\cR}(\bi,\bj)
\]
 so that multiplication in $\cR(\alpha)$ corresponds to vertical composition of morphisms in $\cR$. There is an obvious isomorphism of categories $\bT:\cR \cong \cR^\op$ given by reflection with respect to the horizontal axis.

\subsection{oKLR algebras}
\label{sec:oKLR}

We refer to \cite[Definition~5.1]{VV11} and \cite[Definition~2.4]{Prz23} for the algebraic definition of the oKLR algebra.  
For our purposes, we adopt a diagrammatic description of the oKLR algebra, in the spirit of Section~\ref{sec:KLR}.  
In particular, the construction depends on the choice of a dominant weight.  
For any dominant weight $\mu\in \dom$ of the form
\(
\mu = \sum_{i \in Q_0} \mu_i \varpi_i ,
\)
we define
\begin{gather}
    \label{fi}
    f^\mu_i(u)=(-u)^{\mu_i} u^{\mu_{\tau i}}, \text{ for all }i\in Q_0.
\end{gather}

Our interpretation of the oKLR algebra comes from a module category over the quiver Hecke category $\cR$. Let $(\cC,\otimes,\one )$ be a $\kk$-linear strict monoidal category, and let $\cM$ be a $\kk$-linear category.  Let $\cEnd(\cM)$ denote the strict monoidal category of $\kk$-linear endofunctors of $\cM$.  A \emph{right action} of $\cC$ on $\cM$ is a monoidal functor $\bA \colon \cC \to \cEnd(\cM)^\rev$, where $\cD^\rev$ denotes the \emph{reverse} of a monoidal category $\cD$, where we reverse the tensor product.  Given such a right action, we also say that $\cM$ is a \emph{right module category} over $\cC$, or \emph{$\cC$-module} for short; see \cite[Section 7.1]{ENGO15}.  For objects $C \in \cC$ and $M \in \cM$, we set
\begin{equation} \label{beaver}
    M \otimes C := \bA(C)(M).
\end{equation}
We say that the $\cC$-module $\cM$ is \emph{strict} if $\bA$ is a strict monoidal functor.  Equivalently, $\cM$ is strict if
\[
    M \otimes (C_1 \otimes C_2) = (M \otimes C_1) \otimes C_2,\qquad
    \text{for all } M \in \cM,\ C_1,C_2 \in \cC,
\]
and similarly for morphisms. The notion of a \emph{presentation} of $\mathcal{C}$-modules by generators and relations is introduced in \cite[Sect.~2.2]{SSS25}, to which we refer for details. 

Let $\cM$, $\cN$ be strict $\cC$-modules.  A \emph{morphism of $\cC$-modules} from $\cM$ to $\cN$ is a pair $(\bH,\omega)$, where $\bH$ is a $\kk$-linear functor from $\cM$ to $\cN$ and $\omega$ is a natural isomorphism with components
\[
    \omega_{M,C} \colon \bH(M \otimes C) \xrightarrow{\cong} \bH(M) \otimes C,\qquad M \in \cM,\ C \in \cC,
\]
such that the diagram
\begin{equation} \label{buffalo}
    \begin{tikzcd}
        & \bH (M \otimes C \otimes D) \arrow[dl,"\omega_{M \otimes C, D}"'] \ar[dr,"\omega_{M, C \otimes D}"] &
        \\
        \bH(M \otimes C) \otimes D \arrow[rr,"\omega_{M,C} \otimes 1_D"] & & \bH(M) \otimes C \otimes D
    \end{tikzcd}
\end{equation}
commutes for all $M \in \cM$ and $C,D \in \mathcal{C}$.  (See \cite[Def.~7.2.1]{ENGO15} for a more general definition, where the module categories are not required to be strict.)  We say that a morphism of $\cC$-modules is \emph{strict} if $\omega_{M,C}$ is the identity morphism for all $M \in \cM$ and $C \in \cC$.  An \emph{equivalence of $\cC$-modules} is a morphism of $\cC$-modules that is also an equivalence of categories.

\begin{definition}
\label{defcRB}
    The iquiver Hecke category associated with a dominant weight $\mu$, denoted as $\cRB$, is the $\cR$-module generated by the morphisms
    \begin{gather*}
        \begin{tikzpicture}
            [anchorbase]
            \mirror{-.5}{0}{1};
            \draw (0,0) \botlabel{i} to[out=up,in=right] (-.5,.5) to[out=right,in=down] (0,1)\toplabel{\tau i};
        \end{tikzpicture} \colon i \to \tau i,\quad \text{ for all } i\in Q_0,
    \end{gather*}
    subject to the relations (for all $i,j\in Q_0$)
    \begin{gather}
    \label{quadratic}
        \begin{tikzpicture}
            [anchorbase]
            \mirror{-.25}{0}{1};
             \draw (0,0) \botlabel{i} to[out=up,in=right] (-.25,.25) to[out=right,in=down] (0,.5) to [out=up,in=right] (-.25,.75) to[out=right,in=down] (0,1)\toplabel{i};
        \end{tikzpicture}
        =
         \begin{tikzpicture}
            [anchorbase]
            \mirror{-.25}{0}{1};
             \draw (.6,0) \botlabel{i} -- (.6,1)\toplabel{i};
             \coupon{.6,.5}{f^\mu_{i}(-y_1)};
        \end{tikzpicture},
        \qquad
             \begin{tikzpicture}
            [anchorbase]
            \mirror{-.5}{0}{1};
            \draw (0,0) \botlabel{i} to[out=up,in=right] (-.5,.5) to[out=right,in=down] (0,1)\toplabel{\tau i};
            \bdot{-.1,.3};
        \end{tikzpicture}
        =-\ 
         \begin{tikzpicture}
            [anchorbase]
            \mirror{-.5}{0}{1};
            \draw (0,0) \botlabel{i} to[out=up,in=right] (-.5,.5) to[out=right,in=down] (0,1)\toplabel{\tau i};
            \bdot{-.1,.7};
        \end{tikzpicture},\\
    \label{reflection}
    \begin{tikzpicture}
        [anchorbase]
        \mirror{-.25}{0}{1};
       \draw (0,0) \botlabel{i} to[out=up,in=-15] (-.25,.2) to [out=15,in=down] (.3,.6) \braidup (0,1) \toplabel{\tau i};
        \draw (.3,0) \botlabel{j} to[out=up,in=right] (-.25,.5)  to [out=right,in=down] (.3,1) \toplabel{\tau j};
    \end{tikzpicture}
    -
    \begin{tikzpicture}
        [anchorbase]
        \mirror{-.25}{0}{1};
        \draw (0,0) \botlabel{i} \braidup (.3,.4) to [out=up,in=-15] (-.25,.8) to [out=15,in=down] (0,1) \toplabel{\tau i};
        \draw (.3,0) \botlabel{j} to[out=up,in=right] (-.25,.5)  to [out=right,in=down] (.3,1) \toplabel{\tau j};
    \end{tikzpicture}
    =\ \delta_{\tau i,j}\ 
        \begin{tikzpicture}
            [anchorbase]
            \mirror{-.6}{0}{1.2};
            \draw (.4,0)\botlabel{i} \braidup (.8,.4) -- (.8,1.2) \toplabel{i};
            \draw (.8,0)\botlabel{j} \braidup (.4,.4) -- (.4,1.2) \toplabel{j};
            \coupon{.6,.7}{\frac{f^\mu_i(y_1)-f^\mu_j(y_2)}{y_1+y_2}};
        \end{tikzpicture}. 
    \end{gather}
We refer to $\lambda$ as the framing weight of $\cRB$.
\end{definition} 
\begin{remark}
    Note that our convention for \cref{reflection} differs from that in \cite{VV11,Prz23} by a sign.  
We adopt this choice so that the polynomial representation of $\cRB$ is compatible with that of the KLRW algebras; see \cref{Polklrw1,Polklrw2}.  
This sign discrepancy ultimately comes from the fact that the polynomial representation used in \cite[Proposition~2.7]{Prz23} differs by a sign from that in \cite[Lemma~4.12]{Web17a} on the generators
\(
\begin{tikzpicture}[anchorbase]
    \draw (0,0) \botlabel{i} \braidup (.4,.4) \toplabel{i};
    \draw (.4,0) \botlabel{i} \braidup (0,.4) \toplabel{i};
\end{tikzpicture}
\)
for all $i\in Q_0$.
\end{remark}

\begin{remark}
\label{embed}
In $\cRB$, the identity morphism of the unit object $\one$ is represented by the diagram
\[
\begin{tikzpicture}[anchorbase]
    \mirror{-.5}{0}{1};
\end{tikzpicture}
\; : \; \one \to \one ,
\]
which we refer to as a {mirror}.
 The relations \cref{quadratic,reflection} are diagrammatic interpretations of relations in \cite[Def. 5.1]{VV11}. 
\end{remark}
For any $\alpha\in X^+$, we define $\talpha:=\tau \alpha+\alpha$ and  
\begin{gather}
\label{objRB}
    \objRB{\alpha}:=\{ \bi\in \Word\mid \talpha=\talpha_{\bi} \}.
\end{gather}
Then the oKLR algebra $\cRB(\alpha)$ associated with $\alpha$ and $\lambda$ is defined to be 
\begin{gather}
\label{oKLRalgebra}
\cRB(\alpha):= \bigoplus_{\bi,\bj\in \objRB{\alpha}}\Hom_{\cRB}(\bi,\bj).
\end{gather}
Note that $\Word_{\alpha}\subset \objRB{\alpha}$. By \cref{embed}, we may regard $\cR(\alpha)$ as a natural subalgebra of $\cRB(\alpha)$. 
\begin{lemma}
    There is a  $\kk$-linear isomorphism of categories \[\tilde{\bT}:\cRB \rightarrow \cRB^{\op}\] given by reflection with respect to the horizontal axis.
\end{lemma}
\begin{proof}
  Since reflection across the horizontal axis induces a $\kk$-linear isomorphism of categories
\(
\cR \cong \cR^{\op},
\)
it suffices to verify that the relations in \cref{quadratic,reflection} are preserved under $\tilde{\bT}$.  
This is straightforward for \cref{quadratic}.  
For \cref{reflection}, assume that $\tau i = j$. Then necessarily $i \neq j$, and the claim follows from \cref{dotcross}.
\end{proof}




\subsection{KLRW algebras}
We now recall the KLRW algebra following \cite[Definition~4.1]{Web17a}.  
In \cite{Web17a}, it is defined as the $\kk$-span of Stendhal diagrams modulo certain local relations. For our purposes, we instead present the definition in the language of strict monoidal categories, in parallel with the approach used for $\cR$.
\begin{definition} {\rm (\cite[Def.~4.1,Def.~4.4]{Web17a})}
    \label{KLRWdef}
    The KLRW category $\klrw$ is the graded $\kk$-linear strict monoidal category generated by objects $i$ and $\rlam$ for all $i\in Q_0$ and all $\rlam\in \dom$ together with homogeneous morphisms
    \begin{gather*}
\begin{tikzpicture}
    [anchorbase]
    \draw (0,0) \botlabel{i} -- (0,.4)\toplabel{i};
     \bdot{0,.2};
\end{tikzpicture} \colon i \to i \quad \text{of degree } 2,
\qquad 
\begin{tikzpicture}
    [anchorbase]
    \draw (0,0) \botlabel{i} \braidup (.4,.4) \toplabel{i};
    \draw (.4,0) \botlabel{j} \braidup (0,.4) \toplabel{j};
\end{tikzpicture}  \colon ij \to ji \quad \text{of degree } - \alpha_i \cdot \alpha_j, \\
\begin{tikzpicture}
    [anchorbase]
    \draw (0,0) \botlabel{i} \braidup (.4,.4) \toplabel{i};
    \draw[red,thick] (.4,0) \botlabel{\rlam} \braidup (0,.4) \toplabel{\rlam};
\end{tikzpicture}  \colon i\rlam \to \rlam i \quad \text{of degree }  \alpha_i \cdot \rlam, \quad
\begin{tikzpicture}
    [anchorbase]
    \draw[red,thick] (0,0) \botlabel{\rlam} \braidup (.4,.4) \toplabel{\rlam};
    \draw (.4,0) \botlabel{i} \braidup (0,.4) \toplabel{i};
\end{tikzpicture}  \colon \rlam i \to i\rlam  \quad \text{of degree }  \alpha_i \cdot \rlam,
    \end{gather*}
for all $i,j \in Q_0$ and all dominant weights $\rlam$, subjecting to \cref{dotcross}--\cref{braidA}  and the following generating relations:
\begin{gather}
\label{redbraid1}
\begin{tikzpicture}
        [baseline=1em]
        \draw (0,0) \botlabel{i} \braidup (1,1);
    \draw (1,0) \botlabel{j} \braidup (0,1);
    \draw[red,thick] (.5,0) \botlabel{\rlam} \braidup (0,.5) \braidup (.5,1);
    \end{tikzpicture}
    ~=~
        \begin{tikzpicture}
        [baseline=1em]
        \draw (0,0) \botlabel{i} \braidup (1,1);
    \draw (1,0) \botlabel{j} \braidup (0,1);
    \draw[red,thick] (.5,0) \botlabel{\rlam} \braidup (1,.5) \braidup (.5,1);
    \end{tikzpicture}
    ~+\delta_{i,j}\sum_{a+b+1=\rlam_i}~
     \begin{tikzpicture}
        [baseline=1em]
        \draw (0,0) \botlabel{i} \braidup (0,1);
        \bdot{0,.5};
        \bdot{1,.5};
        \node at (-.2,.5){$\scriptstyle a$};
        \node at (1.2,.5){$\scriptstyle b$};
    \draw (1,0) \botlabel{j} \braidup (1,1);
    \draw[red,thick] (.5,0) \botlabel{\rlam} \braidup (.5,.5) \braidup (.5,1);
    \end{tikzpicture}, \\
\label{redbraid2}
\begin{tikzpicture}
        [baseline=1em]
        \draw (0,0) \botlabel{i} \braidup (1,1);
    \draw[red,thick] (1,0) \botlabel{\rlam} \braidup (0,1);
    \draw (.5,0) \botlabel{j} \braidup (0,.5) \braidup (.5,1);
    \end{tikzpicture}
    ~=~
        \begin{tikzpicture}
        [baseline=1em]
        \draw (0,0) \botlabel{i} \braidup (1,1);
    \draw[red,thick] (1,0) \botlabel{\rlam} \braidup (0,1);
    \draw (.5,0) \botlabel{j} \braidup (1,.5) \braidup (.5,1);
    \end{tikzpicture}, \quad
    \begin{tikzpicture}
        [baseline=1em]
        \draw[red,thick] (0,0) \botlabel{\rlam} \braidup (1,1);
    \draw (1,0) \botlabel{i} \braidup (0,1);
    \draw (.5,0) \botlabel{j} \braidup (0,.5) \braidup (.5,1);
    \end{tikzpicture}
    ~=~
        \begin{tikzpicture}
        [baseline=1em]
        \draw[red,thick] (0,0) \botlabel{\rlam} \braidup (1,1);
    \draw (1,0) \botlabel{i} \braidup (0,1);
    \draw (.5,0) \botlabel{j} \braidup (1,.5) \braidup (.5,1);
    \end{tikzpicture}, \\
    \label{dotmovered}
    \begin{tikzpicture}
        [baseline=.6]
         \draw (0,0)[red,thick] \botlabel{\rlam} \braidup (.4,.4);
    \draw (.4,0) \botlabel{i} \braidup (0,.4);
    \bdot{.04,.3};
    \end{tikzpicture}
    =
     \begin{tikzpicture}
        [baseline=.6]
         \draw (0,0)[red,thick] \botlabel{\rlam} \braidup (.4,.4);
    \draw (.4,0) \botlabel{i} \braidup (0,.4);
    \bdot{.36,.1};
    \end{tikzpicture},\quad 
        \begin{tikzpicture}
        [baseline=.6]
         \draw (0,0) \botlabel{i} \braidup (.4,.4);
    \draw (.4,0)[red,thick] \botlabel{\rlam} \braidup (0,.4);
    \bdot{.04,.1};
    \end{tikzpicture}
    =
     \begin{tikzpicture}
        [baseline=.6]
         \draw (0,0)\botlabel{i} \braidup (.4,.4);
    \draw (.4,0)[red,thick]  \botlabel{\rlam} \braidup (0,.4);
    \bdot{.36,.3};
    \end{tikzpicture}, \\
    \label{reddoublecross}
    \begin{tikzpicture}
        [baseline=.8em]
         \draw (0,0)[red,thick] \botlabel{\rlam} \braidup (.4,.4) \braidup (0,.8);
    \draw (.4,0) \botlabel{i} \braidup (0,.4) \braidup (.4,.8);
    \end{tikzpicture}
    = \begin{tikzpicture}
        [baseline=.8em]
         \draw (0,0)[red,thick] \botlabel{\rlam} \braidup (0,.4) \braidup (0,.8);
    \draw (.4,0) \botlabel{i} \braidup (.4,.4) \braidup (.4,.8);
    \bdot{.4,.4}; \node at (.6,.4){$\scriptstyle \rlam_i$};
    \end{tikzpicture}, \quad
    \begin{tikzpicture}
        [baseline=.8em]
         \draw (0,0) \botlabel{i} \braidup (.4,.4) \braidup (0,.8);
    \draw (.4,0)[red,thick] \botlabel{\rlam} \braidup (0,.4) \braidup (.4,.8);
    \end{tikzpicture}
    =
     \begin{tikzpicture}
        [baseline=.8em]
         \draw (0,0) \botlabel{i} \braidup (0,.4) \braidup (0,.8);
         \bdot{0,.4}; \node at (-.2,.4){$\scriptstyle \rlam_i$};
    \draw (.4,0)[red,thick] \botlabel{\rlam} \braidup (.4,.4) \braidup (.4,.8);
    \end{tikzpicture},
\end{gather}
where $\rlam=\sum_{i\in Q_0}\rlam_i \varpi_i$ and $i,j\in Q_0$.
\end{definition}

Let $\obj{\klrw}$ denote the set of objects of $\klrw$.  
In particular, each object in $\obj{\klrw}$ is uniquely determined by a triple
\(
(\bi,\redkappa,\lamseq),
\)
where $\bi=(i_1,\dots,i_n)$ records the labels of the black strands from left to right, $\redkappa$ is a weakly increasing function
\(
\redkappa \colon [1,\ell] \to [0,n]
\)
such that the $k$-th red strand lies between the $\redkappa(k)$-th and $(\redkappa(k)+1)$-st black strands, and
\[
\lamseq = (\rlam^{(1)}, \rlam^{(2)}, \ldots, \rlam^{(\ell)})
\]
labels the red strands from left to right.  
We interpret $\redkappa(k)=0$ to mean that the $k$-th red strand lies to the left of all black strands, and $\redkappa(k)=n$ to mean that it lies to the right of all black strands.  
When $\lamseq$ is clear from the context, we omit it and write simply $(\bi,\redkappa)$ for the corresponding object. For any two sequences of dominant weights $\lamseq$ and $\lamseq\red{'}$, we write $\lamseq \cup \lamseq\red{'}$ for their obvious horizontal concatenation. Moreover, for any $\lamseq$ we define 
\begin{gather}
\label{objklrwalg}
\obj{\klrw}^{\lamseq}=\{ (\bi,\redkappa,\red{\underline{\lambda'}})\in \obj{\klrw}\mid \red{\underline{\lambda'}}=\lamseq \}.
\end{gather}
Then we define the KLRW algebra associated with $\lamseq$ to be 
\begin{equation}
    \label{klrwalg}
    \klrw^{\lamseq}:=\oplus_{(\bi,\redkappa_1),(\bj,\redkappa_2)\in \obj{\klrw}^{\lamseq}}\Hom_{\klrw}((\bi,\redkappa_1),(\bj,\redkappa_2)).
\end{equation}
We note that in \cref{klrwalg} $\Hom_{\klrw}((\bi,\redkappa_1),(\bj,\redkappa_2))=\{0\}$ unless $\alpha_\bi=\alpha_\bj$.

\begin{remark}
    As Webster explained in \cite[Prop. 4.21]{Web17a}, one can split $\lamseq$ into a sequence consisting of fundamental weights only, denoted by $\lamseq^{\red{\operatorname{spl}}}$, then the algebra $\klrw^{\lamseq}$ is isomorphic to an idempotent truncation of $\klrw^{\lamseq^{\red{\operatorname{spl}}}}$.
\end{remark}

As in \cite[Def. 4,7]{Web17a}, we let $\klrw_{\alpha}^\lamseq$ be the subalgebra of $\klrw^{\lamseq}$ where the sum of the roots associated to the black strands is $\sum_{i=1}^\ell\rlam_i-\alpha$ and let $\klrw_{n}^\lamseq$ denote the subalgebra of diagrams with $n$ black strands.

We then describe the polynomial representation of $\klrw^{\lamseq}$ inherited from the one of $\klrw$. Define
\begin{equation}
\label{eq:Pol}
 \Pol(\bi,\redkappa,\lamseq):=\kk[Y_1(\bi,\redkappa,\lamseq),\dots Y_n(\bi,\redkappa,\lamseq)] \text{ where }\bi=\{i_1,\ldots,i_n\}.
\end{equation}
As implied by \cite[Lem. 4.12]{Web17a}, there is a strict monoidal functor $\cP: \klrw \to \cVect_\kk$ such that
\begin{equation}
\label{Polobj}
\cP(\bi,\kappa,\lamseq)= \Pol(\bi,\kappa,\lamseq)
\end{equation}
for all $(\bi,\kappa,\lamseq)\in \obj{\klrw}$, and 
\begin{gather}
\label{Polklrw1}
    \cP\Big(\begin{tikzpicture}
        [baseline=.6]
         \draw (0,0)[red,thick] \botlabel{\rlam} \braidup (.4,.4);
    \draw (.4,0) \botlabel{i} \braidup (0,.4);
    \end{tikzpicture}\Big)
    ~: f\mapsto Y_1^{\rlam_i}\cdot f~\quad
    \cP\Big(\begin{tikzpicture}
        [baseline=.6]
         \draw (0,0) \botlabel{i} \braidup (.4,.4);
    \draw (.4,0)[red,thick] \botlabel{\rlam} \braidup (0,.4);
    \end{tikzpicture}\Big)
    ~: f\mapsto f~\quad
    \cP\Big(\begin{tikzpicture}
        [baseline=.6]
        \draw (0,0)\botlabel{i} -- (0,.4);
        \bdot{0,.2};
    \end{tikzpicture}\Big)
    ~: f\mapsto Y_1\cdot f \\
    \label{Polklrw2}
    \cP\Big(\begin{tikzpicture}
        [baseline=.3]
         \draw (0,0) \botlabel{i} \braidup (.6,.6);
    \draw (.6,0) \botlabel{j} \braidup (0,.6);
    \end{tikzpicture}\Big)
    ~: f\mapsto \begin{cases}
        X_{ij}(Y_{2},Y_1)f^{s_1}, & i\neq j, \\[1ex]
        \dfrac{f^{s_1}-f}{Y_{2}-Y_{1}}, & i=j.
    \end{cases}
\end{gather}
where 
\begin{align}
\label{weylact}
f^{s_j}(Y_1,\ldots,Y_n) =
\begin{cases} 
 f(\cdots,Y_{j+1},Y_{j},\cdots),&\text{ if } j>0; \\
 f(-Y_1,Y_2,\cdots), &\text{ if } j=0.
\end{cases} 
\end{align}
In particular, restricting $\cP$ to $\klrw^{\lamseq}$ defines a polynomial representation of $\klrw^{\lamseq}$ on \[\Pol^\lamseq:=\oplus_{(\bi,\redkappa)\in \obj{\klrw}^{\lamseq}}\Pol(\bi,\redkappa).\]

\subsection{oKLRW algebras}
\label{sec:oKLRW}
We are now ready to introduce the oKLRW category and the oKLRW algebra. Recall the iquiver Hecke category $\cRB$ from \cref{defcRB} as a module category over $\cR$. 
\begin{definition}
    \label{defoKLRWC}
    The oKLRW category associated with a dominant weight $\mu$, denoted as $\oklrwmu$, is the $\klrw$-module generated by the morphisms
     \begin{gather*}
        \begin{tikzpicture}
            [anchorbase]
            \mirror{-.5}{0}{1};
            \draw (0,0) \botlabel{i} to[out=up,in=right] (-.5,.5) to[out=right,in=down] (0,1)\toplabel{\tau i};
        \end{tikzpicture} \colon i \to \tau i,\quad \text{ for all } i\in Q_0,
    \end{gather*}
    subject to the relations \cref{quadratic,reflection} (which depend on $\mu$).
\end{definition}
By \cref{defoKLRWC}, we observe that the red strands will never touch the mirror. Moreover, the category $\oklrwmu$ shares the same set of objects as $\klrw$. Let $\lamseq$ be a fixed sequence of dominant weights, recall $\obj{\klrw}^{\lamseq}$ from \cref{objklrwalg}. We define the oKLRW algebra to be
\begin{gather}
\label{eq:oklrwalg}
    \oklrwmu^{\lamseq}:=\bigoplus_{(\bi,\redkappa_1),(\bj,\redkappa_2)\in \obj{\klrw}^{\lamseq}} \Hom_{\oklrwmu}((\bi,\redkappa_1),(\bj,\redkappa_2)).    
\end{gather}
As in \cref{embed}, we can regard the KLRW algebra $\klrwalg$ as a natural subalgebra of the oKLRW algebra $\oklrwalg$. Recall the functor $\cP$ from \cref{Polobj,Polklrw1,Polklrw2}.
\begin{prop}
\label{prop:polyrep}
     Regarding $\cVect_\kk$ as a strict $\klrw$-module via $\cP$,  there is a unique strict morphism of $\klrw$-modules $\ctP:\oklrwmu\to \cVect_\kk$ such that
     \begin{gather}
         \label{Polreflect}
         \ctP \Big(~
            \begin{tikzpicture}
            [anchorbase]
            \mirror{-.5}{0}{1};
            \draw (0,0) \botlabel{i} to[out=up,in=right] (-.5,.5) to[out=right,in=down] (0,1)\toplabel{\tau i};
        \end{tikzpicture} 
        ~ \Big) : f \mapsto (-Y_1)^{\mu_i}f^{s_0}.
     \end{gather}
\end{prop}
\begin{proof}
    By our definition, $\ctP$ agrees with $\cP$ when restricting to $\klrw$. Thus it suffices to show that \cref{Polreflect} satisfies \cref{quadratic,reflection}, which is straightforward to check; see also \cite[Prop 2.7]{Prz23}.
\end{proof}
To summarize, we now have the following diagram:
\begin{equation}
\label{diagram}
\begin{tikzcd}
&\text{Quiver Hecke category } \cR \arrow[rd] \arrow[ld]& \\
\text{iQuiver Hecke category } \cRB \arrow[rd,dashed] & & \text{KLRW category } \klrw  \arrow[ld] \\
& \text{oKLRW category } \oklrwmu&
\end{tikzcd}
\end{equation}
To interpret the dashed arrow in \cref{diagram}, we set $\mu=\bzero$ and consider the full subcategory
\[
{}_{\bzero}\cS_\lambda \subset \oklrwzero
\]
whose objects are those containing exactly one red strand labelled by
\(
\lamseq = (\lambda).
\)
 By definition, ${}_\bzero\cS_\lambda$ is naturally a $\cR$-module.
\begin{theorem}
\label{framezero}
    For any framing weight $\lambda$, 
    there exists an embedding of $\cR$-modules  $\cF\colon\cRB\hookrightarrow {}_\bzero\cS_\lambda$ which sends any object $\bi\in \cRB$ to $(\bi,\redkappa,\rlam)$ with $\redkappa(1)=0$ and
    \begin{gather}
        \begin{tikzpicture}
            [anchorbase]
            \mirror{-.5}{0}{1};
            \draw (0,0) \botlabel{i} to[out=up,in=right] (-.5,.5) to[out=right,in=down] (0,1)\toplabel{\tau i};
        \end{tikzpicture} ~\mapsto~~\begin{tikzpicture}
            [anchorbase]
            \mirror{-.5}{0}{1};
            \draw (0,0) \botlabel{i} to[out=up,in=right] (-.5,.5) to[out=right,in=down] (0,1)\toplabel{\tau i};
            \draw[red] (-.25,0)\botlabel{\mu} \braidup (.1,.5) \braidup (-.25,1);
        \end{tikzpicture} 
    \end{gather}
\end{theorem}
\begin{proof}
    We observe that $\ctP\circ\cF$ defines a faithful polynomial representation of $\cRB$ which agrees with \cite[Prop. 2.7]{Prz23} (up to a sign), hence $\cF$ is faithful as well. Then it suffices to check that \cref{quadratic,reflection} are preserved under $\cF$. For the first relation in \cref{quadratic}, in ${}_\bzero\cS_\lambda$ we have
    \begin{gather*}
        \begin{tikzpicture}
            [anchorbase]
            \mirror{-.5}{0}{1};
             \draw (0,0) \botlabel{i} to[out=up,in=right] (-.5,.25) to[out=right,in=down] (0,.5) to [out=up,in=right] (-.5,.75) to[out=right,in=down] (0,1)\toplabel{i};
             \draw[red] (-.25,0) \botlabel{\mu} \braidup (0,.25) \braidup (-.25,.5) \braidup (0,.75)\braidup (-.25,1);
        \end{tikzpicture}
        \overset{\cref{reddoublecross}}{=}
         \begin{tikzpicture}
            [anchorbase]
            \mirror{-.5}{0}{1};
             \draw (0,0) \botlabel{i} to[out=up,in=right] (-.5,.25) to[out=right,in=down] (-.25,.5) to [out=up,in=right] (-.5,.75) to[out=right,in=down] (0,1)\toplabel{i};
             \draw[red] (-.25,0) \botlabel{\mu} \braidup (.2,.5) \braidup (-.25,1);
              \bdot{-.25,.5}; \node at (0,.5){$\scriptstyle \rmu_{\tau i}$};
        \end{tikzpicture}
        \underset{\cref{dotmovered}}{\overset{\cref{quadratic}}{=}}(-1)^{\rmu_{\tau i}}
        \begin{tikzpicture}
            [anchorbase]
            \mirror{-.5}{0}{1};
             \draw (0,0) \botlabel{i} to[out=up,in=right] (-.5,.25) to[out=right,in=down] (-.25,.5) to [out=up,in=right] (-.5,.75) to[out=right,in=down] (0,1)\toplabel{i};
             \draw[red] (-.25,0) \botlabel{\mu} \braidup (0,.5) \braidup (-.25,1);
              \bdot{-.02,.9}; \node at (.2,.8){$\scriptstyle \rmu_{\tau i}$};
        \end{tikzpicture}
        \overset{\cref{quadratic}}{=}(-1)^{\rmu_{\tau i}}
        \begin{tikzpicture}
            [anchorbase]
            \mirror{-.5}{0}{1};
             \draw (0,0) \botlabel{i} \braidup (-.25,.5)\braidup (0,1)\toplabel{i};
             \draw[red] (-.25,0) \botlabel{\mu} \braidup (0,.5) \braidup (-.25,1);
              \bdot{-.02,.9}; \node at (.2,.8){$\scriptstyle \rmu_{\tau i}$};
        \end{tikzpicture}
        \overset{\cref{reddoublecross}}{=}(-1)^{\rmu_{\tau i}}
         \begin{tikzpicture}
            [anchorbase]
            \mirror{-.5}{0}{1};
             \draw (0,0) \botlabel{i} \braidup  (0,1)\toplabel{i};
             \draw[red] (-.25,0) \botlabel{\mu} \braidup (-.25,1);
              \bdot{0,.7}; \node at (.3,.7){$\scriptstyle \rmu_{\tau i}$};
               \bdot{0,.3}; \node at (.3,.3){$\scriptstyle \rmu_{i}$};
        \end{tikzpicture}
        =   \begin{tikzpicture}
            [anchorbase]
            \mirror{-.25}{0}{1};
            \draw[red] (-.1,0)\botlabel{\mu} -- (-.1,1);
             \draw (.6,0) \botlabel{i} -- (.6,1)\toplabel{i};
             \coupon{.6,.5}{f^{\rmu}_{i}(-y_1)};
        \end{tikzpicture}.
    \end{gather*}
Thus, $\cF$ preserves the first relation in \cref{quadratic}.
For the second relation in \cref{quadratic}, \cref{dotmovered} implies that it is preserved by $\cF$ as well.

For \cref{reflection}, first we suppose that $i\neq \tau j$, and hence $\tau i\neq j$. Then we have 
\begin{gather*}
        \begin{tikzpicture}
        [anchorbase]
        \mirror{-.25}{0}{1};
       \draw (0,0) \botlabel{i} to[out=up,in=-15] (-.25,.2) to [out=15,in=down] (.3,.6) \braidup (0,1) \toplabel{\tau i};
        \draw (.3,0) \botlabel{j} to[out=up,in=right] (-.25,.5)  to [out=right,in=down] (.3,1) \toplabel{\tau j};
        \draw[red] (-.125,0)\botlabel{\mu} \braidup (0,.3) \braidup (-.125,1);
    \end{tikzpicture}
    \overset{\cref{redbraid1}}{=}
     \begin{tikzpicture}
        [anchorbase]
        \mirror{-.25}{0}{1};
       \draw (0,0) \botlabel{i} to[out=up,in=-15] (-.25,.2) to [out=15,in=down] (.3,.6) \braidup (0,1) \toplabel{\tau i};
        \draw (.3,0) \botlabel{j} to[out=up,in=right] (-.25,.5)  to [out=right,in=down] (.3,1) \toplabel{\tau j};
        \draw[red] (-.125,0)\botlabel{\mu} \braidup (.5,.3) \braidup (-.125,1);
    \end{tikzpicture}
    \overset{\cref{redbraid2}}{=}
     \begin{tikzpicture}
        [anchorbase]
        \mirror{-.25}{0}{1};
       \draw (0,0) \botlabel{i} to[out=up,in=-15] (-.25,.2) to [out=15,in=down] (.3,.6) \braidup (0,.8) \braidup (.2,1) \toplabel{\tau i};
        \draw (.3,0) \botlabel{j} to[out=up,in=right] (-.25,.5)  to [out=right,in=down] (.5,1) \toplabel{\tau j};
        \draw[red] (-.125,0)\botlabel{\mu} \braidup (.5,.5) \braidup (0,1);
    \end{tikzpicture}
    \overset{\cref{reflection}}{=}
     \begin{tikzpicture}
        [anchorbase]
        \mirror{-.25}{0}{1};
      \draw (0,0) \botlabel{i} \braidup (-.125,.2) \braidup (.3,.4) to [out=up,in=-15] (-.25,.8) to [out=15,in=down] (.2,1) \toplabel{\tau i};
        \draw (.3,0) \botlabel{j} to[out=up,in=right] (-.25,.5)  to [out=right,in=down] (.5,1) \toplabel{\tau j};
        \draw[red] (-.125,0)\botlabel{\mu} \braidup (.5,.5) \braidup (0,1);
    \end{tikzpicture}
    \underset{\cref{redbraid2}}{\overset{\cref{redbraid1}}{=}}
    \begin{tikzpicture}
        [anchorbase]
        \mirror{-.25}{0}{1};
        \draw (0,0) \botlabel{i} \braidup (.3,.4) to [out=up,in=-15] (-.25,.8) to [out=15,in=down] (0,1) \toplabel{\tau i};
        \draw (.3,0) \botlabel{j} to[out=up,in=right] (-.25,.5)  to [out=right,in=down] (.5,1) \toplabel{\tau j};
         \draw[red] (-.125,0)\botlabel{\mu} \braidup (0,.3) \braidup (-.125,1);
    \end{tikzpicture}.
\end{gather*}
Now suppose that $i=\tau j$, hence $i\neq j$. Then in ${}_\bzero\cS_\lambda\subset \oklrwzero$ we have
\begin{gather}
\label{steak1}
      \begin{tikzpicture}
        [anchorbase]
        \mirror{-.25}{0}{1};
       \draw (0,0) \botlabel{i} to[out=up,in=-15] (-.25,.2) to [out=15,in=down] (.3,.6) \braidup (0,1) \toplabel{j};
        \draw (.3,0) \botlabel{j} to[out=up,in=right] (-.25,.5)  to [out=right,in=down] (.3,1) \toplabel{i};
        \draw[red] (-.125,0)\botlabel{\mu} \braidup (0,.3) \braidup (-.125,1);
    \end{tikzpicture}
    \underset{\cref{redbraid2}}{\overset{\cref{redbraid1}}{=}}
    \begin{tikzpicture}
        [anchorbase]
        \mirror{-.25}{0}{1};
       \draw (0,0) \botlabel{i} to[out=up,in=-15] (-.25,.2) to [out=15,in=down] (.3,.6) \braidup (0,.8) \braidup (.2,1) \toplabel{j};
        \draw (.3,0) \botlabel{j} to[out=up,in=right] (-.25,.5)  to [out=right,in=down] (.5,1) \toplabel{i};
        \draw[red] (-.125,0)\botlabel{\mu} \braidup (.5,.5) \braidup (0,1);
    \end{tikzpicture}
    ~+\sum_{a+b+1=\rmu_j}~
    \begin{tikzpicture}
        [anchorbase]
        \mirror{-.5}{0}{1};
        \draw (-.1,0) \botlabel{i} to[out=up,in=-15] (-.5,.3) to [out=15,in=down] (-.3,.5) to[out=up,in=-15] (-.5,.7) to [out=15,in=down] (.1,1) \toplabel{i};
        \draw (.1,0) \botlabel{j} \braidup (-.1,1) \toplabel{j};
        \draw[red] (-.3,0)\botlabel{\mu} \braidup (-.1,.5) \braidup (-.3,1);
        \bdot{-.3,.5}; \node at (-.18,.5){$\scriptstyle a$};
        \bdot{.08,.2}; \node at (.2,.2){$\scriptstyle b$};
    \end{tikzpicture}
    \underset{\cref{dotmovered}}{\overset{\cref{quadratic},\cref{reddoublecross}}{=}}
     \begin{tikzpicture}
        [anchorbase]
        \mirror{-.25}{0}{1};
       \draw (0,0) \botlabel{i} to[out=up,in=-15] (-.25,.2) to [out=15,in=down] (.3,.6) \braidup (0,.8) \braidup (.2,1) \toplabel{j};
        \draw (.3,0) \botlabel{j} to[out=up,in=right] (-.25,.5)  to [out=right,in=down] (.5,1) \toplabel{i};
        \draw[red] (-.125,0)\botlabel{\mu} \braidup (.5,.5) \braidup (0,1);
    \end{tikzpicture}
    ~+\sum_{a+b+1=\rmu_j}(-1)^a~
    \begin{tikzpicture}
        [anchorbase]
        \mirror{-1}{0}{1};
        \draw (.1,0) \botlabel{i} \braidup (.5,1) \toplabel{i};
        \draw (.5,0) \botlabel{j} \braidup (.1,1) \toplabel{j};
        \draw[red] (-.8,0)\botlabel{\mu} \braidup (-.8,1);
        \bdot{.48,.2}; \node at (.6,.2){$\scriptstyle b$};
        \bdot{.12,.2}; \node at (-.3,.38){$\scriptstyle \rmu_i+a$};
    \end{tikzpicture}
\\
\label{steak2}
    \begin{tikzpicture}
        [anchorbase]
        \mirror{-.25}{0}{1};
        \draw (0,0) \botlabel{i} \braidup (.3,.4) to [out=up,in=-15] (-.25,.8) to [out=15,in=down] (0,1) \toplabel{j};
        \draw (.3,0) \botlabel{j} to[out=up,in=right] (-.25,.5)  to [out=right,in=down] (.5,1) \toplabel{i};
         \draw[red] (-.125,0)\botlabel{\mu} \braidup (0,.3) \braidup (-.125,1);
    \end{tikzpicture}
    \underset{\cref{redbraid2}}{\overset{\cref{redbraid1}}{=}}
     \begin{tikzpicture}
        [anchorbase]
        \mirror{-.25}{0}{1};
      \draw (0,0) \botlabel{i} \braidup (-.125,.2) \braidup (.3,.4) to [out=up,in=-15] (-.25,.8) to [out=15,in=down] (.2,1) \toplabel{j};
        \draw (.3,0) \botlabel{j} to[out=up,in=right] (-.25,.5)  to [out=right,in=down] (.5,1) \toplabel{i};
        \draw[red] (-.125,0)\botlabel{\mu} \braidup (.5,.5) \braidup (0,1);
    \end{tikzpicture}
    ~+\sum_{c+d+1=\rmu_i}~
     \begin{tikzpicture}
        [anchorbase]
        \mirror{-.5}{0}{1};
        \draw (.1,0) \botlabel{j} to[out=up,in=-15] (-.5,.3) to [out=15,in=down] (-.3,.5) to[out=up,in=-15] (-.5,.7) to [out=15,in=down] (-.1,1) \toplabel{j};
        \draw (-.1,0) \botlabel{i} \braidup (.1,1) \toplabel{i};
        \draw[red] (-.3,0)\botlabel{\mu} \braidup (-.1,.5) \braidup (-.3,1);
        \bdot{-.3,.5}; \node at (-.18,.5){$\scriptstyle c$};
        \bdot{0,.5}; \node at (.2,.5){$\scriptstyle d$};
    \end{tikzpicture}
    \underset{\cref{dotmovered},\cref{reddoublecross}}{\overset{\cref{dotcross},\cref{quadratic}}{=}}
    \begin{tikzpicture}
        [anchorbase]
        \mirror{-.25}{0}{1};
      \draw (0,0) \botlabel{i} \braidup (-.125,.2) \braidup (.3,.4) to [out=up,in=-15] (-.25,.8) to [out=15,in=down] (.2,1) \toplabel{j};
        \draw (.3,0) \botlabel{j} to[out=up,in=right] (-.25,.5)  to [out=right,in=down] (.5,1) \toplabel{i};
        \draw[red] (-.125,0)\botlabel{\mu} \braidup (.5,.5) \braidup (0,1);
    \end{tikzpicture}
    ~+\sum_{c+d+1=\rmu_i}(-1)^c~
    \begin{tikzpicture}
        [anchorbase]
        \mirror{-.3}{0}{1};
        \draw (.1,0) \botlabel{i} \braidup (.5,1) \toplabel{i};
        \draw (.5,0) \botlabel{j} \braidup (.1,1) \toplabel{j};
        \draw[red] (-.2,0)\botlabel{\mu} \braidup (-.2,1);
        \bdot{.48,.2}; \node at (.9,.2){$\scriptstyle \rmu_j+c$};
        \bdot{.12,.2}; \node at (-.05,.3){$\scriptstyle d$};
    \end{tikzpicture}
\end{gather}
Notice that
\begin{gather} 
    \label{steak3}
    \sum_{a+b+1=\rmu_j}(-1)^a~
    \begin{tikzpicture}
        [anchorbase]
        \mirror{-1}{0}{1};
        \draw (.1,0) \botlabel{i} \braidup (.5,1) \toplabel{i};
        \draw (.5,0) \botlabel{j} \braidup (.1,1) \toplabel{j};
        \draw[red] (-.8,0)\botlabel{\mu} \braidup (-.8,1);
        \bdot{.48,.2}; \node at (.6,.2){$\scriptstyle b$};
        \bdot{.12,.2}; \node at (-.3,.38){$\scriptstyle \rmu_i+a$};
    \end{tikzpicture}
    = 
    \begin{tikzpicture}
            [anchorbase]
            \mirror{-.8}{0}{1.2};
            \draw[red] (-.7,0)\botlabel{\mu} -- (-.7,1.2);
            \draw (.4,0)\botlabel{i} \braidup (.4,.8) \braidup (.8,1.2) \toplabel{i};
            \draw (.8,0)\botlabel{j} \braidup (.8,.8) \braidup (.4,1.2) \toplabel{j};
            \coupon{.6,.5}{\frac{y_2^{\rmu_j}-(-y_1)^{\rmu_j}}{y_1+y_2}y_1^{\rmu_i}};
        \end{tikzpicture},\quad
    \sum_{c+d+1=\rmu_i}(-1)^c~
    \begin{tikzpicture}
        [anchorbase]
        \mirror{-.3}{0}{1};
        \draw (.1,0) \botlabel{i} \braidup (.5,1) \toplabel{i};
        \draw (.5,0) \botlabel{j} \braidup (.1,1) \toplabel{j};
        \draw[red] (-.2,0)\botlabel{\mu} \braidup (-.2,1);
        \bdot{.48,.2}; \node at (.9,.2){$\scriptstyle \rmu_j+c$};
        \bdot{.12,.2}; \node at (-.05,.3){$\scriptstyle d$};
    \end{tikzpicture}
    =
     \begin{tikzpicture}
            [anchorbase]
            \mirror{-.8}{0}{1.2};
            \draw[red] (-.7,0)\botlabel{\mu} -- (-.7,1.2);
            \draw (.4,0)\botlabel{i} \braidup (.4,.8) \braidup (.8,1.2) \toplabel{i};
            \draw (.8,0)\botlabel{j} \braidup (.8,.8) \braidup (.4,1.2) \toplabel{j};
            \coupon{.6,.5}{\frac{y_1^{\rmu_i}-(-y_2)^{\rmu_i}}{y_1+y_2}y_2^{\rmu_j}};
        \end{tikzpicture}
\end{gather}
Since  $ \begin{tikzpicture}
        [anchorbase]
        \mirror{-.25}{0}{1};
       \draw (0,0) \botlabel{i} to[out=up,in=-15] (-.25,.2) to [out=15,in=down] (.3,.6) \braidup (0,.8) \braidup (.2,1) \toplabel{j};
        \draw (.3,0) \botlabel{j} to[out=up,in=right] (-.25,.5)  to [out=right,in=down] (.5,1) \toplabel{i};
        \draw[red] (-.125,0)\botlabel{\mu} \braidup (.5,.5) \braidup (0,1);
    \end{tikzpicture}
    =
    \begin{tikzpicture}
        [anchorbase]
        \mirror{-.25}{0}{1};
      \draw (0,0) \botlabel{i} \braidup (-.125,.2) \braidup (.3,.4) to [out=up,in=-15] (-.25,.8) to [out=15,in=down] (.2,1) \toplabel{j};
        \draw (.3,0) \botlabel{j} to[out=up,in=right] (-.25,.5)  to [out=right,in=down] (.5,1) \toplabel{i};
        \draw[red] (-.125,0)\botlabel{\mu} \braidup (.5,.5) \braidup (0,1);
    \end{tikzpicture}$ in ${}_\bzero\cS_\lambda$,
by taking difference of \cref{steak1,steak2}, we obtain that
\begin{gather*}
     \begin{tikzpicture}
        [anchorbase]
        \mirror{-.25}{0}{1};
       \draw (0,0) \botlabel{i} to[out=up,in=-15] (-.25,.2) to [out=15,in=down] (.3,.6) \braidup (0,1) \toplabel{j};
        \draw (.3,0) \botlabel{j} to[out=up,in=right] (-.25,.5)  to [out=right,in=down] (.3,1) \toplabel{i};
        \draw[red] (-.125,0)\botlabel{\mu} \braidup (0,.3) \braidup (-.125,1);
    \end{tikzpicture}
    -
     \begin{tikzpicture}
        [anchorbase]
        \mirror{-.25}{0}{1};
        \draw (0,0) \botlabel{i} \braidup (.3,.4) to [out=up,in=-15] (-.25,.8) to [out=15,in=down] (0,1) \toplabel{j};
        \draw (.3,0) \botlabel{j} to[out=up,in=right] (-.25,.5)  to [out=right,in=down] (.5,1) \toplabel{i};
         \draw[red] (-.125,0)\botlabel{\mu} \braidup (0,.3) \braidup (-.125,1);
    \end{tikzpicture}
    \overset{\cref{steak3}}{=}
    \begin{tikzpicture}
            [anchorbase]
            \mirror{-.8}{0}{1.2};
            \draw[red] (-.65,0)\botlabel{\mu} -- (-.65,1.2);
            \draw (.4,0)\botlabel{i} \braidup (.4,.8) \braidup (.8,1.2) \toplabel{i};
            \draw (.8,0)\botlabel{j} \braidup (.8,.8) \braidup (.4,1.2) \toplabel{j};
            \coupon{.6,.5}{\frac{f^{\rmu}_i(y_2)-f^{\rmu}_j(y_1)}{y_1+y_2}};
        \end{tikzpicture}
        \overset{\cref{dotcross}}{=}
        \begin{tikzpicture}
            [anchorbase]
            \mirror{-.8}{0}{1.2};
            \draw[red] (-.65,0)\botlabel{\mu} -- (-.65,1.2);
            \draw (.4,0)\botlabel{i} \braidup (.8,.4) -- (.8,1.2) \toplabel{i};
            \draw (.8,0)\botlabel{j} \braidup (.4,.4) -- (.4,1.2) \toplabel{j};
            \coupon{.6,.7}{\frac{f^{\rmu}_i(y_1)-f^{\rmu}_j(y_2)}{y_1+y_2}};
        \end{tikzpicture}
\end{gather*}
This finishes the proof.
\end{proof}
\cref{framezero} shows that, for any framing weight $\lambda$, we can always regard $\cRB$ as a subcategory of $\oklrwzero$ by adding a red strand labeled $\lambda$ to the left end of each generating morphism. The exact same proof shows the following:
\begin{corollary}
\label{embedzeroklrw}
    For any framing weight $\mu$, there exists an embedding of $\klrw$-modules
\(
\tilde{\cF}\colon \oklrwmu \hookrightarrow \oklrwzero,
\)
which sends an object $(\bi,\redkappa,\lamseq)$ to
\(
(\bi,\redkappa',\, \mu \cup \lamseq),
\)
where \(\redkappa'(1)=0, 
\redkappa'(k+1)=\redkappa(k)\ \text{ for all } k\ge 1,
\)
and such that
    \begin{gather*}
        \begin{tikzpicture}
            [anchorbase]
            \mirror{-.5}{0}{1};
            \draw (0,0) \botlabel{i} to[out=up,in=right] (-.5,.5) to[out=right,in=down] (0,1)\toplabel{\tau i};
        \end{tikzpicture} ~\mapsto~~\begin{tikzpicture}
            [anchorbase]
            \mirror{-.5}{0}{1};
            \draw (0,0) \botlabel{i} to[out=up,in=right] (-.5,.5) to[out=right,in=down] (0,1)\toplabel{\tau i};
            \draw[red] (-.25,0)\botlabel{\mu} \braidup (.1,.5) \braidup (-.25,1);
        \end{tikzpicture} 
    \end{gather*}
\end{corollary}
Therefore, by \cref{framezero,embedzeroklrw}, to study the oKLRW category associated with an arbitrary framing weight, it suffices to understand the case where the framing weight is $\bzero$.

\subsection{Diagonal type}
\label{sec:diagnoal}
Recall the notion of a quiver with an involution from \cref{sec:oKLR}. In this subsection, we assume that our quiver $\Qui=(Q_0,Q_1,s,t)$ is the disjoint union of two
 copies of quivers with identical vertices and opposite arrows, denoted by
\[
\Qui=\Qui_{+}\bigsqcup \Qui_{-}
\]
and let $\tau :\Qui\to\Qui$  be the involution that interchanges
the vertices in corresponding positions. For example, we may have the following quiver with involution
$$
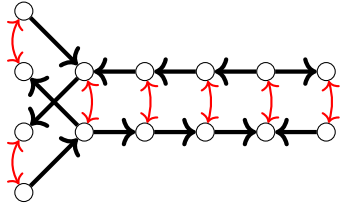
\begin{tikzpicture}[
  scale=.8,
  every node/.style={circle, draw, inner sep=3pt},
  arrow/.style={red, <->, thick}
]

\node (t1) at (-1,1) {};   
\node (t2) at (0,0) {};   
\node (t3) at (-1,-1) {};  
\node (t4) at (1,0) {};
\node (t5) at (2,0) {};
\node (t6) at (3,0) {};
\node (t7) at (4,0) {};

\draw[->,ultra thick] (t1) -- (t2); 
\draw[->,ultra thick] (t2) -- (t3);
\draw[<-,ultra thick] (t2) -- (t4);
\draw[<-,ultra thick] (t4) -- (t5);
\draw[<-,ultra thick] (t5) -- (t6);
\draw[->,ultra thick] (t6) -- (t7);

\node (b1) at (-1,0) {};   
\node (b2) at (0,-1) {};   
\node (b3) at (-1,-2) {};   
\node (b4) at (1,-1) {};
\node (b5) at (2,-1) {};
\node (b6) at (3,-1) {};
\node (b7) at (4,-1) {};

\draw[<-,ultra thick] (b1) -- (b2);
\draw[<-,ultra thick] (b2) -- (b3);
\draw[->,ultra thick] (b2) -- (b4);
\draw[->,ultra thick] (b4) -- (b5);
\draw[->,ultra thick] (b5) -- (b6);
\draw[<-,ultra thick] (b6) -- (b7);

\draw[arrow] (t1) to[bend right=30] (b1);
\draw[arrow] (t2) to[bend left=20] (b2);
\draw[arrow] (t3) to[bend right=30] (b3);
\draw[arrow] (t4) to[bend left=15] (b4);
\draw[arrow] (t5) to[bend left=15] (b5);
\draw[arrow] (t6) to[bend left=15] (b6);
\draw[arrow] (t7) to[bend left=15] (b7);

\end{tikzpicture}$$
where each copy is a type D Dynkin quiver.

Let $\lamseq$ be a fixed sequence of dominant weights associated with $\Qui_{+}$. Recall the oKLRW algebra ${\oklrwzero}^{\lamseq}$ associated with $(\Qui,\tau)$ with framing weight $\bzero$ from \cref{oKLRalgebra} and the KLRW algebra ${\klrwalg}$ associated with $\Qui^{(+)}$ from \cref{klrwalg}.
\begin{theorem}
\label{thm:diagonal}
    The oKLRW algebra ${\oklrwzero}^{\lamseq}$ associated with $(\Qui,\tau)$ is Morita equivalent to the KLRW algebra ${\klrwalg}$ associated with $\Qui_{+}$. 
\end{theorem}
\begin{proof}
    Let $Q_{0,\pm}$ denote the set of indices of $\Qui_{\pm}$. By our definition, $\klrwalg$ is a span of Stendhal diagrams (\cite[Def. 4.1]{Web17a}) and elements of $\obj{\klrw}^{\lamseq}$ are of the form $(\bi,\redkappa)$ where $\bi$ are words in $Q_{0,+}$. Let \[
    e:=\sum_{(\bi,\redkappa)\in \obj{\klrw}^{\lamseq}} \id_{(\bi,\redkappa)} \in {\oklrwzero}^{\lamseq}
    \]
 where $\id_{(\bi,\redkappa)}$ is the identity morphism of $(\bi,\redkappa)$ within ${\oklrwzero}$. To show that ${\oklrwzero}^{\lamseq}$ is Morita equivalent to ${\klrwalg}$, it suffices to show 
 \[
 e({\oklrwzero}^{\lamseq})e\cong \klrwalg \quad\text{ and }\quad ({\oklrwzero}^{\lamseq} )e({\oklrwzero}^{\lamseq})={\oklrwzero}^{\lamseq}.
 \]
 Note that 
 \[
   e({\oklrwzero}^{\lamseq})e=\oplus_{(\bi,\redkappa_1),(\bj,\redkappa_2)\in \obj{\klrw}^{\lamseq}} \Hom_{\oklrwzero}((\bi,\redkappa_1),(\bj,\redkappa_2)).    
 \]
Since $\bi,\bj$ are both words in $Q_{0,+}$, we see that any black strand in those Stendhal diagrams must touch the mirror even number of times. However, inside $\oklrwzero$, by \cref{quadratic,reflection} we have for any $i,j$
\[
 \begin{tikzpicture}
            [anchorbase]
            \mirror{-.25}{0}{1};
             \draw (0,0) \botlabel{i} to[out=up,in=right] (-.25,.25) to[out=right,in=down] (0,.5) to [out=up,in=right] (-.25,.75) to[out=right,in=down] (0,1)\toplabel{i};
        \end{tikzpicture}
        =
         \begin{tikzpicture}
            [anchorbase]
            \mirror{-.25}{0}{1};
             \draw (.1,0) \botlabel{i} -- (.1,1)\toplabel{i};
        \end{tikzpicture}
    \text{ and }
     \begin{tikzpicture}
        [anchorbase]
        \mirror{-.25}{0}{1};
       \draw (0,0) \botlabel{i} to[out=up,in=-15] (-.25,.2) to [out=15,in=down] (.3,.6) \braidup (0,1) \toplabel{\tau i};
        \draw (.3,0) \botlabel{j} to[out=up,in=right] (-.25,.5)  to [out=right,in=down] (.3,1) \toplabel{\tau j};
    \end{tikzpicture}
    =
    \begin{tikzpicture}
        [anchorbase]
        \mirror{-.25}{0}{1};
        \draw (0,0) \botlabel{i} \braidup (.3,.4) to [out=up,in=-15] (-.25,.8) to [out=15,in=down] (0,1) \toplabel{\tau i};
        \draw (.3,0) \botlabel{j} to[out=up,in=right] (-.25,.5)  to [out=right,in=down] (.3,1) \toplabel{\tau j};
    \end{tikzpicture}.
\]
Thus we may always use the second relation above, together with \cref{braidA,redbraid1,redbraid2}, to detach the black strands from the mirror.  
It follows that $e(\oklrwzero^{\lamseq})e$ is spanned by Stendhal diagrams in which the black strands never touch the mirror and whose domains and codomains lie in $\obj{\klrw}^{\lamseq}$.  
Consequently, the identification with $\klrwalg$ is obtained simply by removing the mirror; see also \cref{embed}.

For the second equality, it suffices to show that for every object $(\bi,\redkappa,\lamseq)$ in $\oklrwzero$, the identity morphism
\(
\id_{(\bi,\redkappa,\lamseq)}
\)
lies in the ideal $(\oklrwzero^{\lamseq})\,e\,(\oklrwzero^{\lamseq})$.  
Here $\bi$ is a word in the alphabet $Q_0$, and it may contain letters in $Q_{0,-}$.  
We argue by an induction on the number $k$ of occurrences of vertices from $Q_{0,-}$ in the word $\bi$.
When $k=0$, there is nothing to prove. Suppose the statement is true for $k-1$, and $i_l$ is the $k$-th occurrence of vertices from $Q_{0,-}$ in the word $\bi$. Note that we may cross any red strand with the black strand labelled by $i_l$ freely since weights in $\lamseq$ are associated with $\Qui_{+}$ only. By the inductive hypothesis, we may assume $i_1,i_2,\ldots i_{l-1}\in  Q_{0,+}$ and then we have (omitting the positions of the red strands):
\begin{gather*}
   \begin{tikzpicture}
       [anchorbase]
       \mirror{-.25}{0}{1};
       \draw (0,0)\botlabel{i_1} -- (0,1)\toplabel{};
       \draw (.25,0)\botlabel{i_2} -- (.25,1);
       \node at (.5,.5){$\scriptstyle \cdots$};
       \draw (.75,0)\botlabel{i_{l-1}} -- (.75,1);
       \draw (1.25,0)\botlabel{i_l} -- (1.25,1);
      \node at (1.5,.5){$\scriptstyle \cdots$};
   \end{tikzpicture}
   \overset{\cref{doublecross}}{=}
   \begin{tikzpicture}
       [anchorbase]
       \mirror{-.25}{0}{1};
       \draw (0,0)\botlabel{i_1} -- (0,1)\toplabel{};
       \draw (.25,0)\botlabel{i_2} -- (.25,1);
       \node at (.5,.5){$\scriptstyle \cdots$};
       \draw (.75,0)\botlabel{i_{l-1}} -- (.75,1);
       \draw (1.25,0)\botlabel{i_l} to[out=150,in=down] (-.15,.5) to[out=up,in=-150] (1.25,1);
      \node at (1.5,.5){$\scriptstyle \cdots$};
   \end{tikzpicture}
   \overset{\cref{reflection}}{=}
      \begin{tikzpicture}
       [anchorbase]
       \mirror{-.25}{0}{1};
       \draw (0,0)\botlabel{i_1} -- (0,1)\toplabel{};
       \draw (.25,0)\botlabel{i_2} -- (.25,1);
       \node at (.5,.5){$\scriptstyle \cdots$};
       \draw (.75,0)\botlabel{i_{l-1}} -- (.75,1);
       \draw (1.25,0)\botlabel{i_l} to[out=150,in=-15] (-.25,.3) to [out=15,in=-15] (-.25,.7) to[out=15,in=-150] (1.25,1);
      \node at (1.5,.5){$\scriptstyle \cdots$};
      \draw[dashed] (-.25,.65) -- (1.5,.65);
      \draw[dashed] (-.25,.35) -- (1.5,.35);
   \end{tikzpicture}
   \in (\oklrwzero^{\lamseq})\,e\,(\oklrwzero^{\lamseq})
\end{gather*}
where the middle identity morphism is a summand of $e$ since $\tau i_l\in Q_0^{(+)}$. This completes the proof.
\end{proof}

\subsection{Basis}
Recall the Weyl group of type $B_n$ is isomorphic to the signed symmetric group 
$SS_n$, the group of permutations $w$ of 
$\{\pm 1, \pm 2, \ldots, \pm n\}$ such that 
$w(-i)=-w(i)$. It is generated by 
$\{s_i\mid 0\le n-1 \}$, where $s_0=(-1,1)$ and $s_i=(i,i+1)(-i,-i-1)$. 

We will use  the mirror strand 
diagrams to represent each $w$ in $SS_n$. These are a special kind of $2n$-strand diagrams such that each strand connects $i$ and $w(i)$, but the picture is required to be symmetric across a central axis (the mirror). For example, 
\[ s_0= 
\begin{tikzpicture}
            [anchorbase]
            \mirror{-.5}{0}{1};
            \draw (0,0) \botlabel{1} to[out=up,in=right] (-.5,.5) to[out=right,in=down] (0,1)\toplabel{1};
            \draw (-1,0) \botlabel{-1} to[out=up,in=left] (-.5,.5) to[out=left,in=down] (-1,1)\toplabel{-1};
            \draw (.5,0)\botlabel{2} to (.5,1)\toplabel{2};
            \node at (1,.5) {$\cdots$};
            \draw (1.5,0)\botlabel{n} to (1.5,1)\toplabel{n};
             \draw (-1.5,0)\botlabel{-2} to (-1.5,1)\toplabel{-2};
            \node at (-2,.5) {$\cdots$};
            \draw (-2.5,0)\botlabel{-n} to (-2.5,1)\toplabel{-n};
        \end{tikzpicture}.
        \]
  So, each mirror pair of crossings corresponds to a simple reflection $s_i$, while a crossing which occurs on the top of the mirror corresponds to the simple refection $s_0$. 

  Since the picture is symmetric, we can only draw the right half of the picture so as to corresponds to the diagrams in $\cRB$. For example, 
  \[s_0= 
  \begin{tikzpicture}
            [anchorbase]
            \mirror{-.5}{0}{1};
            \draw (0,0) \botlabel{1} to[out=up,in=right] (-.5,.5) to[out=right,in=down] (0,1)\toplabel{1};
            \draw (.5,0)\botlabel{2} to (.5,1)\toplabel{2};
            \node at (1,.5) {$\cdots$};
            \draw (1.5,0)\botlabel{n} to (1.5,1)\toplabel{n};
        \end{tikzpicture}.
        \]
 For any  $\bi\in Q_0^n$ and $w\in SS_n$, let $w\bi=(j_1,\ldots,j_n)\in Q_0^n$ be obtained from 
 $\bi$ such that if $w(k)>0$,  then $j_{w(k)}=i_k$ and if $w(k)<0$, then $j_{-w(k)}=\tau(i_k) $. 

  We fix a sequence of dominant weight $\lamseq$ in the following. 
\begin{definition}
\label{def:reduceddiagram}
For any $w\in SS_n$ and  $(\bi,\redkappa)$
and $\redkappa':[1,\ell]\rightarrow [0,n]$, we  choose a Stendhal diagram (without dots)  $D_{w,\redkappa'}=1_{w\bi, \redkappa',\lamseq}D_{w,\redkappa'} 1_{\bi, \redkappa,\lamseq}$  from $ (\bi,\redkappa,\lamseq)$ to $(w\bi, \redkappa',\lamseq)$ such that 
\begin{enumerate}
    \item the underlying diagram without red lines is  a mirror strand  diagram of $w$
    such that no two black strands cross twice and each strand meets the mirror at most once. In other words, the crossing reading bottom to top is a reduced expression for $w$.
    \item the red and black strand which does not meet the mirror cross at most once. 
    \item the red and black strands which meet the mirror cross at most twice.
    We call such a diagram a  reduced Stendhal diagram for $w$ and $(\bi,\redkappa)$ and $\redkappa'$. 
\end{enumerate}
\end{definition}

\begin{example} For the following two diagrams, 
\[ D=\begin{tikzpicture}
            [anchorbase]
            \mirror{-.5}{-1}{1};
            \draw (0.3,-1)to [out=up,in=down](.5,-.5)  to[out=up,in=right] (-.5,.5) to[out=right,in=down] (0,1);
            \draw (.6,1) to [out=down,in=right] (-.5,0) to [out=right,in=up] (.9,-1);
            \draw (.6,-1) to (1.5,1);
            \draw (0.9,1) to (1.5,-1);
            \draw[red] (.2,1) to (-.2,-1);
              \draw[red] (1.2,1) to [out=down, in=up] (1.5,0) to [out=down,in=up] (1.2,-1);
                     \end{tikzpicture}, \quad 
       D'=\begin{tikzpicture}
            [anchorbase]
            \mirror{-.5}{-1}{1};
            \draw (0,-1)to [out=up,in=down](.3,-.5)  to[out=up,in=right] (-.5,.5) to[out=right,in=down] (0,1);
            \draw (.6,1) to [out=down,in=right] (-.5,0) to [out=right,in=up] (.9,-1);
            \draw (.6,-1) to (1.5,1);
            \draw (0.9,1) to (1.5,-1);
            \draw[red] (-.2,1) to   (.2,-1);
              \draw[red] (1.2,1) to [out=down, in=up] (1,0) to [out=down,in=up] (1.7,-1);
                     \end{tikzpicture}              \]
we have  that $D$ is a reduced Stendhal diagram. While $D'$ satisfies condition (1)—as its underlying diagram (omitting red lines) corresponds to a reduced expression for some $w \in SS_n$—it fails to satisfy (2) and (3). Specifically, the first red line intersects a black strand meeting the mirror three times, while the second red line intersects a non-mirror-meeting black strand twice.   Thus, $D'$ is not a  reduced Stendhal diagram.    
\end{example}

Let 
$B_{\redkappa',\redkappa}=\{ D_{w,\redkappa'}y^\mathbf a1_{\bi, \redkappa,\lamseq} \mid w\in SS_n, \bi\in Q_0^n, \mathbf a\in \N^n \}$, 
where $y^\mathbf a=y_1^{a_1}\ldots y_n^{a_n}$.

\begin{lemma}
\label{lem:spanpart}
    The set $B=\bigcup_{\redkappa',\redkappa}B_{\redkappa',\redkappa}$ 
    spans $\oklrwmu^{\lamseq}$.
\end{lemma}
\begin{proof}
We prove by induction on the number $c$ of crossings that any diagram $D$ can be written as a linear combination of elements in $B$. Note that we also view the meet between the strand and the mirror as a crossing since it represents $s_0$ in the language of the mirror strand diagram. 

If $c=0$, then $D$ is the identity morphism of some $1_{\bi,\redkappa,\lamseq} $ with some dots and hence $D\in B$. Suppose $c>0$.
By \eqref{dotcross}, \eqref{quadratic}, and \eqref{dotmovered}, dots can be moved freely through crossings modulo diagrams with fewer crossings. So, we can assume  the dots of $D$ are all at the bottom.  
Using the quadratic \eqref{doublecross}, \eqref{reddoublecross} and \eqref{quadratic}, and also 
 the braid  relations \eqref{braidA} and \eqref{reflection}, and \eqref{redbraid1}--\eqref{redbraid2},  we can assume that  
(1)--(3) in Definition \ref{def:reduceddiagram} hold for the underlying diagram dropping the dots of $D$. Finally, the braid relations also show that for any two choices $D'_{w,\redkappa'}$ and $D_{w,\redkappa'}$ for any  fixed $w$, we must have $D '_{w,\redkappa'}=D_{w,\redkappa'} $  modulo diagrams with fewer crossings. This completes the proof by inductive hypotheses.
\end{proof}

The proof of the following result follows the same triangularity argument as the KLRW basis theorem of Webster, combined with the PBW basis for oKLR algebras.

\begin{theorem}
\label{oklrwbasis}
   The set  $B=\bigcup_{\redkappa',\redkappa}B_{\redkappa',\redkappa}$ 
    is a basis of  $\oklrwmu^{\lamseq}$.
\end{theorem}
\begin{proof}By Lemma \ref{lem:spanpart}, it remains to show that $B$ is linearly independent.
  We first prove the claim that $\{D_{w,\red{n^\ell}}y^\mathbf a 1_{\bi, \red{n^\ell},\lamseq} \mid w\in SS_n, \bi\in Q_0^n, \mathbf a\in \N^n \}$ is linearly independent.
  In fact, we consider the obvious  functor from 
$ _\mu ^\tau\mathcal R$ to  $\oklrwmu^{\lamseq}$, sending each diagram $b$ to the diagram obtained by adding the 
 $\ell$ red lines to the right of $b$.
 Then composing this functor with the polynomial representation $\ctP$ yields the faithful polynomial representation of $ _\mu ^\tau\mathcal R$ in \cite{Prz23}. Hence the claim follows from the PBW basis of  $ _\mu ^\tau\mathcal R$ in \cite[Proposition 2.9]{Prz23} since $B_{\red{n^\ell},\red{n^\ell}}$ is the image of the corresponding  basis element of $ _\mu ^\tau\mathcal R$.

 Next,  we define for any $\redkappa$ an element $\theta_\redkappa$ such that it is a sum over all $\bi$ of the unique   Stendhal  diagram from  $ (\bi,\red{n^\ell}, \lamseq  ) $ to $ (\bi,\redkappa, \lamseq ) $ which has no dots and a minimal number of crossings. 
 Let $\bar \theta_\redkappa$ be the diagram obtained from $\theta_\redkappa$ by reflecting the diagram with respect to the horizontal axis. Then by the 
 same arguments as in \cite[Lemma 4.17]{Web15}, the diagram $\bar \theta_{\redkappa'}D_{w,\redkappa'}y^{\mathbf{a}} 1_{\bi,\redkappa,\lamseq} \theta_\redkappa$ is equal to 
 $D_{w,\red{n^\ell}} y^{\mathbf{a}+\mathbf{b}} 1_{\bi, \red{n^\ell},\lamseq} $ modulo the span of diagrams with fewer crossings than $D_{w, \red{n^\ell}}$, for some  $\mathbf{b}\in \mathbb Z_{\ge0}^n$ determined by the relations in \eqref{reddoublecross}.

 Now we are ready to prove the linear independence of $B_{\redkappa',\redkappa}$. Suppose
 there is a non-trivial linear relation $\sum c_{w,\mathbf{a}}D_{w,\redkappa'}y^{\mathbf{a}} 1_{\bi,\redkappa,\lamseq}=0$ with $c_{w,\mathbf a }\neq 0$.
 Multiply $\bar \theta_{\redkappa'}$ on the left and $\theta_\redkappa$ on the right. Then using the same argument as in the proof of  Lemma \ref{lem:spanpart}, each term of the summation can be rewritten as elements of form $D_{w',\red{n^\ell}}y^{\mathbf a'} 1_{\bi,\red{n^\ell},\lamseq}$.  Furthermore, choose a $w\in SS_n$ and $\mathbf a$ in the summation  (i.e., $c_{w,\mathbf a}\neq 0$) such that $w$ is maximal in the Bruhat order. Then we multiply  by $\bar \theta_{\redkappa'}$ on the left and $\theta_\redkappa$ on the right and rewrite in  terms of elements in $B_{\red{n^\ell},\red{n^\ell}}$.  The claim in the previous paragraph shows that $D_{w,\red{n^\ell}}y^{\mathbf a+\mathbf b}1_{\bi,\red{n^\ell},\lamseq} $ also has coefficient $c_{w,\mathbf a}$. since no element of  $B_{\red{n^\ell},\red{n^\ell}}$ other than  $D_{w',\red{n^\ell}}y^{\mathbf a'} $ could contribute to its coefficient. Since $B_{\red{n^\ell},\red{n^\ell}}$ is linearly independent (which was shown in the beginning of the proof), we must have $c_{w,\mathbf a}=0$, a contradiction. This show the original linear relation must be trivial and hence $B$ is linearly independent.   
\end{proof}

\begin{corollary}
\label{cor:faithfulpolyrep}
The polynomial representation from \cref{prop:polyrep} is faithful.
\end{corollary}

\begin{proof}
It suffices to prove faithfulness on each morphism space. Fix two objects
\[
(\bi,\redkappa,\lamseq)
\qquad\text{and}\qquad
(\bj,\redkappa',\lamseq)
\]
with \(n\) black strands, and set
\(
K:=\operatorname{Frac}\kk[Y_1,\ldots,Y_n].
\)
After extending the polynomial representation to \(K\), the action of a
reduced Stendhal diagram has the triangular form
\[
\ctP(D_{w,\redkappa'})
=
q_w w+\sum_{u<w}q_{u,w}u,
\qquad q_w\in K^\times,
\]
where \(<\) is the Bruhat order on \(SS_n\). Indeed, this follows directly
from \cref{Polklrw1,Polklrw2,Polreflect}: each black crossing or mirror
reflection has a nonzero coefficient of the corresponding simple
reflection, while all other terms have smaller Bruhat order; red--black
crossings act by multiplication by nonzero polynomials.

By \cref{oklrwbasis}, every element of the fixed morphism space has a
unique expression
\[
x=\sum_{w\in SS_n}
D_{w,\redkappa'}p_w(y)
1_{\bi,\redkappa,\lamseq},
\qquad
p_w(y)\in\kk[y_1,\ldots,y_n].
\]
Assume that \(x\neq0\), and choose \(w\) maximal in the Bruhat order such
that \(p_w\neq0\). The coefficient of \(w\) in \(\ctP(x)\) is then
\(
q_w\,w\bigl(p_w(Y)\bigr),
\)
which is nonzero. Since distinct elements of \(SS_n\) induce linearly
independent automorphisms of \(K\), this term cannot cancel. Hence
\(\ctP(x)\neq0\), proving the result.
\end{proof}

\begin{remark}
\label{rem:extend}
    The results of this section admit straightforward generalizations to quivers with involution that have $\tau$-fixed vertices.  
The corresponding oKLR algebra and its polynomial representation have already been developed in \cite{Prz23}, and one can define the analogous oKLRW category/algebra in the same way as in \cref{defoKLRWC}. However, for simplicity we do not treat this case here.
\end{remark}

\section{KLR iYangians}
\label{klrtY}
Recall the quiver with involution from \cref{sec:qwi}. 
Throughout this section we let $\lambda=\sum_{i\in Q_0}\lambda_i \varpi_i$ be a dominant weight, ${}^\tau \lambda=\lambda+\tau \lambda$, $\mu$ be a $\tau$-invariant weight such that ${}^\tau\lambda \geq \mu$. We also fix a set of parameters $\bR=(\bR_i)_{i\in Q_0}\in \prod_{i\in Q_0} \C^{\lambda_i}$ so that each $\bR_i$ is of size $\lambda_i$.

\subsection{Shifted iYangian}
 Recall the symmetric generalized Cartan matrix $C$ from \cref{sec:KLR} and the symmetric Kac-Moody algebra $\mathfrak{g}$. Following \cite[Def. 3.1]{SSX25} (see also \cite{LWW,LWW2}), we now define the shifted iYangian ${}^{\tau}\mathbf{Y}_{\mu}(\mathfrak{g})$ . 
\begin{definition}
\label{shiftedTY}
The shifted iYangian ${}^{\tau}\mathbf{Y}_{\mu}(\mathfrak{g})$ is the $\mathbb{C}(\hbar)$-algebra generated by 
$h_{i,r}$ and $b_{i,s}$ for $i\in Q_0, r\geq -( \alpha_i,\mu)-1$ and $s\geq 0$, subject to the following relations:
\begin{align}
\label{eq:hcomm}
    & {\left[h_{i, r}, h_{j, s}\right]=0, \qquad 
    h_{i,s} = (-1)^{s+1} h_{\tau i,s}}, 
        \\
\label{eq:hbcomm}
    & {\left[h_{i, r+2}, b_{j, s}\right]-\left[h_{i, r}, b_{j, s+2}\right]} 
    =\frac{c_{i j}-c_{\tau i, j}}{2} \hbar\left\{h_{i, r+1}, b_{j, s}\right\}
    \\\notag  & \qquad 
    +\frac{c_{i j}+c_{\tau i, j}}{2} \hbar\left\{h_{i, r}, b_{j, s+1}\right\}
    +\frac{c_{i j} c_{\tau i, j}}{4} \hbar^2\left[h_{i, r}, b_{j, s}\right], 
        \\
\label{eq:bcomm}
    & {\left[b_{i, r+1}, b_{j, s}\right]-\left[b_{i, r}, b_{j, s+1}\right]=\frac{c_{i j}}{2}\hbar \left\{b_{i, r}, b_{j, s}\right\}-2 \delta_{\tau i, j}(-1)^r h_{j, r+s+1}},
\end{align}
and the Serre-type relations  
\begin{itemize}
    \item when $c_{i j}=0$, we have 
\begin{equation}\label{eq:commSerre}
    \left[b_{i, r}, b_{j, s}\right]=\delta_{\tau i, j}(-1)^r h_{j, r+s};
\end{equation}
\item when $c_{i j}=-1, j \neq \tau i \neq i$, we have 
\begin{equation}\label{eq:usualSerre}
    \operatorname{Sym}_{k_1, k_2}\left[b_{i, k_1},\left[b_{i, k_2}, b_{j, r}\right]\right]=0;
\end{equation}
\item when $c_{i, \tau i}=-1$, we have 
\begin{equation}\label{eq:iSerre}
    \operatorname{Sym}_{k_1, k_2}\left[b_{i, k_1},\left[b_{i, k_2}, b_{\tau i, r}\right]\right]=\frac{4}{3} \operatorname{Sym}_{k_1, k_2}(-1)^{k_1} \sum_{p=0}^{\infty} 3^{-p}\left[b_{i, k_2+p}, h_{\tau i, k_1+r-p}\right].
\end{equation}
\end{itemize}
Here we have adopted the convention that 
\begin{equation}\label{eq:vanishing}
\mbox{$h_{i,r}=0$ if $r<-(\alpha_i, \mu)-1$ and $h_{i,r}=\zeta_i$ if $r=-(\alpha_i, \mu)-1$},
\end{equation}
where $\zeta$ is a family of parameters 
$\zeta = (\zeta_i)_{i \in Q_0} \in \bigl(\mathbb{C}(\hbar)^\times\bigr)^I$ satisfying 
\begin{equation}\label{eq:zeta}
  \zeta_i = (-1)^{(\alpha_i, \mu)}\,\zeta_{\tau(i)} 
   \in \mathbb{C}(\hbar)^\times.
\end{equation}
\end{definition}
We note that the right-hand side of \cref{eq:iSerre} is a finite sum by the convention \cref{eq:vanishing}. By \cite[Rem 3.3]{SSX25}, the shifted iYangian ${}^{\tau}\mathbf{Y}_{\mu}(\mathfrak{g})$ is independent of the choice of $\zeta$ up to an algebra automorphism. We then define the following generating functions
\begin{align}
\label{generatingfunction}
h_i(z) & = \hbar\sum_{r\in \mathbb{Z}}h_{i,r}z^{-r-1},&
h_i^\circ(z) & = \hbar\sum_{r\geq 0}h_{i,r}z^{-r-1},&
b_i(z) & = \hbar\sum_{s\geq 0}b_{i,s}z^{-s-1}.
\end{align}

\subsection{iGKLO representations}
\label{sec:iGKLO}
Suppose that ${}^\tau\lambda-\mu=\sum_{i\in Q_0} v_i \alpha_i$ with $v_i=v_{\tau i}\geq 0$. 
Similar to \cite[Sect. 2.1]{KTWWY19}, we fix a parity bipartition $$Q_0=Q_0^{(+)}\sqcup Q_0^{(-)}$$ and  we
call vertices in $Q_0^{(+)}$ {\bf even} and those in $ Q_0^{(-)}$
{\bf odd}. We orient $\Qui$ so that edges of $\Qui$ always point from even vertices to odd ones. Thus we have 
\[
\tau (Q_0^{(\pm)})=Q_0^{(\mp)}.
\]
We then fix any total order on $Q_0^{(+)}$. This also determines a total order on $Q_0^{(-)}$ such that $i<j$ if and only if $\tau i > \tau j$ for any $i,j\in Q_0^{(-)}$.

\begin{remark}
    We note that, the results in this section does not depend on the total order we choose on $Q_0^{(+)}$, unlike \cite[Sect. 2.1]{KTWWY19} which assumes the even indices come before the odd ones. The key point for us is that, by our assumption, there is no arrow between any two even vertices.
\end{remark}

\begin{example}[Type AIII]
    Consider a quiver with involution obtained from a Satake diagram of type AIII (see \cite{Ara62}) as follows:
    
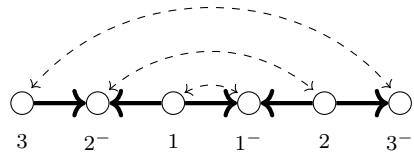
\[\begin{tikzpicture}[scale=1, 
                    arrow/.style={<->,dashed}]
\foreach \i in {0,...,5}
    \node[circle, draw, inner sep=3pt] (a\i) at (\i,0) {};
    \foreach \i in {0,...,5}
    \node (b\i) at (\i,.05) {};
    \node  at (0,-.5) {$\scriptstyle 3$};
    \node  at (1,-.5) {$\scriptstyle 2^-$};
    \node  at (2,-.5) {$\scriptstyle 1$};
    \node  at (3,-.5) {$\scriptstyle 1^-$};
    \node  at (4,-.5) {$\scriptstyle 2$};
    \node  at (5,-.5) {$\scriptstyle 3^-$};
\draw[arrow] (b0) to[bend left=50] (b5);
\draw[arrow] (b1) to[bend left=40] (b4);
\draw[arrow] (b2) to[bend left=30] (b3);
\draw[->,ultra thick] (a0) -- (a1); 
\draw[->,ultra thick] (a2) -- (a1); 
\draw[->,ultra thick] (a2) -- (a3); 
\draw[->,ultra thick] (a4) -- (a3); 
\draw[->,ultra thick] (a4) -- (a5); 
\end{tikzpicture}
\]
where $\tau$ is indicated by the dashed arrows. We can choose \[
Q_0^{(+)}=\{1,2,3\},\quad Q_0^{(-)}=\{1^-,2^-,3^-\}.
\]
The total order on $Q_0^{(+)}$ can be chosen as $1<2<3$, then the corresponding total order on $Q_0^{(-)}$ is $3^-<2^-<1^-$.
\end{example}

 Consider the polynomial ring
$$ P = \C[\hbar][x_{i,j} : i \in Q_0^{(+)}, 1 \leqslant j \leqslant v_i ]. $$
Let $ \Sigma = \prod_{i \in  Q_0^{(+)}} \Sigma_{v_i} $ be the product of symmetric groups, acting on $P$ by permuting each set of variables.  We will be interested in the $\Sigma$-invariant polynomials $P^\Sigma$. For $i\in Q_0^{(+)}$ and $1\leq j\leq v_i$, let $d_{i,j}$ denote the algebra automorphism $P\to P$ that shifts the variable $x_{i,j}$ by $\hbar$. We also introduce $x_{i,j}=-x_{\tau i,j}$ and $d_{i,j}=d_{\tau i,j}^{-1}$ for $i\in Q_0^{(-)}$. 
Hence, we have the following relations
$$
d_{i_1,j_1}x_{i_2,j_2}
=
(x_{i_2,j_2}+
(\delta_{i_1,i_2}
-\delta_{i_1,\tau i_2 })
\delta_{j_1,j_2}\hbar
)d_{i_1,j_1}
$$
for any $i_1,i_2\in Q_0$, $1\leq j_1\leq v_{i_1}$ and $1\leq j_2\leq v_{i_2}$. 
For $i\in Q_0$ and $1\leq r\leq v_i$, we define
$$V_i(z) = \prod_{k=1}^{v_i}(z-x_{i,k}),\quad V_{i,r}(z)=\frac{V_i(z)}{z-x_{i,r}}.$$
 For each $i\in Q_0$, we also define
 $p_i(u):=\prod_{c\in \bR_i}(u-c\hbar)$.
\begin{prop} 
    \label{prop:iGKLO} 
    There is an action of ${}^\tau Y_\mu $ on $P^\Sigma$, depending on $\bR$, defined by 
    \begin{gather}
       \label{biGKLO} b_i(z)\mapsto  \sum_{r=1}^{v_i}\frac{1}{-z-x_{i,r}-\tfrac{\hbar}{2}} \prod_{i\to \tau i}V_{\tau i,r}(x_{i,r}+\tfrac{\hbar}{2})\prod_{j\neq \tau i \atop i \to j}V_{j}(x_{i,r}+\tfrac{\hbar}{2})\frac{p_{\tau i}(x_{\tau i,r})}{V_{i,r}(x_{i,r})}d_{i,r}\\
       \label{hiGKLO}
       h_i(z) \mapsto (-1)^{v_i-1}(2z)^{c_{i,\tau i}}(-1)^{\delta_{i\to \tau(i)}}\frac{p_i(-z)p_{\tau i}(z)}{V_i(-z+\frac{\hbar}{2})V_i(-z-\tfrac{\hbar}{2})}\prod_{i\to j}V_{j}(-z)\prod_{\tau i \to j}V_{j}(z)
    \end{gather}
\end{prop}
\begin{proof}
 By \cite[Thm.~4.2]{SSX25} (see also \cite[Thm.~3.5]{LWW2}), the map above indeed defines an action of ${}^{\tau}\mathbf{Y}_{\mu}(\mathfrak{g})$ on the field of fractions $\operatorname{Frac}(P)$. To complete the proof, it remains to show that this action preserves the subalgebra $P^{\Sigma}$.  
This is proved in the same manner as \cite[Thm.~4.6]{KTWWY19}, using \cite[Lem.~4.7]{KTWWY19}.
\end{proof}

Following \cite{KWWY14,LWW}, we define a “Cartan” series $a_i(z)$ for $i\in Q_0$ by
\begin{gather}
    \label{Aseries}
     h_i(z)= (-1)^{v_i-1}(2z)^{c_{i,\tau i}}(-1)^{\delta_{i\to \tau(i)}}\frac{\prod_{i\to j}(-z)^{v_{j}}a_{j}(-z)\prod_{\tau i \to j}z^{v_{j}}a_{j}(z)}{(z^2-\frac{\hbar^2}{4})a_i(-z+\frac{\hbar}{2})a_i(-z-\tfrac{\hbar}{2})}p_i(-z)p_{\tau i}(z).
\end{gather}
By \cref{prop:iGKLO} and \cref{Aseries}, we obtain that
\begin{corollary} \cite[Lem. 3.9]{LWW2}
     \label{iGKLO2} 
    The action of ${}^\tau Y_\mu $ on $P^\Sigma$ in \cref{prop:iGKLO} sends
    \begin{gather}
           \label{aiGKLO}
       a_i(z) \mapsto z^{-v_i} V_i(z).
    \end{gather}
\end{corollary}

By \cref{prop:iGKLO},  we have a homomorphism (depending on $\bR$)
$$ {}^\tau Y_\mu \longrightarrow \End( P^\Sigma). $$
We define the {\bf truncated shifted iYangian} $ {}^\tau Y^\lambda_\mu = {}^\tau Y^\lambda_\mu(\bR) $ to be the image of $ {}^\tau Y_\mu $ in $ \End(P^\Sigma)$.

\subsection{KLR iYangians and polynomial representations}

We will now introduce a bigger algebra (depending on $ \bR$) inspired by \cite[Sect. 4.4]{KTWWY19}.
\begin{definition}
\label{def:twomirrordiagram}
 A \textbf{double reflective KLR diagram} is a collection of finitely many black curves in $\R \times [0,1]$ with endpoints on the boundary lines $\R \times \{0\}$ and $\R \times \{1\}$, together with two vertical mirrors.  Each curve may carry finitely many dots.  
Moreover, each curve has one endpoint on $\R \times \{0\}$ labelled by some $i \in Q_0$, and one endpoint on $\R \times \{1\}$ labelled by either $i$ or $\tau i$, according to the parity of the number of times the curve touches the mirrors. Diagrams are considered up to isotopy, and are required to be locally of the following forms:
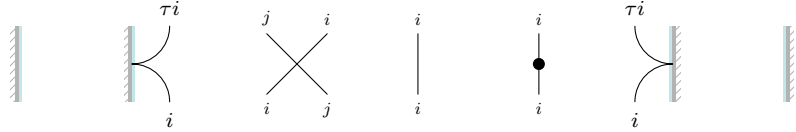
\begin{equation*}
 \begin{tikzpicture}
            [anchorbase]
            \mirror{-2}{0}{1};
            \mirror{-.5}{0}{1};
            \draw (0,0)\botlabel{i}  to[out=up,in=right] (-.5,.5) to[out=right,in=down] (0,1) \toplabel{\tau i};
        \end{tikzpicture}
\qquad
\begin{tikzpicture}[anchorbase,scale=.8]
  \draw (-2,0) +(-.5,-.5) \botlabel{i} -- +(.5,.5) \toplabel{i};
  \draw(-2,0) +(.5,-.5)\botlabel{j} -- +(-.5,.5)\toplabel{j};

 \draw(0,0) +(0,-.5) \botlabel{i} --  +(0,.5)\toplabel{i};

  \draw(2,0) +(0,-.5) \botlabel{i}--  node
  [midway,circle,fill=black,inner sep=2pt]{}
  +(0,.5) \toplabel{i};
\end{tikzpicture}
\qquad
\begin{tikzpicture}
            [anchorbase]
            \rmirror{2}{0}{1};
            \rmirror{.5}{0}{1};
            \draw (0,0)\botlabel{i}  to[out=up,in=left] (.5,.5) to[out=left,in=down] (0,1)\toplabel{\tau i};
        \end{tikzpicture}
\end{equation*}
\end{definition}
As usual, the product of two such diagrams is defined by stacking them on top of each other if their labels match, and is zero otherwise. 

Given any word $\bi\in \word$, we let $e(\bi) $ be the double reflective KLR diagram with vertical strands with these labels in order.  We write $y_k(\bi) $ for the same diagram where the $k$-th strand carries a dot, and $\sigma_k(\bi)$ for the diagram where the $k$-th and ($k+1$)-th strands cross.  (When the labels are clear from the context we will simply write $ y_k$ and $\sigma_k$.). Diagrammatically if $\bi=(i_1,...,i_m)$ then $e(\bi), y_k(\bi)$, and $\sigma_k(\bi)$ are given by:

\begin{equation*}
e(\bi)=\,\begin{tikzpicture}[baseline,scale=1.2]
\mirror{-3.7}{-.5}{.5};
 \draw[-] (-3.4,0) +(0,-.5) -- +(0,.5);
  \draw[-] (-2,0) +(0,-.5) -- +(0,.5);
  \rmirror{-1.7}{-.5}{.5};
  \draw (-3.4,0) +(0.1,-.8) node {\small$i_1$};
  \draw (-2,0) +(0.1,-.8) node {\small$i_m$};
  \draw (-2.65,0) node {$\cdots$};  \end{tikzpicture}
\qquad   y_k(\bi)=\,\begin{tikzpicture}[baseline,scale=1.2]
\mirror{0}{-.5}{.5};
 \draw[-] (.3,0) +(0,-.5) -- +(0,.5);
  \draw[-] (2,0) +(0,-.5) -- +(0,.5);
  \rmirror{2.3}{-.5}{.5};
  \draw (.3,0) +(0.1,-.8) node {\small$i_1$};
  \draw (2,0) +(0.1,-.8) node {\small$i_m$};
  \draw (.75,0) node {$\cdots$};
  \draw (1.2,-.5) -- (1.2,.5);
   \bdot{1.2,0};
     \draw (1.65,0) node {$\cdots$};
      \draw (1.2,0) +(0.1,-.8) node {\small$i_k$};  \end{tikzpicture}
\qquad  \sigma_k(\bi)=\,\begin{tikzpicture}[baseline,scale=1.2]
     \mirror{4}{-.5}{.5};
 \draw[-] (4.3,0) +(0,-.5) -- +(0,.5);
  \draw[-] (6,0) +(0,-.5) -- +(0,.5);
 \rmirror{6.3}{-.5}{.5};
  \draw (4.3,0) +(0.1,-.8) node {\small$i_1$};
  \draw (6,0) +(0.1,-.8) node {\small$i_m$};
  \draw[-] (4.9,-.5)\braidup(5.4,.5);
  \draw[-] (5.4,-.5) \braidup (4.9,.5);
\draw (4.9,0) +(0.1,-.8) node {\small$i_k$};
\draw (5.4,0) +(0.1,-.8) node {\small$i_{k+1}$};
\draw (4.75,0) node {$\cdots$};
\draw (5.72,0) node {$\cdots$};
\end{tikzpicture}
\end{equation*}
We also have left and right reflections
\begin{gather*}
    \rho_L(\bi)=
    \begin{tikzpicture}
        [anchorbase]
        \mirror{0}{0}{1};
        \draw (.5,0) \botlabel{i_1} to [out=up,in=right] (0,.5) to[out=right,in=down] (.5,1) \toplabel{\tau i_1};
        \draw (.9,0) \botlabel{i_2} -- (.9,1);
        \node at (1.35,.5){$\cdots$};
        \draw (1.8,0) \botlabel{i_m} -- (1.8,1);
        \rmirror{2.1}{0}{1};
    \end{tikzpicture},
    \qquad
     \rho_R(\bi)=
    \begin{tikzpicture}
        [anchorbase,xscale=-1]
        \mirror{0}{0}{1};
        \draw (.5,0) \rbotlabel{i_m} to [out=up,in=right] (0,.5) to[out=right,in=down] (.5,1) \rtoplabel{\tau i_m};
        \draw (.9,0)  -- (.9,1);
        \node at (1.35,.5){$\cdots$};
        \draw (1.8,0) \rbotlabel{i_1} -- (1.8,1);
        \rmirror{2.1}{0}{1};
    \end{tikzpicture}.
\end{gather*}
Let
\begin{gather}
\label{shiftX}
    \overline{X}_{ij}(u,v)=
 \begin{cases}
  1 & i \not \leftarrow j,\\
 u-v-\frac{\hbar}{2} & i\leftarrow j.\\
\end{cases} \qquad \overline{f}_{ij}(u,v)=\overline{X}_{ij}(u,v)\overline{X}_{ji}(v,u)=
\begin{cases}
  1, & i \not \leftrightarrow j,\\
 u-v-\frac{\hbar}{2}, & i\leftarrow j,\\
v-u-\frac{\hbar}{2}, & i\to j.
\end{cases}
\end{gather}

\begin{definition}
    \label{def:Ya}
    The KLR iYangian algebra $\tR= \tR(\bR)$ is the quotient of the span of double reflective KLR diagrams by the following local relations:
    \begin{itemize}
        \item the usual KLR relations \cref{dotcross}--\cref{braidA} using $\overline{f}_{ij}(u,v)$,
        \item the  relations \eqref{quadratictwo}-\eqref{reflectiontwo} around mirrors: 
   \begin{gather}
    \label{quadratictwo}
        \begin{tikzpicture}
            [anchorbase,xscale=-1]
            \mirror{.25}{0}{1};
             \draw (0.5,0) \rbotlabel{i_m} to[out=up,in=right] (.25,.25) to[out=right,in=down] (0.5,.5) to [out=up,in=right] (.25,.75) to[out=right,in=down] (0.5,1)\rtoplabel{i_m};
             \draw (1,0) \rbotlabel{i_{m-1}} -- (1,1);
        \node at (1.35,.5){$\cdots$};
        \draw (1.8,0) \rbotlabel{i_1} -- (1.8,1);
        \rmirror{2.1}{0}{1};
        \end{tikzpicture}
        =
         \begin{tikzpicture}
            [anchorbase,xscale=-1]
            \mirror{.25}{0}{1};
             \draw (.5,0) \rbotlabel{i_m} -- (.5,1)\rtoplabel{i_m};
             \draw (1,0) \rbotlabel{i_{m-1}} -- (1,1);
        \node at (1.35,.5){$\cdots$};
        \draw (1.8,0) \rbotlabel{i_1} -- (1.8,1);
        \rmirror{2.1}{0}{1};
        \end{tikzpicture},\\
             \begin{tikzpicture}
        [anchorbase,xscale=-1]
        \mirror{-.25}{0}{1};
       \draw (0,0) \rbotlabel{i} to[out=up,in=-15] (-.25,.2) to [out=15,in=down] (.3,.6) \braidup (0,1) \rtoplabel{\tau i};
        \draw (.3,0) \rbotlabel{j} to[out=up,in=right] (-.25,.5)  to [out=right,in=down] (.3,1) \rtoplabel{\tau j};
        \draw (.9,0) \rbotlabel{i_{m-2}} -- (.9,1);
        \node at (1.35,.5){$\cdots$};
        \draw (1.8,0) \rbotlabel{i_1} -- (1.8,1);
        \rmirror{2.1}{0}{1};
    \end{tikzpicture}
    =
    \begin{tikzpicture}
        [anchorbase,xscale=-1]
        \mirror{-.25}{0}{1};
        \draw (0,0) \rbotlabel{i} \braidup (.3,.4) to [out=up,in=-15] (-.25,.8) to [out=15,in=down] (0,1) \rtoplabel{\tau i};
        \draw (.3,0) \rbotlabel{j} to[out=up,in=right] (-.25,.5)  to [out=right,in=down] (.3,1) \rtoplabel{\tau j};
        \draw (.9,0) \rbotlabel{i_{m-2}} -- (.9,1);
        \node at (1.35,.5){$\cdots$};
        \draw (1.8,0) \rbotlabel{i_1} -- (1.8,1);
        \rmirror{2.1}{0}{1};
    \end{tikzpicture},  \\
 \label{rightdotmove}  \begin{tikzpicture}
            [anchorbase,xscale=-1]
            \mirror{-.5}{0}{1};
            \draw (0,0) \rbotlabel{i} to[out=up,in=right] (-.5,.5) to[out=right,in=down] (0,1)\rtoplabel{\tau i};
            \bdot{-.1,.3};
            \draw (.9,0) \rbotlabel{i_{m-1}} -- (.9,1);
        \node at (1.35,.5){$\cdots$};
        \draw (1.8,0) \rbotlabel{i_1} -- (1.8,1);
        \rmirror{2.1}{0}{1};
        \end{tikzpicture}
        =-\ 
         \begin{tikzpicture}
            [anchorbase,xscale=-1]
            \mirror{-.5}{0}{1};
            \draw (0,0) \rbotlabel{i} to[out=up,in=right] (-.5,.5) to[out=right,in=down] (0,1)\rtoplabel{\tau i};
            \bdot{-.1,.7};
            \draw (.9,0) \rbotlabel{i_{m-1}} -- (.9,1);
        \node at (1.35,.5){$\cdots$};
        \draw (1.8,0) \rbotlabel{i_1} -- (1.8,1);
        \rmirror{2.1}{0}{1};
        \end{tikzpicture},\\
        \begin{tikzpicture}
            [anchorbase]
            \mirror{.25}{0}{1};
             \draw (0.5,0) \botlabel{i_1} to[out=up,in=right] (.25,.25) to[out=right,in=down] (0.5,.5) to [out=up,in=right] (.25,.75) to[out=right,in=down] (0.5,1)\toplabel{i_1};
             \draw (.9,0) \botlabel{i_{2}} -- (.9,1);
        \node at (1.35,.5){$\cdots$};
        \draw (1.8,0) \botlabel{i_m} -- (1.8,1);
        \rmirror{2.1}{0}{1};
        \end{tikzpicture}
        =
        \begin{tikzpicture}
            [anchorbase]
            \mirror{-.25}{0}{1};
             \draw (.6,0) \botlabel{i_1} -- (.6,1)\toplabel{i_1};
             \coupon{.6,.5}{f_{i_1}(X_1)};
             \draw (1.9,0) \botlabel{i_{2}} -- (1.9,1);
        \node at (2.35,.5){$\cdots$};
        \draw (2.8,0) \botlabel{i_m} -- (2.8,1);
        \rmirror{3.1}{0}{1};
        \end{tikzpicture},\\
         \begin{tikzpicture}
            [anchorbase]
            \mirror{-.5}{0}{1};
            \draw (0,0) \botlabel{i_1} to[out=up,in=right] (-.5,.5) to[out=right,in=down] (0,1)\toplabel{\tau i_1};
            \bdot{-.1,.3};
            \draw (.9,0) \botlabel{i_{2}} -- (.9,1);
        \node at (1.35,.5){$\cdots$};
        \draw (1.8,0) \botlabel{i_m} -- (1.8,1);
        \rmirror{2.1}{0}{1};
        \end{tikzpicture}
    \ =\ -
        \ 
         \begin{tikzpicture}
            [anchorbase]
            \mirror{-.5}{0}{1};
            \draw (0,0) \botlabel{i_1} to[out=up,in=right] (-.5,.5) to[out=right,in=down] (0,1)\toplabel{\tau i_1};
            \bdot{-.1,.7};
            \draw (.9,0) \botlabel{i_{2}} -- (.9,1);
        \node at (1.35,.5){$\cdots$};
        \draw (1.8,0) \botlabel{i_m} -- (1.8,1);
        \rmirror{2.1}{0}{1};
        \end{tikzpicture}
      \ +\ \hbar\  \begin{tikzpicture}
            [anchorbase]
            \mirror{-.5}{0}{1};
            \draw (0,0) \botlabel{i_1} to[out=up,in=right] (-.5,.5) to[out=right,in=down] (0,1)\toplabel{\tau i_1};
            \draw (.9,0) \botlabel{i_{2}} -- (.9,1);
        \node at (1.35,.5){$\cdots$};
        \draw (1.8,0) \botlabel{i_m} -- (1.8,1);
        \rmirror{2.1}{0}{1};
        \end{tikzpicture}  ,\\
    \label{reflectiontwo}
    \begin{tikzpicture}
        [anchorbase]
        \mirror{-.25}{0}{1};
       \draw (0,0) \botlabel{i} to[out=up,in=-15] (-.25,.2) to [out=15,in=down] (.3,.6) \braidup (0,1) \toplabel{\tau i};
        \draw (.3,0) \botlabel{j} to[out=up,in=right] (-.25,.5)  to [out=right,in=down] (.3,1) \toplabel{\tau j};
         \draw (.9,0) \botlabel{i_{3}} -- (.9,1);
        \node at (1.35,.5){$\cdots$};
        \draw (1.8,0) \botlabel{i_m} -- (1.8,1);
        \rmirror{2.1}{0}{1};
    \end{tikzpicture}
    -
    \begin{tikzpicture}
        [anchorbase]
        \mirror{-.25}{0}{1};
        \draw (0,0) \botlabel{i} \braidup (.3,.4) to [out=up,in=-15] (-.25,.8) to [out=15,in=down] (0,1) \toplabel{\tau i};
        \draw (.3,0) \botlabel{j} to[out=up,in=right] (-.25,.5)  to [out=right,in=down] (.3,1) \toplabel{\tau j}; \draw (.9,0) \botlabel{i_{3}} -- (.9,1);
        \node at (1.35,.5){$\cdots$};
        \draw (1.8,0) \botlabel{i_m} -- (1.8,1);
        \rmirror{2.1}{0}{1};
    \end{tikzpicture}
    =\delta_{\tau i,j}
        \begin{tikzpicture}
            [anchorbase]
            \mirror{-.8}{0}{1.2};
            \draw (.4,0)\botlabel{i} \braidup (.8,.4) -- (.8,1.2) \toplabel{i};
            \draw (.8,0)\botlabel{j} \braidup (.4,.4) -- (.4,1.2) \toplabel{j};
            \coupon{.6,.7}{\frac{f_j(X_1)-f_i(X_{2})}{-X_1-X_{2}+\hbar}}; 
            \draw (2.1,0) \botlabel{i_{3}} -- (2.1,1);
        \node at (2.45,.5){$\cdots$};
        \draw (2.9,0) \botlabel{i_m} -- (2.9,1);
        \rmirror{3.1}{0}{1};
        \end{tikzpicture},
    \end{gather} 
    where $f_i(u)=p_i(u)p_{\tau i}(-u+\hbar)$.
    \end{itemize}
\end{definition}

Let$$\Pol = \bigoplus_{\substack{m \geq 0, \\ \bi\in Q_0^m }} \Pol_\bi, \quad \Pol_\bi := \C[\hbar][X_1(\bi),\dots, X_m(\bi)]$$
be a direct sum of polynomial rings, one for each word $ \bi$. For $\bi=(\bi_1,\ldots,\bi_m)$, we define
\[
s_0(\bi)=(\tau \bi_1,\bi_2,\ldots,\bi_m),\quad s_k(\bi)=(\ldots,\bi_{k+1},\bi_{k},\ldots) (1\leq k<m) ,\quad s_m(\bi)=(\bi_1,\ldots,\bi_{m-1},\tau \bi_m)
\]
and extend these to automorphisms of $\Pol$ by
\begin{gather}
\label{mirrorpol}
s_0(X_j(\bi))
=
\begin{cases}
-X_1(s_0\bi)+\hbar, & j=1,\\
X_j(s_0\bi), & j>1,
\end{cases}
\qquad
s_m(X_j(\bi))
=
\begin{cases}
-X_m(s_m\bi), & j=m,\\
X_j(s_m\bi), & j<m.
\end{cases}
\end{gather}
and \[
s_k(X_j(\bi))
=
\begin{cases}
X_{k+1}(s_k\bi), & j=k,\\
X_k(s_k\bi), & j=k+1,\\
X_j(s_k\bi), & j\neq k,k+1.
\end{cases}
\]
Whenever we have a word $ \bi $ with $ \bi_k = \bi_{k+1} $, we also define the divided difference operator $ \partial_k : \Pol_\bi \rightarrow \Pol_\bi $ by \[\partial_k =\frac{1}{X_{k+1}(\bi) - X_k(\bi)}(s_k - 1),\quad k\geq 1.\]

\begin{theorem}
\label{thm:tRPol}
    There is a faithful action of $\tR$ on $\Pol$ sending:
    \begin{itemize}
        \item $e(\bi)$ to the projection onto $\Pol_{\bi}$,
         \item  $y_k(\bi)$ to the multiplication operator $X_k(\bi)$,
         \item  \(\sigma_k(\bi) \) to the operator \(
    \begin{cases}
      \partial_k & \bi_k=\bi_{k+1}\\
      s_k \circ \overline{X}_{i_k,i_{k+1}}(X_k(\bi),X_{k+1}(\bi)) &\bi_k\neq \bi_{k+1}
    \end{cases}
\)
\item $\rho_L(\bi)$ to the operator
\[
p_{\tau i_1}\bigl(X_1(s_0\bi)\bigr)s_0,
\] 
\item $\rho_R(\bi)$ to the operator $s_m$,
    \end{itemize}
    for $\bi=(\bi_1,\ldots,\bi_m)$.
\end{theorem}

\begin{proof}
The KLR relations are verified exactly as in
\cref{prop:polyrep}, after omitting the red strands. The relations
involving the two mirrors follow directly from \cref{mirrorpol} and
the definitions of the operators \(\rho_L\) and \(\rho_R\). Hence the
displayed formulas define an action of \(\tR\) on \(\Pol\).

It remains to prove faithfulness. The argument is similar to
  \cite[Thm.~4.20]{KTWWY19}. Filter
\(\tR\) by the length of the underlying word in
\[
    s_0,s_1,\ldots,s_{m-1},s_m
\]
after forgetting the dots. The defining relations imply that
\(\tR\) is spanned by reduced diagrams with polynomials in the dots
placed at the bottom.

After extending scalars to the fraction field of \(\Pol\), the action
of a reduced diagram with underlying affine Weyl group element \(w\)
has the form
\[
    q_w w
    +
    \sum_{\ell(u)<\ell(w)} q_u u,
    \qquad
    q_w\neq 0,
\]
where the coefficients are rational functions. Indeed, every
crossing or mirror reflection has a nonzero leading coefficient
multiplying the corresponding simple reflection. Since distinct
affine Weyl group elements act as linearly independent automorphisms
of the fraction field, no nonzero linear combination of reduced
diagrams can act by zero. Therefore the polynomial representation is
faithful.
\end{proof}

\begin{remark}
    The framing in \cref{def:Ya} is one-sided: all framing parameters are attached to the left mirror, while the right mirror is unframed. In fact, one may distribute the framing data between the two mirrors and define a two-sided analogue of $\tR$ and its polynomial representation.
\end{remark}

\subsection{Diagrammatic iGKLO homomorphisms}
\label{sec:diagiGKLO}
Recall that we fixed $\lambda$ such that ${}^\tau\lambda-\mu=\sum_{i\in Q_0}v_i\alpha_i$ where $v_i=v_{\tau i}$ for any $i\in Q_0$. 

\newcommand{\NH}{\operatorname{NH}_\aleph}

We now define a sequence $\aleph$ by listing the vertices of $Q_0^{(+)}$ according to the chosen total order (from smallest to largest), with each vertex $i$ appearing exactly $v_i$ times.  
Let $e(\aleph)$ denote the corresponding idempotent in $\tR$.
By \cref{thm:tRPol}, we can view 
\(
e(\aleph)\tR e(\aleph)
\)
as a subalgebra of the endomorphism ring of $\Pol_{\aleph}$. Suppose $m=\sum_{i\in Q_0^{(+)}}v_i=\sum_{i\in Q_0^{(-)}}v_{i}$. Using the chosen total order on $Q_0^{(+)}$, we identify the variables $x_{i,k}\ (i\in Q_0^{(+)},1\leq k \leq v_i)$ with the variables $ X_1(\aleph), \ldots, X_m(\aleph)$. Using the above, we may identify $\Pol_{\aleph}$ with $P$. The algebra $e(\aleph) \tR e(\aleph)$ naturally contains a tensor product $\bigotimes_i \operatorname{NH}_{v_i}$ of nilHecke algebras, which we denoted by $\NH$. 

By \cite[Prop. 4.15]{KTWWY19}, $\End_{P^\Sigma} P $ contains the projection operator onto the subspace $ P^\Sigma $.  This projection is given by
\begin{equation}
\label{symidempotent}
e_{\Sigma} := \prod_{i\in Q_0^{(+)}} \frac{1}{v_i !}\partial_{i, w_0} \prod_{1 \le k < l \le v_i} (x_{i,k} - x_{i,l} )
\end{equation}
where $ \partial_{i,w_0} $ denotes the product of the divided difference operators $ \partial_{i,r} $ corresponding to the longest element in $ \Sigma_{v_i} $. 

For each $i\in Q_0^{(+)}$, we define the following elements diagrammatically in $e(\aleph)\tR e(\aleph)$:
\begin{gather}
    \label{Di}
    D_{i}:=
    \begin{tikzpicture}
        [anchorbase,scale=1.4]
        \mirror{-.1}{0}{1};
        \draw[-] (.125,0) -- (.125,1)  \toplabel{}
                 (.5,0) -- (.5,1)
                 (.825,0) -- (.825,1)
                 (1.2,0) -- (1.2,1)
                 (1.5,0) -- (1.5,1)
                 (1.9,0) -- (1.9,1);
        \rmirror{2.125}{0}{1};
        \draw[-] (0.625,0) -- (-.1,1/4) -- (2.125,2/3) -- (0.625,1);
        \node at (.3125,.5){$\scriptstyle \cdots$};
         \node at (1.0125,.5){$\scriptstyle \cdots$};
          \node at (1.7,.5){$\scriptstyle \cdots$};
           \draw [decorate,decoration={brace,amplitude=5pt}]  (.5,-.05) -- (.125,-.05); \node at (.625/2,-.35){$\scriptstyle <i$};
           \draw [decorate,decoration={brace,amplitude=5pt}]  (1.2,-.05) -- (.625,-.05); \node at (1.825/2,-.35){$\scriptstyle i$};
           \draw [decorate,decoration={brace,amplitude=5pt}]  (1.9,-.05) -- (1.5,-.05); \node at (1.7,-.35){$\scriptstyle >i$};
    \end{tikzpicture}
,\qquad
    D_{\tau i}:=
       \begin{tikzpicture}
        [anchorbase,scale=1.4]
        \mirror{-.1}{0}{1};
        \draw[-] (.125,0) -- (.125,1)  \toplabel{}
                 (.5,0) -- (.5,1)
                 (.825,0) -- (.825,1)
                 (1.2,0) -- (1.2,1)
                 (1.5,0) -- (1.5,1)
                 (1.9,0) -- (1.9,1);
        \rmirror{2.125}{0}{1};
        \draw[-] (1.375,0) -- (2.125,1/4) -- (-.1,2/3) -- (1.375,1);
        \node at (.3125,.5){$\scriptstyle \cdots$};
         \node at (1.0125,.5){$\scriptstyle \cdots$};
          \node at (1.7,.5){$\scriptstyle \cdots$};
           \draw [decorate,decoration={brace,amplitude=5pt}]  (.5,-.05) -- (.125,-.05); \node at (.625/2,-.35){$\scriptstyle <i$};
           \draw [decorate,decoration={brace,amplitude=5pt}]  (1.375,-.05) -- (.8,-.05); \node at (2.175/2,-.35){$\scriptstyle i$};
           \draw [decorate,decoration={brace,amplitude=5pt}]  (1.9,-.05) -- (1.5,-.05); \node at (1.7,-.35){$\scriptstyle >i$};
    \end{tikzpicture}
\end{gather}
where the braces at the bottom denote labels from $Q_0^{(+)}$, appearing in the order of $\aleph$.

For \(i\in Q_0^{(+)}\), set
\[
\delta_i:=\delta_{i\to\tau i},
\qquad
\epsilon_i:=
v_i+\sum_{j\to\tau i}v_j-\delta_i.
\]
\begin{theorem} (cf. \cite[Thm. 4.22]{KTWWY19})
\label{lambsh}
As subalgebras of $\End(P)$, we have \[\trust\subset e(\aleph)\tR e(\aleph),\] where
    the elements $b_{i,r}$ of ${}^\tau Y_\mu^{\lambda}$ correspond to elements of $\tR$ as follows:
    \begin{gather}
\label{biimage}
b_{i,r}
=
(-1)^{r+\epsilon_i}
D_i
\left(
y_{1+\sum_{j<i}v_j}-\frac{\hbar}{2}
\right)^r
e_\Sigma,
\qquad i\in Q_0^{(+)},
\\
\label{btiimage}
b_{\tau i,r}
=
-
D_{\tau i}
\left(
y_{\sum_{j\leq i}v_j}+\frac{\hbar}{2}
\right)^r
e_\Sigma,
\qquad \tau i\in Q_0^{(-)}.
\end{gather}
\end{theorem}
\begin{proof}
We identify $\trust$ with its action on $P^\Sigma$ and regard
$\Hom(P^\Sigma,P)$ as the subspace $\End(P)e_\Sigma$ of $\End(P)$.
We verify \cref{biimage,btiimage} by comparing the corresponding
operators on $P^\Sigma$.

For $1\leq a\leq v_i$, define
\[
G_{i,a}
:=
p_{\tau i}(x_{\tau i,a})
\left(
\prod_{i\to\tau i}
V_{\tau i,a}
\left(x_{i,a}+\frac{\hbar}{2}\right)
\right)
\left(
\prod_{\substack{j\neq\tau i\\i\to j}}
V_j
\left(x_{i,a}+\frac{\hbar}{2}\right)
\right).
\]

We first compute the polynomial action of $D_i$. In the middle part of
the diagram, the distinguished strand is labelled by $\tau i$ and
crosses the strands labelled by $j\in Q_0^{(+)}$. If $j\neq i$ and
$j\to\tau i$, then
\[
\prod_{k=1}^{v_j}
\overline X_{\tau i,j}
(x_{\tau i,a},x_{j,k})
=
\prod_{k=1}^{v_j}
\left(
x_{\tau i,a}-x_{j,k}-\frac{\hbar}{2}
\right)
=
(-1)^{v_j}
V_{\tau j}
\left(x_{i,a}+\frac{\hbar}{2}\right).
\]
Under the substitution $j\mapsto\tau j$, these are precisely the
factors
\[
V_j\left(x_{i,a}+\frac{\hbar}{2}\right),
\qquad
j\neq\tau i,\quad i\to j.
\]
If $i\to\tau i$, the crossings with the remaining $i$-labelled strands
give
\[
\prod_{\substack{1\leq k\leq v_i\\k\neq a}}
\overline X_{\tau i,i}
(x_{\tau i,a},x_{i,k})
=
(-1)^{v_i-1}
V_{\tau i,a}
\left(x_{i,a}+\frac{\hbar}{2}\right).
\]
Thus the total sign arising from the middle crossings is
\(
(-1)^{\epsilon_i-v_i}.
\)

By the polynomial action in \cref{thm:tRPol}, the two
reflections contribute
\[
p_{\tau i}(x_{\tau i,a})d_{i,a}.
\]
Moreover, we have
\[
d_{i,a}
\left(x_{i,a}-\frac{\hbar}{2}\right)^r
=
\left(x_{i,a}+\frac{\hbar}{2}\right)^r d_{i,a}.
\]
The crossings between equally labelled $i$-strands in the upper part
of $D_i$ give the divided differences
$\partial_{i,1}\cdots\partial_{i,v_i-1}$. Therefore,
\begin{align}
\label{eq:Di-action}
D_i
\left(
y_{1+\sum_{j<i}v_j}-\frac{\hbar}{2}
\right)^r e_\Sigma
=
(-1)^{\epsilon_i-v_i}
\partial_{i,1}\cdots\partial_{i,v_i-1}
\left(
\left(x_{i,v_i}+\frac{\hbar}{2}\right)^r
G_{i,v_i}d_{i,v_i}
\right)e_\Sigma.
\end{align}

We recall a standard nilHecke identity
\begin{equation}
\label{eq:nilHecke-sum}
\partial_{i,1}\cdots\partial_{i,v_i-1}
\left(F_{v_i}d_{i,v_i}\right)e_\Sigma
=
(-1)^{v_i-1}
\sum_{a=1}^{v_i}
\frac{F_a}{V_{i,a}(x_{i,a})}
d_{i,a}e_\Sigma,
\end{equation}
where $(F_1,\ldots,F_{v_i})$ is equivariant under the natural
$\Sigma_{v_i}$-action; see \cite[Lem.~4.7]{KTWWY19}. Applying
\cref{eq:nilHecke-sum} to \cref{eq:Di-action}, we obtain
\begin{align*}
D_i
\left(
y_{1+\sum_{j<i}v_j}-\frac{\hbar}{2}
\right)^r e_\Sigma
=
(-1)^{\epsilon_i+v_i-1}
\sum_{a=1}^{v_i}
\left(x_{i,a}+\frac{\hbar}{2}\right)^r
\frac{G_{i,a}}{V_{i,a}(x_{i,a})}
d_{i,a}e_\Sigma.
\end{align*}
Consequently,
\begin{align*}
(-1)^{
r+\epsilon_i
}
D_i
\left(
y_{1+\sum_{j<i}v_j}-\frac{\hbar}{2}
\right)^r e_\Sigma
=
(-1)^{r+1}
\sum_{a=1}^{v_i}
\left(x_{i,a}+\frac{\hbar}{2}\right)^r
\frac{G_{i,a}}{V_{i,a}(x_{i,a})}
d_{i,a}e_\Sigma,
\end{align*}
since
\[
(r+\epsilon_i)+(\epsilon_i-v_i+v_i-1)
\equiv r+1\pmod 2.
\]
The right-hand side is exactly the coefficient of $z^{-r-1}$ in
\cref{biGKLO}. This proves \cref{biimage}.

We now prove \cref{btiimage}. Applying \cref{biGKLO} with the vertex
$\tau i$, all products over outgoing arrows are empty, since
$\tau i\in Q_0^{(-)}$. Using
\[
x_{\tau i,a}=-x_{i,a},
\qquad
d_{\tau i,a}=d_{i,a}^{-1},
\qquad
V_{\tau i,a}(x_{\tau i,a})
=
(-1)^{v_i-1}V_{i,a}(x_{i,a}),
\]
we obtain
\begin{equation}
\label{eq:bti-explicit}
b_{\tau i,r}
=
(-1)^{v_i}
\sum_{a=1}^{v_i}
\left(x_{i,a}-\frac{\hbar}{2}\right)^r
\frac{p_i(x_{i,a})}{V_{i,a}(x_{i,a})}
d_{i,a}^{-1}.
\end{equation}

On the other hand, in the diagram $D_{\tau i}$ the crossings in the
middle segment contribute no polynomial factors. Indeed, they are of
the form $\overline X_{j,\tau i}$ with $j\in Q_0^{(+)}$, and there are
no arrows $\tau i\to j$. The two reflections contribute
$p_i(x_{i,a})d_{i,a}^{-1}$. Furthermore,
\[
d_{i,a}^{-1}
\left(x_{i,a}+\frac{\hbar}{2}\right)^r
=
\left(x_{i,a}-\frac{\hbar}{2}\right)^r
d_{i,a}^{-1}.
\]
The same nilHecke calculation as above gives
\[
D_{\tau i}
\left(
y_{\sum_{j\leq i}v_j}+\frac{\hbar}{2}
\right)^r e_\Sigma
=
(-1)^{v_i-1}
\sum_{a=1}^{v_i}
\left(x_{i,a}-\frac{\hbar}{2}\right)^r
\frac{p_i(x_{i,a})}{V_{i,a}(x_{i,a})}
d_{i,a}^{-1}e_\Sigma.
\]
Comparing this with \cref{eq:bti-explicit} yields
\[
b_{\tau i,r}
=
-
D_{\tau i}
\left(
y_{\sum_{j\leq i}v_j}+\frac{\hbar}{2}
\right)^r e_\Sigma,
\]
which proves \cref{btiimage}.

Finally, by \cref{aiGKLO}, the coefficients of the series $a_i(z)$ act
by multiplication by elements of $P^\Sigma$, and hence belong to
\(
\NH\subset e(\aleph)\tR e(\aleph).
\)
Together with \cref{biimage,btiimage}, this shows that the images of
the generators of $\trust$ lie in $e(\aleph)\tR e(\aleph)$. This concludes the proof.
\end{proof}

\begin{remark}
We note that, in \cite[Theorem~4.22]{KTWWY19}, the analogue of this inclusion was shown to be an equality for quiver gauge theories by passing through the geometry of Coulomb branches.  
In our setting, however, such an approach is not currently available, since it remains open whether the iGKLO representation in \cref{prop:iGKLO} is full.  
Instead, in order to study weight modules for the flag truncated shifted iYangian, we later realize it as a subalgebra of $e(\aleph)\tR e(\aleph)$ with explicit generators in \cref{thm:flag=ab}.
\end{remark}

\subsection{Flag truncated shifted iYangian}

Motivated by \cite[\S 4.3]{KTWWY19}, we define the flag truncated shifted iYangian
\[
{}^\tau FY^\lambda_\mu = {}^\tau FY^\lambda_\mu(\bR)
\]
to be the subalgebra of $\End(P)$ generated by the divided difference operators, multiplication by elements of $P$, and the image of ${}^\tau Y^\lambda_\mu(\bR)$.  

\begin{theorem}
\label{thm:morita}
    There is a Morita equivalence between $\ftrust$ and $\trust$.
\end{theorem}

\begin{proof}
Recall the idempotent $e_{\Sigma}$ from \cref{symidempotent}. By definition, we have
\(
e_{\Sigma} \in \NH \subset \ftrust.
\)
To prove that $\ftrust$ and $\trust$ are Morita equivalent, it suffices to show that
\(
e_{\Sigma} \ftrust e_{\Sigma} = \trust
\)
and that $e_{\Sigma}$ is a full idempotent in $\ftrust$.

By definition and \cref{prop:iGKLO}, it is clear that
\(
\trust \subset e_{\Sigma} \ftrust e_{\Sigma}.
\)
Conversely, every element of $\ftrust$ is a linear combination of terms of the form
\(
K_0 L_1 K_1 \cdots L_l K_l,
\)
where $K_0,\dots,K_l \in \NH$ and $L_1,\dots,L_l \in \trust$.
Here $K_0$ and $K_l$ may be taken to be $1$, since \(1\in \NH\).
Therefore,
\[
e_{\Sigma} (K_0L_1K_1\cdots L_l K_l)e_{\Sigma}
=
(e_{\Sigma} K_0 e_{\Sigma})L_1(e_{\Sigma} K_1 e_{\Sigma})\cdots L_l(e_{\Sigma} K_l e_{\Sigma}).
\] Now regard these operators as acting on $P^\Sigma$. For every $K\in\NH$, the corner $e_\Sigma K e_\Sigma$ acts on
$P^\Sigma$ by multiplication by an element of $P^\Sigma$.
Since $P^\Sigma\subset\trust$, it follows that $e_{\Sigma} \ftrust e_{\Sigma}\subset \trust$.

It remains to show that $1\in \ftrust e_{\Sigma} \ftrust$. By \cite[Lemma 2.12]{We19}, we see that $1\in \NH (\prod_{i\in Q_0^{(+)}} \partial_{i,w_0}) \NH$. Therefore, we must have
\[
1\in \NH (\prod_{i\in Q_0^{(+)}} \partial_{i,w_0}) \NH =\NH  (\prod_{i\in Q_0^{(+)}} \partial_{i,w_0}) e_{\Sigma} \NH\subset \NH e_{\Sigma} \NH \subset \ftrust e_{\Sigma} \ftrust
\]
since $\prod_{i\in Q_0^{(+)}} \partial_{i,w_0}= (\prod_{i\in Q_0^{(+)}} \partial_{i,w_0})e_{\Sigma} $. This concludes the proof.
\end{proof}

There are three variants of the quantized Coulomb branch algebra introduced in \cite{BFN18}.  
Roughly speaking, the truncated shifted iYangians $\trust$ correspond to the spherical version, while the flag truncated shifted iYangians $\ftrust$ correspond to a flag version of the quantized Coulomb branch algebra.  

In \cite{We19}, Weekes showed that, for quiver gauge theories, the flag version is generated by certain nilHecke algebras together with another variation of the quantized Coulomb branch algebra (denoted by $\mathcal{A}^{\mathrm{ab}}$ \textit{loc.\ cit.}), which admits an explicit set of generators.  
Motivated by this, in the remainder of this section we introduce, at the Yangian level, an analogue of this variation for the gauge theory associated with a quiver with involution, and show that $\ftrust$ is generated by this subalgebra together with $\NH$.

Recall the KLR iYangian $\tR$ from \cref{def:Ya}.  
For any word $\bi=(i_1,\ldots,i_n)\in \word$, let $\tR(\bi)$ denote the subalgebra of $\tR$ spanned by double reflective KLR diagrams whose top and bottom labels lie in $\objRB{\alpha_{\bi}}$. We also define $r_k\in \tR(\bi)$ for $1\leq k<n$ by
\begin{equation}
\label{eq:intertw}
r_k e(\bj)=
\begin{cases}
\bigl((y_{k+1}-y_k)\sigma_k+1\bigr)e(\bj),
& \bj_k=\bj_{k+1},\\
\sigma_k e(\bj),
& \bj_k\neq\bj_{k+1}.
\end{cases}
\end{equation}
They are called the {\bf intertwiners}. It is well known that $r_k$ satisfy the braid relations. Thus for any $w\in \Sigma$, $r_w$ is well-defined via a reduced expression of $w$.

Recall the ordered word $\aleph$ from \cref{sec:diagiGKLO}.  
In particular, we identify the variables $x_{i,k}$ $(i\in Q_0^{(+)},\,1\le k\le v_i)$ with the variables
\(
X_1(\aleph),\ldots,X_m(\aleph).
\)
For convenience, we often identify the pair $(i,k)$ $(i\in Q_0^{(+)},\,1\le k\le v_i)$ with the number \begin{equation}
 \label{eq:pair}
 v_1+v_2+\cdots +v_{i-1}+k.
\end{equation} 

Recall for each $i\in Q_0^{(+)}$ the elements $D_i$ and $D_{\tau i}$ from \cref{Di}. Then we define $\gamma_i$ as follows:
\begin{gather}
    \label{gammai}
    \gamma_i=
    \begin{tikzpicture}
        [anchorbase,scale=1.4]
        \mirror{-.1}{0}{1};
        \draw[-] (.125,0) -- (.125,1)  \toplabel{}
                 (.5,0) -- (.5,1)
                 (.825,0) -- (.825,1)
                 (1.2,0) -- (1.2,1)
                 (1.5,0) -- (1.5,1)
                 (1.9,0) -- (1.9,1);
        \rmirror{2.125}{0}{1};
        \draw[-] (0.625,0) -- (-.1,1/3) -- (2.125,2/3) -- (1.375,1);
        \node at (.3125,.5){$\scriptstyle \cdots$};
         \node at (1.0125,.5){$\scriptstyle \cdots$};
          \node at (1.7,.5){$\scriptstyle \cdots$};
           \draw [decorate,decoration={brace,amplitude=5pt}]  (.5,-.05) -- (.125,-.05); \node at (.625/2,-.35){$\scriptstyle <i$};
           \draw [decorate,decoration={brace,amplitude=5pt}]  (1.375,-.05) -- (.625,-.05); \node at (2.025/2,-.35){$\scriptstyle i$};
           \draw [decorate,decoration={brace,amplitude=5pt}]  (1.9,-.05) -- (1.5,-.05); \node at (1.7,-.35){$\scriptstyle >i$};
    \end{tikzpicture}
,\qquad
    \gamma_{\tau i}:=
     \begin{tikzpicture}
        [anchorbase,scale=1.4]
        \mirror{-.1}{0}{1};
        \draw[-] (.125,0) -- (.125,1)  \toplabel{}
                 (.5,0) -- (.5,1)
                 (.825,0) -- (.825,1)
                 (1.2,0) -- (1.2,1)
                 (1.5,0) -- (1.5,1)
                 (1.9,0) -- (1.9,1);
        \rmirror{2.125}{0}{1};
        \draw[-] (1.375,0) -- (2.125,1/3) -- (-.1,2/3) -- (0.625,1);
        \node at (.3125,.5){$\scriptstyle \cdots$};
         \node at (1.0125,.5){$\scriptstyle \cdots$};
          \node at (1.7,.5){$\scriptstyle \cdots$};
           \draw [decorate,decoration={brace,amplitude=5pt}]  (.5,-.05) -- (.125,-.05); \node at (.625/2,-.35){$\scriptstyle <i$};
           \draw [decorate,decoration={brace,amplitude=5pt}]  (1.375,-.05) -- (.625,-.05); \node at (2.025/2,-.35){$\scriptstyle i$};
           \draw [decorate,decoration={brace,amplitude=5pt}]  (1.9,-.05) -- (1.5,-.05); \node at (1.7,-.35){$\scriptstyle >i$};
    \end{tikzpicture}
\end{gather}

Now for each $i\in Q_0^{(+)}$ and $1\leq k\leq v_i$, we define the following elements in $e(\aleph)\tR e(\aleph)$:
\begin{gather}
    \label{Gammai}
    \Gamma_{i,k}:=r_{i,k}\circ \cdots \circ r_{i,v_i-1}\circ \gamma_i \circ r_{i,1}\circ \cdots \circ r_{i,k-1},\\
    \label{Gammati}
    \Gamma_{\tau i,k}=r_{i,k-1}\circ \cdots \circ r_{i,1}\circ \gamma_{\tau i} \circ r_{i,v_i-1}\circ \cdots \circ r_{i,k},
\end{gather}
Here we have used the convention in \cref{eq:pair}. 
By \cref{Gammai,Gammati}, it is straightforward to check that 
\[
\Gamma_{i,k+1}=r_{i,k}\circ \Gamma_{i,k}\circ r_{i,k}, \qquad \Gamma_{\tau i,k+1}=r_{i,k}\circ \Gamma_{\tau i,k}\circ r_{i,k},
\]
for each $i\in Q_0^{(+)}$ and $1\leq k<v_i$. 

Intuitively, $\Gamma_{i,1}$ is obtained from $D_i$ by replacing each crossing between strands labelled by $i$ with the corresponding intertwiner, which acts on $\Pol$ simply by permuting the variables. Let $\abtrust$ denote the subalgebra of $e(\aleph)\tR e(\aleph)$ generated by the elements $\Gamma_{i,k}$ $(i\in Q_0,\;1\le k\le v_i)$ together with $\NH$.

\begin{theorem}
\label{thm:flag=ab}
    The flag truncated shifted iYangian $\ftrust$ is equal to $\abtrust$.
\end{theorem}

\begin{proof}
 Since $r_{i,k}\in \NH$ acts trivally on $P^\Sigma=e_{\Sigma} P$ for every $i\in Q_0^{(+)}$ and $1\le k<v_i$, it follows that $\abtrust$ is precisely the subalgebra of $e(\aleph)\tR e(\aleph)$ generated by $\gamma_i$ $(i\in Q_0)$ together with $\NH$.  
Therefore, \cref{biimage,btiimage} implies that $\ftrust\subset \abtrust$.

On the other hand, to show that $\abtrust\subset \ftrust$, it suffices to show that $\gamma_i ,\gamma_{\tau i} \in \ftrust$ for all $i\in Q_0^{(+)}$. We prove for $\gamma_i$ and the case for $\gamma_{\tau i}$ is similar. By \cref{biimage}, up to a sign we have
\begin{align}
\label{bread}
    b_{i,r}=&\partial_{i,1}\circ \cdots \circ \partial_{i,v_i-1}(x_{i,v_i}+\frac{\hbar}{2})^r \gamma_i e_{\Sigma}
    =\sum_{k=1}^{v_i}
\frac{\left(x_{i,k}+\frac{\hbar}{2}\right)^r}
{\prod_{\substack{1\le j\le v_i\\ j\ne k}}(x_{i,k}-x_{i,j})} \gamma_i e_{\Sigma}
\end{align}
when acting on $P$. Now we set 
\[F(z):=\prod_{j=2}^{v_i}(z-x_{i,j}-\frac{\hbar}{2})=\sum_{t=0}^{v_i-1}f_t(x_{i,2},\ldots,x_{i,v_i}) z^t\]
for some polynomials $f_t$. Then up to a sign we have
\begin{multline*}
    \sum_{t=0}^{v_i-1} f_t(x_{i,2},\ldots,x_{i,v_i})b_{i,t} \overset{\cref{bread}}{=}  \sum_{t=0}^{v_i-1} f_t(x_{i,2},\ldots,x_{i,v_i}) \sum_{k=1}^{v_i}
\frac{\left(x_{i,k}+\frac{\hbar}{2}\right)^t}
{\prod_{\substack{1\le j\le v_i\\ j\ne k}}(x_{i,k}-x_{i,j})} \gamma_i e_{\Sigma} \\
= \sum_{k=1}^{v_i} F(x_{i,k}+\frac{\hbar}{2})\frac{1}{\prod_{\substack{1\le j\le v_i\\ j\ne k}}(x_{i,k}-x_{i,j})}  \gamma_i e_{\Sigma} =\gamma_i e_{\Sigma}
\end{multline*}
when acting on $P$. Thus we have shown that $\gamma_i e_\Sigma\in \ftrust$.

We now remove the factor $e_\Sigma$. Since $e_\Sigma$ is a full idempotent in
$\NH$, and $P$ is finite free over $P^\Sigma$, there exist elements
$p_a,q_a\in P$ such that
\[
    1=\sum_a p_a e_\Sigma q_a
    \qquad \text{in } \NH .
\]
Moreover, by the definition of $\gamma_i$ in terms of intertwiners and mirror
reflections, $\gamma_i$ normalizes the polynomial algebra $P$: there is an
automorphism $\phi_i$ of $P$ such that
\[
    \gamma_i p=\phi_i(p)\gamma_i
    \qquad \text{for all }p\in P .
\]
Therefore
\[
    \gamma_i
    =
    \gamma_i\cdot 1
    =
    \sum_a \gamma_i p_a e_\Sigma q_a
    =
    \sum_a \phi_i(p_a)\gamma_i e_\Sigma q_a .
\]
Since $P\subset \NH\subset \ftrust$ and $\gamma_i e_\Sigma\in \ftrust$, the
right-hand side belongs to $\ftrust$. Hence $\gamma_i\in\ftrust$. This concludes the proof.
\end{proof}

Combining \cref{thm:morita} and \cref{thm:flag=ab}, we immediately obtain the following:
\begin{corollary}
    $\abtrust$ and $\trust$ are Morita equivalent.
\end{corollary}

\section{Weight modules and interval oKLRW algebras}
In this section, we study integral weight modules for the (flag)
truncated shifted iYangians and establish an equivalence with
nilpotent modules over a variant of the oKLRW algebra at
\(\hbar=1\). 
To this end, we also assume that
\begin{itemize}
    \item every element of $\bR_i$ lies in $\frac{1}{4}+\Z$ for each $i\in Q_0^{(+)}$;
    \item every element of $\bR_{\tau i}$ lies in $\frac{3}{4}+\Z$ for each $i\in Q_0^{(+)}$.
\end{itemize}
We say that the parameters $\bR$ are \textbf{integral} if these two conditions are satisfied.

\newcommand{\Spec}{\text{Spec}}

\subsection{Weight modules for (flag) truncated shifted iYangians}
Recall the Cartan series $a_i(z)$ from \cref{Aseries}.  
In \cite[(3.11)]{LWW2}, the \emph{Gelfand--Tsetlin subalgebra} of $\trust$ is defined in terms of certain coefficients of the series $a_i(z)$.  
In particular, by \cite[Proposition~3.12 and Remark~3.13]{LWW2}, it follows that $\trust$ contains the polynomial subalgebra
\[
P^\Sigma=\C\bigl[a_i^{(s)} \mid i\in Q_0^{(+)},\ 1\le s\le v_i\bigr].
\]

We write $\Spec P^\Sigma$ for the set of collections $\bS=(S_i)_{i\in Q_0^{(+)}}$ of multisets such that $|S_i|=v_i$ for each $i\in Q_0^{(+)}$.  
Any such collection $\bS$ determines an evaluation map
\[
\ev_{\bS}\colon P^\Sigma \to \C,
\]
which sends $a_i^{(s)}$ to the $s$-th elementary symmetric function in the multiset $S_i$.

On the other hand, the algebra $\ftrust$ contains the polynomial subalgebra
\[
P=\C\bigl[x_{i,k}\mid i\in Q_0^{(+)},\ 1\le k\le v_i\bigr].
\]
We also write
\(
\Spec P := \prod_{i\in Q_0^{(+)}} \C^{v_i}.
\)
An element $\bnu=(\nu_i)_{i\in Q_0^{(+)}}\in \Spec P$, where $\nu_i=(\nu_{i,1},\dots,\nu_{i,v_i})$, gives rise to an evaluation map
\[
\ev_{\bnu}\colon P\to \C,
\]
sending $x_{i,k}$ to $\nu_{i,k}$ for all $i\in Q_0^{(+)}$ and $1\le k\le v_i$.

\begin{definition}\label{def:wtmod}
A {\bf weight module} over $\trust$ (resp.\ $\ftrust$) is a module $M$
admitting a decomposition
\[
M=\bigoplus_{\bS\in\Spec P^\Sigma}W_{\bS}(M)
\qquad
\left(
\text{resp. }
M=\bigoplus_{\bnu\in\Spec P}W_{\bnu}(M)
\right)
\]
into finite-dimensional generalized eigenspaces, where
\[
W_{\bS}(M)
=
\bigl\{
v\in M
\;\big|\;
\exists\,N\gg 0\text{ such that }
(f-\ev_{\bS}(f))^Nv=0
\text{ for all }f\in P^\Sigma
\bigr\},
\]
and, respectively,
\[
W_{\bnu}(M)
=
\bigl\{
v\in M
\;\big|\;
\exists\,N\gg 0\text{ such that }
(f-\ev_{\bnu}(f))^Nv=0
\text{ for all }f\in P
\bigr\}.
\]
\end{definition}

We call $W_{\bnu}(M)$ the $\bnu$-weight space of $M$, and refer to the
assignment
\[
M\longmapsto W_{\bnu}(M)
\]
as the $\bnu$-weight functor. Similarly, we call $W_{\bS}(M)$ the
$\bS$-weight space of $M$, and refer to the assignment
\[
M\longmapsto W_{\bS}(M)
\]
as the $\bS$-weight functor.

As pointed out in \cite[Sect. 5.1]{KTWWY19}, the Morita equivalence of \cref{thm:morita} identifies weight modules over $\ftrust$ with weight modules over $\trust$.  
There is a natural map
\[
\Spec P \to \Spec P^\Sigma
\]
sending a tuple $\ba=(a_i)$ to the collection of multisets
\(
\bS(\ba)=(S_i(\ba)),
\)
where $S_i(\ba)$ is the multiset underlying the tuple $a_i$.

Let $M$ be a weight module over $\ftrust$, and let $M^\Sigma$ denote its image under the Morita equivalence.  
The group $\Sigma$ acts naturally on 
\(
\Spec P=\prod_{i\in Q_0^{(+)}} \C^{v_i},
\)
and hence one may consider the stabilizer $\text{Stab}_\Sigma(\ba)$ of a weight $\ba$.  
Then we have (cf. \cite[(5.1)]{KTWWY19})
\begin{equation}
\label{eq:weights-Y}
W_{\bS(\ba)}(M^\Sigma)=W_{\ba}(M)^{\text{Stab}_\Sigma(\ba)}.
\end{equation}
We say that $\ba \in \Spec P$ is an \textbf{integral weight} if, for every $i \in Q_0^{(+)}$ and $1 \le k \le v_i$, we have 
\begin{gather}
\label{integral}
a_{i,k}\in \frac{1}{4}+\Z
\end{gather}
(Recall that we have identified the pair $(i,k)$ with the number \(
 v_1+v_2+\cdots +v_{i-1}+k.
\) )
For any weight $\ba$, we extend the notation by setting
\[
a_{\tau i,k}:=-a_{i,k}\qquad \text{ for any }i\in Q_0^{(+)},1\leq k \leq v_i=v_{\tau i}.
\]

\begin{remark}
The restriction to integral weights is imposed only for simplicity. A more
general treatment can be organized according to the orbits of the spectral
parameters, in a way analogous to the discussion of integrality classes in
\cite[\S 5.6]{KTWWY19}. We note, however, that the conventions are different:
the normalization in \cite{KTWWY19} corresponds to setting \(\hbar=2\),
whereas in the present paper we set \(\hbar=1\).

In our conventions, the translation orbit of a spectral parameter \(a\in\C\)
is
\[
   \mathcal O_a
   :=
   a+\frac{\hbar}{2}\Z
   =
   a+\frac12\Z.
\]
The integrality condition in \eqref{integral} together with the convention
$a_{\tau i,k}=-a_{i,k}$ select the orbit
\[
   \mathcal O_{\frac14}
   =
   \frac14+\frac12\Z.
\]
This orbit is preserved by the involution \(a\mapsto -a\), since
\[
   -\frac14+\frac12\Z
   =
   \frac14+\frac12\Z.
\]
Moreover, it decomposes into two distinct cosets modulo \(\Z\):
\[
   \frac14+\frac12\Z
   =
   \left(\frac14+\Z\right)
   \sqcup
   \left(-\frac14+\Z\right).
\]
The two copies \(Q_0^{(+)}\) and \(Q_0^{(-)}\) in our construction encode
these two cosets, and the involution \(\tau\) exchanges them. More generally, the involution sends
\[
   \mathcal O_a
   =
   a+\frac12\Z
   \qquad\text{to}\qquad
   \mathcal O_{-a}
   =
   -a+\frac12\Z.
\]
If \(\mathcal O_a=\mathcal O_{-a}\), then a single self-dual iquiver component
is sufficient. This is the reason for restricting to
\(\mathcal O_{\frac14}\) in the present paper. If
\(\mathcal O_a\neq\mathcal O_{-a}\), then one must consider a pair of quiver
components indexed by \(\mathcal O_a\) and \(\mathcal O_{-a}\), exchanged by
the induced involution. For an arbitrary collection of spectral parameters,
one should take the disjoint union of the components associated with the
relevant orbits.
\end{remark}

A weight module $M$ over $\ftrust$ is said to have \textbf{integral weights} if
\(
W_{\bnu}(M)\neq 0
\)
implies that $\bnu$ is an integral weight.  
We write $\fywtmod$ to denote the category of finitely generated $\ftrust$-weight modules with integral weights.

\subsection{Interval oKLRW algebras}
Recall the set of integral parameters $\bR=(\bR_i)_{i\in Q_0}$, and suppose that
\[
\bR_i=\{w_{i,j}\mid 1\le j\le \lambda_i\}
\qquad (i\in Q_0),
\]
where
\(
\lambda=\sum_{i\in Q_0}\lambda_i\varpi_i.
\)

From $\lambda$ and the integral parameters $\bR$, we construct a sequence
of dominant weights
\begin{equation}
\label{eq:lamseq}
\lamseq:=\lamseq(\bR)=(\rlam^{(0)},\rlam^{(1)},\rlam^{(2)},\ldots)
\end{equation}
as follows. For \(i\in Q_0^{(+)}\) and \(n\in\mathbb Z\), define
\[
z_{i,n}
:=
\#\left\{
c\in\bR_{\tau i}
\ \middle|\
c=\frac34-n
\right\},
\qquad
z_{\tau i,n}
:=
\#\left\{
c\in\bR_i
\ \middle|\
c=\frac14-n
\right\}.
\]
We then define
\[
\rlam^{(2n)}
:=
\sum_{i\in Q_0^{(+)}}
\left(
z_{i,n}\varpi_i
+
z_{\tau i,n}\varpi_{\tau i}
\right)
\qquad
(n\geq0),
\]
and
\[
\rlam^{(1)}
:=
\sum_{i\in Q_0^{(+)}}
z_{\tau i,-1}\varpi_i.
\]
For \(n\geq1\), we also define
\[
\rlam^{(2n+1)}
:=
\sum_{i\in Q_0^{(+)}}
\left(
z_{\tau i,-n-1}\varpi_i
+
z_{i,-n}\varpi_{\tau i}
\right).
\]
\newcommand{\eint}{e_{\text{int}}}

We note that $\lamseq$ contains only finitely many nonzero dominant weights, determined by the choice of $\bR$. Recall that we have fixed a total order on $Q_0^{(+)}$, which induces an order on $Q_0^{(-)}$ by requiring that
\[
i>j \quad \text{if and only if} \quad \tau i < \tau j.
\]

\begin{definition}\label{Omega}
Let $\Omega$ be the set of sequences 
\[ \nu=(\bi^{(1)},\bi^{(2)},\ldots) \] 
such that each $\bi^{(r)}$ is a finite word in the alphabet $Q_0$, all but finitely many of the words $\bi^{(r)}$ are empty, and the following conditions hold: 
\begin{enumerate} 
\item for each $n\geq 0$, the word $\bi^{(2n+1)}$ consists of vertices in $Q_0^{(+)}$ arranged in increasing order; 
\item for each $n\geq 0$, the word $\bi^{(2n+2)}$ consists of vertices in $Q_0^{(-)}$ arranged in increasing order; 
\item for every $i\in Q_0^{(+)}$, the total number of letters equal to either $i$ or $\tau i$ in all the words $\bi^{(r)}$ is equal to $v_i$. \end{enumerate} 
\end{definition}
We regard $\nu$ as an interval configuration: the word $\bi^{(r)}=(\bi^{(r)}_1,\bi^{(r)}_2,\ldots,\bi^{(r)}_{\ell_r})$ records
the black strands lying between the red strands labeled by $\rlam^{(r-1)}$
and $\rlam^{(r)}$. We also let $e(\nu)$ denote the idempotent represented
diagrammatically by the identity diagram on $\nu$, consisting of vertical
black strands together with vertical red strands labeled, from left to right,
by the terms of $\lamseq=(\rlam^{(0)},\rlam^{(1)},\rlam^{(2)},\ldots)$.
Since all but finitely many words $\bi^{(r)}$ are empty, this diagram is
trivial outside a sufficiently large finite set of intervals.

Before defining the interval oKLRW algebra, we first extend the definition of
oKLRW algebras to allow infinite sequences
\[
    \lamseq=(\rlam^{(0)},\rlam^{(1)},\rlam^{(2)},\ldots)
\]
where only finitely many $\rlam^{(r)}$ are nonzero. For a fixed $\lamseq$ and
$N\in \mathbb N$, we let
\[
    \lamseq_{\leq N}
    =
    (\rlam^{(0)},\rlam^{(1)},\ldots,\rlam^{(N)}).
\]

Choose \(N_0\) such that
\(\rlam^{(r)}=0\) for all \(r>N_0\). For every \(N\geq N_0\), there
is a natural algebra embedding
\[
    \oklrwzero^{\lamseq_{\leq N}}
    \hookrightarrow
    \oklrwzero^{\lamseq_{\leq N+1}}
\]
given diagrammatically by adding one vertical red strand labeled by
\(\rlam^{(N+1)}=0\) at the far right of every diagram. Since red
strands labeled by \(0\) do not contribute to the local red--black
relations, these maps are compatible with multiplication.

We define
\[
    \oklrwzero^{\lamseq}
    :=
    \varinjlim_{N\geq N_0}
    \oklrwzero^{\lamseq_{\leq N}}.
\]
where the direct limit is taken with respect to the embeddings above. Thus
\(\oklrwzero^{\lamseq}\) is a locally unital algebra. (For a general discussion about locally unital algebras and their basic properties, see \cite[Sect. 2.1]{BD17}.) 

Equivalently, every diagram appearing in $\oklrwzero^{\lamseq}$ is required to be trivial outside a sufficiently large finite set of intervals. Since $\rlam^{(r)}=0$ for all sufficiently large $r$, adding or removing empty zero-labeled red strands at the far right does not change the diagrammatic relations. Now we let $\eint$ denote the formal sum 
\begin{gather}
\label{def:eint}
\eint:=\sum_{\nu\in \Omega}e(\nu).
\end{gather}
Then the interval oKLRW algebra is defined to be the locally unital algebra
\begin{gather}
\label{def:intklrwalg}
    \intklrw
    := \eint \cdot \oklrwzero^{\lamseq} \cdot \eint=
    \bigoplus_{\nu,\nu'\in\Omega}
    e(\nu')\cdot \oklrwzero^{\lamseq}\cdot e(\nu).
\end{gather}

To describe the generators of the interval oKLRW algebra $\intklrw$, we define the following elements diagrammatically in $\intklrw$. 

Now fix a $\nu\in \Omega$. For any $a,b,l,l'\geq 1$ such that $a-b\equiv 1 \bmod 2$, $i=\bi^{(a)}_l$ and $\tau i=(\bi')^{(b)}_{l'}$, we define {\em long reflections} between intervals of opposite parity
\begin{gather}
\label{Longreflect}
    \rho_{a,l}^{b,l'}(\nu,i):= \begin{cases}
        \begin{tikzpicture}
        [anchorbase,scale=1.2]
        \mirror{0}{0}{1};
        \node at (.5,.5){$\cdots$};
         \node at (3.2,.5){$\cdots$};
           \node at (5.9,.5){$\cdots$};
           \node at (1.2,.5){$\scriptstyle \cdots$};
           \node at (3.8,.5){$\scriptstyle \cdots$};
        \node at (1.7,.5){$\scriptstyle \cdots$};
        \node at (2.6,.5){$\scriptstyle \cdots$};
        \node at (4.25,.5){$\scriptstyle \cdots$};
        \node at (5.2,.5){$\scriptstyle \cdots$};
        \draw[red] (1,0)\botlabel{\rlam^{\textcolor{black}{(a-1)}}} -- (1,1);
        \draw[red] (2.8,0)\botlabel{\rlam^{\textcolor{black}{(a)}}} -- (2.8,1);
        \draw[red] (3.6,0)\botlabel{\rlam^{\textcolor{black}{(b-1)}}} -- (3.6,1);
         \draw[red] (5.4,0)\botlabel{\rlam^{\textcolor{black}{(b)}}} -- (5.4,1);
        \draw[-] (1.4,0) -- (1.4,1)  \toplabel{}
                 (1.9,0) -- (1.9,1)
                 (2.4,0) -- (2.4,1);
        \draw[-] (4,0) -- (4,1)  \toplabel{}
                 (4.5,0) -- (4.5,1)
                 (5,0) -- (5,1);
        \draw[-] (2.2,0) \botlabel{i} -- (0,.35) -- (4.75,1) \toplabel{\tau i};
    \end{tikzpicture}, & a<b;\\
       \begin{tikzpicture}
        [anchorbase,scale=1.2]
        \mirror{0}{0}{1};
        \node at (.5,.5){$\cdots$};
         \node at (3.2,.5){$\cdots$};
           \node at (5.9,.5){$\cdots$};
           \node at (1.2,.5){$\scriptstyle \cdots$};
           \node at (3.8,.5){$\scriptstyle \cdots$};
        \node at (1.7,.5){$\scriptstyle \cdots$};
        \node at (2.6,.5){$\scriptstyle \cdots$};
        \node at (4.25,.5){$\scriptstyle \cdots$};
        \node at (5.2,.5){$\scriptstyle \cdots$};
        \draw[red] (1,0)\botlabel{\rlam^{\textcolor{black}{(b-1)}}} -- (1,1);
        \draw[red] (2.8,0)\botlabel{\rlam^{\textcolor{black}{(b)}}} -- (2.8,1);
        \draw[red] (3.6,0)\botlabel{\rlam^{\textcolor{black}{(a-1)}}} -- (3.6,1);
         \draw[red] (5.4,0)\botlabel{\rlam^{\textcolor{black}{(a)}}} -- (5.4,1);
        \draw[-] (1.4,0) -- (1.4,1)  \toplabel{}
                 (1.9,0) -- (1.9,1)
                 (2.4,0) -- (2.4,1);
        \draw[-] (4,0) -- (4,1)  \toplabel{}
                 (4.5,0) -- (4.5,1)
                 (5,0) -- (5,1);
        \draw[-] (4.75,0) \botlabel{i} -- (0,1-.35) -- (2.2,1) \toplabel{\tau i};
    \end{tikzpicture}, & a>b;
    \end{cases}
\end{gather}

In both cases of \cref{Longreflect}, the bottom of the diagram is labeled by $\nu$.  
Suppose the top is labeled by $(\bj,\redkappa_2)$, then the unique non-vertical strand is the one joining the occurrence \(i=\bi^{(a)}_l\) to the occurrence \(\tau i=\bj^{(b)}_{l'}\), which touches the mirror exactly once. By \cref{def:intklrwalg}, the indices of the intervals containing the two endpoints differ by an odd number.


For any $a,b,l,l' \geq 1$ such that $a-b\equiv 0\bmod 2$ and $i=\bi^{(a)}_l={\bi'}^{(b)}_{l'}$, we also define  {\em long crossings} between intervals of same parity
\begin{align}
    \label{Longcross}
    C_{a,l}^{b,l'}(\nu,i):=\begin{cases}
\begin{tikzpicture}
        [anchorbase]
        \mirror{0}{0}{1};
        \node at (.5,.5){$\cdots$};
         \node at (3.2,.5){$\cdots$};
           \node at (5.9,.5){$\cdots$};
           \node at (1.2,.5){$\scriptstyle \cdots$};
           \node at (3.8,.5){$\scriptstyle \cdots$};
        \node at (1.7,.5){$\scriptstyle \cdots$};
        \node at (2.6,.5){$\scriptstyle \cdots$};
        \node at (4.5,.5){$\scriptstyle \cdots$};
        \node at (5.2,.5){$\scriptstyle \cdots$};
        \draw[red] (1,0)\botlabel{\rlam^{\textcolor{black}{(a-1)}}} -- (1,1);
        \draw[red] (2.8,0)\botlabel{\rlam^{\textcolor{black}{(a)}}} -- (2.8,1);
        \draw[red] (3.6,0)\botlabel{\rlam^{(b-1)}} -- (3.6,1);
         \draw[red] (5.4,0)\botlabel{\rlam^{\textcolor{black}{(b)}}} -- (5.4,1);
        \draw[-] (1.4,0) -- (1.4,1)  \toplabel{}
                 (1.9,0) -- (1.9,1)
                 (2.4,0) -- (2.4,1);
        \draw[-] (4.75,0) -- (4.75,1)  \toplabel{}
                 (4.25,0) -- (4.25,1)
                 (5,0) -- (5,1);
        \draw[-] (2.2,0) \botlabel{i} -- (4,1)\toplabel{i};
    \end{tikzpicture}, & a<b;\\
\begin{tikzpicture}
        [anchorbase,scale=1.2]
        \mirror{0}{0}{1};
        \node at (.5,.5){$\cdots$};
         \node at (3.2,.5){$\cdots$};
           \node at (5.9,.5){$\cdots$};
           \node at (1.2,.5){$\scriptstyle \cdots$};
           \node at (3.8,.5){$\scriptstyle \cdots$};
        \node at (1.7,.5){$\scriptstyle \cdots$};
        \node at (2.6,.5){$\scriptstyle \cdots$};
        \node at (4.5,.5){$\scriptstyle \cdots$};
        \node at (5.2,.5){$\scriptstyle \cdots$};
        \draw[red] (1,0)\botlabel{\rlam^{\textcolor{black}{(b-1)}}} -- (1,1);
        \draw[red] (2.8,0)\botlabel{\rlam^{\textcolor{black}{(b)}}} -- (2.8,1);
        \draw[red] (3.6,0)\botlabel{\rlam^{\textcolor{black}{(a-1)}}} -- (3.6,1);
         \draw[red] (5.4,0)\botlabel{\rlam^{\textcolor{black}{(a)}}} -- (5.4,1);
        \draw[-] (1.4,0) -- (1.4,1)  \toplabel{}
                 (1.9,0) -- (1.9,1)
                 (2.4,0) -- (2.4,1);
        \draw[-] (4.75,0) -- (4.75,1)  \toplabel{}
                 (4.25,0) -- (4.25,1)
                 (5,0) -- (5,1);
        \draw[-] (4,0) -- (2.2,1)\toplabel{i};
    \end{tikzpicture}, & a>b.
    \end{cases}
\end{align}

In both cases of \cref{Longcross}, the bottom of the diagram is labeled by $\nu$.  
Suppose that the top is labeled by $(\bj,\redkappa_2)$, then the unique non-vertical strand in \cref{Longcross} joins the occurrence \(i=\bi^{(a)}_l\) to the occurrence \(i=\bj^{(b)}_{l'}\). By \cref{def:intklrwalg}, the two intervals containing these occurrences must differ by an even number.

Besides, when $\bi^{(j)}_k=\bi^{(j)}_{k+1}$, we define a crossing that fixes $\nu$ by
\begin{gather}
    \label{adcross}
    \sigma_j(\nu,k):=
    \begin{tikzpicture}
        [anchorbase,scale=1.2]
        \mirror{0}{0}{1};
        \node at (.5,.5){$\cdots$};
         \node at (3,.5){$\cdots$};
         \node at (1.3,.5){$\scriptstyle \cdots$};
        \node at (2.2,.5){$\scriptstyle \cdots$};
        \draw[red] (1,0)\botlabel{\rlam^{\textcolor{black}{(j-1)}}} -- (1,1);
        \draw[red] (2.6,0)\botlabel{\rlam^{\textcolor{black}{(j)}}} -- (2.6,1);
        \draw[-] (1.5,0) -- (2,1)  \toplabel{\bi_k^{(j)}}
                 (2,0)  \botlabel{\bi_{k+1}^{(j)}} -- (1.5,1);
    \end{tikzpicture}
    \qquad 1\leq k\leq \ell_j-1.
\end{gather}
We refer to such crossings as \emph{admissible crossings}; cf. \cite[Section~5.4]{KWWY14}.

At last, we also define the dot on the strand labeled by $\bi_k^{(j)}$ as
\begin{gather}
    \label{Longdot}
    Y_j(\nu,k):=
    \begin{tikzpicture}
        [anchorbase,scale=1.2]
        \mirror{0}{0}{1};
        \node at (.5,.5){$\cdots$};
         \node at (3,.5){$\cdots$};
         \node at (1.3,.5){$\scriptstyle \cdots$};
        \node at (2.2,.5){$\scriptstyle \cdots$};
        \draw[red] (1,0)\botlabel{\rlam^{\textcolor{black}{(j-1)}}} -- (1,1);
        \draw[red] (2.6,0)\botlabel{\rlam^{\textcolor{black}{(j)}}} -- (2.6,1);
        \draw[-] (1.75,0)\botlabel{\bi_k^{(j)}} -- (1.75,1) \toplabel{}  ;
        \bdot{1.75,.5};
    \end{tikzpicture}
    \qquad
    1\leq k\leq \ell_j.
\end{gather}

\begin{prop}
\label{prop:intgen}
The interval oKLRW algebra $\intklrw$ is the subalgebra of $\oklrwzero$ generated by the idempotents $e(\nu)$ and the diagrams in \cref{Longreflect}--\cref{Longdot}, where $\nu$ ranges over $\Omega$ and $i,a,b,j,k$ range over all admissible values.
\end{prop}

\begin{proof}
Let \(A\) be the locally unital subalgebra of
\(\oklrwzero^{\lamseq}\) generated by the idempotents \(e(\nu)\), for
\(\nu\in\Omega\), together with the diagrams in
\cref{Longreflect,Longcross,adcross,Longdot}. Since the source and
target of each of these generators belong to \(\Omega\), we have
\[
    A
    \subseteq
    \eint\oklrwzero^{\lamseq}\eint
    =
    \intklrw.
\]

We now prove the reverse inclusion. By \cref{lem:spanpart}, for every
\(\nu,\nu'\in\Omega\), the space
\[
    e(\nu')\oklrwzero^{\lamseq}e(\nu)
\]
is spanned by reduced Stendhal diagrams from \(\nu\) to \(\nu'\),
with arbitrary dots placed at the bottom. Since the bottom dots are
generated by the elements \(Y_j(\nu,k)\) in \cref{Longdot}, it is
enough to show that every reduced dot-free Stendhal diagram from
\(\nu\) to \(\nu'\) belongs to \(A\).

Fix such a diagram \(D\). We argue by induction on the number of
non-vertical black strands of \(D\). If every black strand is
vertical, then the source and target differ only by permutations of
equal strands inside individual subintervals. Such permutations are
generated by the admissible crossings in \cref{adcross}, and hence
\(D\in A\).

Suppose now that \(D\) contains a non-vertical black strand. Choose a
black strand whose top endpoint is leftmost among the top endpoints
of all non-vertical black strands. Suppose that its bottom endpoint
is the \(l\)-th occurrence in the \(a\)-th subinterval and its top
endpoint is the \(l'\)-th occurrence in the \(b\)-th subinterval.

Assume first that this strand does not meet the mirror. Its label is
therefore unchanged. Since odd subintervals contain vertices of
\(Q_0^{(+)}\) and even subintervals contain vertices of
\(Q_0^{(-)}\), the integers \(a\) and \(b\) have the same parity.
The movement of this strand from its bottom position to its top
position is represented by a long crossing
\(
    C_{a,l}^{b,l'}(\widetilde\nu,i)
\)
for a suitable intermediate object \(\widetilde\nu\in\Omega\).

Assume instead that the strand meets the mirror. Since \(D\) is
reduced, it meets the mirror exactly once, and its label changes from
\(i\) to \(\tau i\), or from \(\tau i\) to \(i\). Thus \(a\) and
\(b\) have opposite parity. The movement of this strand is represented
by a long reflection
\(
    \rho_{a,l}^{b,l'}(\widetilde\nu,i)
\)
for a suitable intermediate object \(\widetilde\nu\in\Omega\).

In either case, the remaining strands form a reduced dot-free
Stendhal diagram \(D'\) with one fewer non-vertical black strand.
The diagram \(D\) is obtained from \(D'\) by composing with the
corresponding long crossing or long reflection, together with
admissible crossings between equal adjacent strands when necessary.

The intermediate objects remain in \(\Omega\). Indeed, removing a
letter from an increasing word and inserting it into the unique
position preserving the prescribed order again gives an increasing
word. For even subintervals, this uses the order on \(Q_0^{(-)}\)
induced by
\[
    i>j
    \quad\Longleftrightarrow\quad
    \tau i<\tau j.
\]
Any ambiguity caused by equal adjacent labels can be resolved by the
admissible crossings in \cref{adcross}.

By the induction hypothesis, \(D'\in A\). The corresponding long
crossing, long reflection, and admissible crossings also belong to
\(A\). Hence \(D\in A\).

It follows that every reduced dot-free Stendhal diagram between two
objects of \(\Omega\) belongs to \(A\). Adding arbitrary bottom dots
does not leave \(A\), so \cref{lem:spanpart} gives
\[
    e(\nu')\oklrwzero^{\lamseq}e(\nu)
    \subseteq
    A
\]
for all \(\nu,\nu'\in\Omega\). Therefore
\[
    \intklrw
    =
    \bigoplus_{\nu,\nu'\in\Omega}
    e(\nu')\oklrwzero^{\lamseq}e(\nu)
    \subseteq
    A.
\]
Thus \(A=\intklrw\).
\end{proof}

\subsection{Parity polynomial representations}
Recall that we defined a polynomial representation of oKLRW algebras in
\cref{prop:polyrep}. In this subsection, we introduce a parity polynomial
representation of \(\oklrwzero^{\lamseq}\) for later use.

Recall from \cref{eq:Pol,Polobj} that
\[
\Pol^\lamseq
=
\bigoplus_{(\bi,\redkappa)\in\obj{\klrw}^{\lamseq}}
\Pol(\bi,\redkappa,\lamseq).
\]
We also define
\[
\Polint
:=
\bigoplus_{\bi\in\Omega}
\Pol(\bi,\redkappa,\lamseq)
\]
(Here the information of $\redkappa$ and $\lamseq$ are already encoded in elements of $\Omega$).

\begin{definition}
    We define the parity polynomial representation
    \begin{equation}
        \label{eq:paritypol}
        \ctP^{\mathrm{par}}: \oklrwzero^{\lamseq} \to \End(\Pol^\lamseq)
    \end{equation}
    where the actions of the black dots, the black--black
crossings, and the reflections at the mirror are defined to be the same as $\ctP$,
(see \cref{Polklrw1,Polklrw2,Polreflect} respectively). For the red--black crossings, suppose that the
\(k\)-th black strand, labeled by \(i\), crosses the red strand labeled
by \(\rlam^{(r)}\). Then we define
\begin{equation}
    \label{paritycross}
\begin{aligned}
\ctP^{\mathrm{par}}
\Big(
\begin{tikzpicture}[baseline=.6,scale=1.2]
    \draw (0,0)[red,thick]
        \botlabel{\rlam^{\textcolor{black}{(r)}}}
        \braidup (.4,.4);
    \draw (.4,0)
        \botlabel{i}
        \braidup (0,.4);
\end{tikzpicture}
\Big)(f)
&=
\begin{cases}
Y_k^{\rlam^{(r)}_i}f,
&
r\equiv0\pmod 2,
\\[0.4em]
f,
&
r\equiv1\pmod 2,
\end{cases}
\\[1em]
\ctP^{\mathrm{par}}
\Big(
\begin{tikzpicture}[baseline=.6,scale=1.2]
    \draw (0,0)
        \botlabel{i}
        \braidup (.4,.4);
    \draw (.4,0)[red,thick]
        \botlabel{\rlam^{\textcolor{black}{(r)}}}
        \braidup (0,.4);
\end{tikzpicture}
\Big)(f)
&=
\begin{cases}
f,
&
r\equiv0\pmod 2,
\\[0.4em]
Y_k^{\rlam^{(r)}_i}f,
&
r\equiv1\pmod 2.
\end{cases}
\end{aligned}
\end{equation}
(All other vertical strands are omitted from the red--black crossing diagrams above.)
\end{definition}

Thus, the even red strands act according to the usual polynomial
representation as in \cref{prop:polyrep}, whereas the odd red strands act according to the
opposite convention. It is straightforward to check that $\ctP^{\mathrm{par}}$ defines an action of $\oklrwzero^\lamseq$ on $\Pol^\lamseq$. Indeed, the actions of all generators other than the red--black crossings are unchanged, so it suffices to verify the relations involving a red--black crossing. We observe that \cref{redbraid1}--\cref{reddoublecross} can all be verified locally with only one red strand.  It follows from similar arguments as in \cref{cor:faithfulpolyrep} that $\ctP^{\mathrm{par}}$ is faithful as well. Restricting to the subalgebra $\intklrw$, this also defines a faithful action of $\intklrw$ on $\Polint$.

Recall that, when the shift operators are moved to the left, the
framing coefficients in the polynomial actions of
\(\Gamma_{i,k}\) and \(\Gamma_{\tau i,k}\) are, respectively,
\[
p_{\tau i}(1-x_{i,k})
=
\prod_{c\in\bR_{\tau i}}
(1-x_{i,k}-c)
\qquad\text{
and
}\qquad
p_i(x_{i,k}+1)
=
\prod_{c\in\bR_i}
(x_{i,k}+1-c).
\]
We refer to the individual factors
\[
1-x_{i,k}-c,
\qquad
c\in\bR_{\tau i},
\qquad\text{
and
}\qquad
x_{i,k}+1-c,
\qquad
c\in\bR_i,
\]
as the \emph{framing factors} of \(\Gamma_{i,k}\) and
\(\Gamma_{\tau i,k}\), respectively.

\begin{remark}
The parity polynomial representation \(\ctP^{\mathrm{par}}\) is
essential to the relation between the truncated shifted twisted
Yangian and the interval oKLRW algebra. As will be seen in
\cref{prop:Theta,prop:Gamma}, the operators \(\Gamma_{i,k}\) and
\(\Gamma_{\tau i,k}\), which move the corresponding black strand in
opposite directions, both involve the framing factors defined above.
Thus, in order to realize the iYangian action diagrammatically
and, conversely, to recover the oKLRW generators from this action,
powers of the black dots must arise from red--black crossings in both
directions.

This differs from the untwisted setting of \cite[(4.2b)--(4.2c)]{KTWWY19},
where the framing polynomial occurs only in the action of \(F_i\), and
not in that of \(E_i\). Consequently, the framing factors can be
realized by red--black crossings in a single direction, so no
parity-dependent modification of the polynomial representation is
needed there.

In the twisted setting, the two parities of red strands encode the
framing factors occurring in \(\Gamma_{i,k}\) and
\(\Gamma_{\tau i,k}\), respectively. This feature will be crucial for
the equivalence in \cref{thm:equivalence}.
\end{remark}

\subsection{An equivalence relating $\ftrust$ and $\intklrw$}
Recall that $\ftrust$ is Morita equivalent to $\trust$ by \cref{thm:morita} and we have shown in \cref{thm:flag=ab} that $\ftrust$ can be generated by $\Gamma_{i,k}$ $(i\in Q_0,1\leq k \leq v_i)$ together with $\NH$ as a subalgebra of $e(\aleph)\tR e(\aleph)$.

Recall that  $\aleph$ is the sequence we defined by listing the vertices of $Q_0^{(+)}$ according to the chosen total order (from smallest to largest), with each vertex $i$ appearing exactly $v_i$ times and $m=\sum_{i\in Q_0^{(+)}}v_i=\sum_{i\in Q_0^{(-)}}v_{i}$. Suppose $\ba$ is an integral weight as in \cref{integral}. We define $\nu_\ba$ to be the unique sequence such that:
\begin{itemize}
\item $\nu_\ba\in \Omega$, i.e.  $e(\nu_\ba)\in \intklrw$;
 \item 
the multiplicity of $i\in Q_0^{(+)}$ between $\rlam^{(2n)}$ and $\rlam^{(2n+1)}$ is exactly 
the number of $1\leq k\leq v_i$ such that $a_{i,k}=n+\frac{1}{4}$ for all $n\geq 0$; 
 \item 
the multiplicity of $\tau i\in Q_0^{(-)}$ between $\rlam^{(2n+1)}$ and $\rlam^{(2n+2)}$ is exactly 
the number of $1\leq k\leq v_i$ such that $a_{i,k}=-n-\frac{3}4$ for all $n\geq 0$.
\end{itemize}
This defines a many-to-one map
\[
\{\text{integral weights } \ba\} \longrightarrow \{ \text{idempotents $\nu_\ba$ for } \intklrw \}.
\]
Moreover, this construction assigns to each pair \((i,k)\), with \(i\in Q_0^{(+)}\) and \(1\le k\le v_i\), a pair \((u_\ba(i,k),v_\ba(i,k))\) as follows:
\begin{itemize}
    \item We set
\[
u_{\ba}(i,k)=
\begin{cases}
2n+1, & \text{if } a_{i,k}=n+\frac{1}{4},\\
2n+2, & \text{if } a_{i,k}=-n-\frac{3}{4},
\end{cases}
\] for some \(n\geq 0\). This records the interval.
\item We define \(v_{\ba}(i,k)\) to be the number of pairs \((j,l)\) satisfying
\(
a_{j,l}=a_{i,k}
\)
and occurring no later than \((i,k)\) in the lexicographic order induced by the chosen total order on \(Q_0^{(+)}\). That is,
\[
v_{\ba}(i,k)
:=
\begin{cases}
\displaystyle
\#\Bigl\{(j,l)\ \Big|\
j<i,\ 1\leq l\leq v_j,\
a_{j,l}=a_{i,k}
\Bigr\}
+
\#\Bigl\{1\leq l\leq k\ \Big|\
a_{i,l}=a_{i,k}
\Bigr\},
&
a_{i,k}>0,
\\[1.2em]
\displaystyle
\#\Bigl\{(j,l)\ \Big|\
j>i,\ 1\leq l\leq v_j,\
a_{j,l}=a_{i,k}
\Bigr\}
+
\#\Bigl\{k\leq l\leq v_i\ \Big|\
a_{i,l}=a_{i,k}
\Bigr\},
&
a_{i,k}<0.
\end{cases}
\]
This records the position inside the interval.
\end{itemize}

\begin{example}
    Suppose we have $Q_0^{(+)}=\{i,j,k\}$ with $i<j<k$ and hence $\tau k<\tau j<\tau i$. Moreover, suppose that $v_i=3$ and $v_j=v_k=4$. Hence we have
    \[
    \aleph=(i,i,i,j,j,j,j,k,k,k,k).
    \]
    Suppose we have an integral weight
    \[
    \ba=\Big(\overbrace{\frac{1}{4},-\frac{3}{4},\frac{1}{4}}^i,\overbrace{-\frac{15}{4},-\frac{3}{4},-\frac{15}{4},\frac{9}{4}}^j,\overbrace{\frac{13}{4},\frac{9}{4},-\frac{7}{4},\frac{1}{4}}^k\Big)
    \]
    Thus we have
    \begin{gather*}
          u_\ba(i,1)=1,u_\ba(i,2)=2, u_\ba(i,3)=1,\\ v_\ba(i,1)=1,v_\ba(i,2)=2, v_\ba(i,3)=2,\\
            u_\ba(j,1)=8, u_\ba(j,2)=2, u_\ba(j,3)=8, u_\ba(j,4)=5, \\
            v_\ba(j,1)=2, v_\ba(j,2)=1, v_\ba(j,3)=1, v_\ba(j,4)=1, \\
            u_\ba(k,1)=7, u_\ba(k,2)=5, u_\ba(k,3)=4, u_\ba(k,4)=1, \\
            v_\ba(k,1)=1, v_\ba(k,2)=2, v_\ba(k,3)=1, v_\ba(k,4)=3,
    \end{gather*}
    Thus the corresponding idempotent of $\intklrw$ is 
    \[
    e(\nu_\ba)=
       \begin{tikzpicture}
        [anchorbase,scale=1.4]
        \mirror{0}{0}{1};
        \draw[red] (.3,0)\botlabel{} -- (.3,1)\toplabel{\rlam^{\textcolor{black}{(0)}}};
        \draw[-] (.7,0)\botlabel{i} -- (.7,1)
                 (1,0)\botlabel{i} -- (1,1)
                  (1.3,0)\botlabel{k} -- (1.3,1);
        \draw[red] (1.7,0)\botlabel{} -- (1.7,1)\toplabel{\rlam^{\textcolor{black}{(1)}}};
        \draw[-] (2.1,0)\botlabel{\tau j} -- (2.1,1)
                 (2.4,0)\botlabel{\tau i} -- (2.4,1);
        \draw[red] (2.8,0)\botlabel{} -- (2.8,1)\toplabel{\rlam^{\textcolor{black}{(2)}}};
        \draw[red] (3.3,0)\botlabel{} -- (3.3,1)\toplabel{\rlam^{\textcolor{black}{(3)}}};
        \draw[-] (3.7,0)\botlabel{\tau k} -- (3.7,1);
        \draw[red] (4.1,0)\botlabel{} -- (4.1,1)\toplabel{\rlam^{\textcolor{black}{(4)}}};
         \draw[-] (4.5,0)\botlabel{j} -- (4.5,1)
                 (4.8,0)\botlabel{k} -- (4.8,1);
        \draw[red] (5.2,0)\botlabel{} -- (5.2,1)\toplabel{\rlam^{\textcolor{black}{(5)}}};
        \draw[red] (5.7,0)\botlabel{} -- (5.7,1)\toplabel{\rlam^{\textcolor{black}{(6)}}};
        \draw[-] (6.1,0)\botlabel{k} -- (6.1,1);
        \draw[red] (6.5,0)\botlabel{} -- (6.5,1)\toplabel{\rlam^{\textcolor{black}{(7)}}};
        \draw[-] (6.9,0)\botlabel{\tau j} -- (6.9,1);
        \draw[-] (7.2,0)\botlabel{\tau j} -- (7.2,1);
        \draw[red] (7.6,0)\botlabel{} -- (7.6,1)\toplabel{\rlam^{\textcolor{black}{(8)}}};
        \draw[red] (8.1,0)\botlabel{} -- (8.1,1)\toplabel{\rlam^{\textcolor{black}{(9)}}};
        \node at (8.6,.5){$\cdots$};
    \end{tikzpicture}
    \]
\end{example}

Let $\intmodnil$ denote the category of finitely generated $\intklrw$-modules on which all dot generators act nilpotently.  
Recall from \cref{prop:intgen} that $\intklrw$ is generated by the idempotents $e(\nu)$, for $\nu\in \Omega$, together with the diagrams in \cref{Longreflect}--\cref{Longdot}.  
Accordingly, if $M\in \intmodnil$, then each generator of $\intklrw$ defines an operator on the corresponding idempotent summands of $M$; by abuse of notation, we denote these operators by the same symbols.  
For instance, each $Y_j(\nu,k)$ from \cref{Longdot} induces an operator
\begin{gather}
    Y_j(\nu,k)\colon e(\nu)M \to e(\nu)M
\end{gather}
given by the action of a dot on the strand labeled by $\bi^{(j)}_k$ while each admissible crossing $\sigma_j(\nu,k)$ induces an operator
\begin{gather}
    \sigma_j(\nu,k)\colon e(\nu)M \to e(\nu)M
\end{gather}
given by the action of the crossing of the black strands labeled by $\bi_{k}^{(j)}$ and $\bi_{k+1}^{(j)}$.

\begin{prop}\label{prop:Theta}
There is a functor  $\Theta \colon \intmodnil\to \fywtmod$, such that for any integral weight $\ba$ and $M\in \intmodnil$, we have
  \begin{equation}
  \label{eq:Theta}
W_{\ba}(\Theta(M)) =
e(\nu_\ba) M.
\end{equation}
\end{prop}

\begin{proof}
    The strategy of the proof is inspired by the one of \cite[Lem. 5.17]{KTWWY19}. Let $M\in \intmodnil$, we first define a vector space by 
    \[
   \Theta(M):=
\bigoplus_{\ba\ \mathrm{integral}}
e(\nu_\ba)M\delta_\ba.
    \]
    where $\delta_\ba$ is a formal symbol, this formal symbol is introduced to distinguish between different integral weights $\ba$ which give rise to the same idempotent $e(\nu_\ba)$.

    By \cref{thm:flag=ab}, the algebra $\ftrust$ is generated by \cref{Gammai,Gammati} together with the nil-Hecke algebra $\NH$. In order to define an action of $\ftrust$ on $\Theta(M)$, we define the following operators (for any $v\in M$):
    \begin{itemize}
        \item For $y_{i,k}=y_{i,k}(\aleph)\in \NH$ where $i\in Q_0^{(+)}, 1\leq k \leq v_i$, it acts by 
       \begin{equation}
\label{taco1}
y_{i,k}v\delta_\ba
=
\begin{cases}
\left(
Y_{u_\ba(i,k)}
\bigl(\nu_\ba,v_\ba(i,k)\bigr)
+a_{i,k}
\right)v\delta_\ba,
& a_{i,k}>0,
\\[1.2em]
\left(
-Y_{u_\ba(i,k)}
\bigl(\nu_\ba,v_\ba(i,k)\bigr)
+a_{i,k}
\right)v\delta_\ba,
& a_{i,k}<0.
\end{cases}
\end{equation}
        Note that this has the correct eigenvalue, since
  $Y_{u_{\ba}(i,k)}(\nu_\ba,v_{\ba}(i,k))$ is nilpotent. Indeed, when \(a_{i,k}>0\), the corresponding strand is labeled by \(i\),
so its dot records the nilpotent part \(x_{i,k}-a_{i,k}\). When
\(a_{i,k}<0\), the corresponding strand is labeled by \(\tau i\), and hence its dot records
\(
x_{\tau i,k}-a_{\tau i,k}
=
-\bigl(x_{i,k}-a_{i,k}\bigr).
\)
  \item For $\sigma_{i,k}=\sigma_{i,k}(\aleph)\in \NH$, we must have $\aleph_{i,k}=\aleph_{i,k+1}$ and 
\begin{equation}
\label{taco2}
\sigma_{i,k}(v\delta_\ba)
=
\begin{cases}
\displaystyle
\sigma_{u_\ba(i,k)}
\bigl(\nu_\ba,v_\ba(i,k)\bigr)v\delta_\ba,
&a_{i,k}=a_{i,k+1}>0,
\\[1em]
\displaystyle
\sigma_{u_\ba(i,k)}
\bigl(\nu_\ba,v_\ba(i,k+1)\bigr)v\delta_\ba,
&a_{i,k}=a_{i,k+1}<0,
\\[1em]
\begin{aligned}
\left(
\widetilde Y^\ba_{i,k+1}
-
\widetilde Y^\ba_{i,k}
+
a_{i,k}-a_{i,k+1}
\right)^{-1}
v\delta_\ba
\\
+
\left(
\widetilde Y^\ba_{i,k}
-
\widetilde Y^\ba_{i,k+1}
+
a_{i,k}-a_{i,k+1}
\right)^{-1}
v\delta_{s_{i,k}\ba},
\end{aligned}
&a_{i,k}\neq a_{i,k+1},
\end{cases}
\end{equation}
where $\widetilde Y^\ba_{i,r}:=\sgn(a_{i,r})Y_{u_{\ba}(i,r)}(\nu_\ba,v_{\ba}(i,r))$. When \(a_{i,k}=a_{i,k+1}\), the two coordinates lie in the same
subinterval. Our ordering convention ensures that the corresponding
admissible crossing agrees with the nilHecke divided difference
\(\sigma_{i,k}\). When \(a_{i,k}\neq a_{i,k+1}\), each inverse in
\cref{taco2} is taken on the generalized weight space indicated by the
formal symbol to its right. Its constant term is
\(\pm(a_{i,k}-a_{i,k+1})\neq0\), while its remaining part lies in the maximal ideal of the completion.
Hence the inverse is well-defined as a formal power series.
  \item  For fixed $i\in Q_0^{(+)},1\leq k\leq v_i$, we define $\ba^\pm$ by setting
  \[
  a^{\pm}_{j,l}=a_{j,l}\pm\delta_{i,j}\delta_{k,l}
  \]
  Then $\Gamma_{i,k}$ acts by (recall \cref{Longreflect,Longcross})
  \begin{gather}
\label{taco3}
\Gamma_{i,k}(v\delta_\ba)
=
\begin{cases}
\displaystyle
(-1)^{
\#\{c\in\bR_{\tau i}\mid c=1-a_{i,k}\}
}
C_{u_\ba(i,k),v_\ba(i,k)}
 ^{u_{\ba^-}(i,k),v_{\ba^-}(i,k)}
(\nu_\ba,i)
\\
\displaystyle\hspace{1.5cm}\cdot
\left(
\prod_{\substack{c\in\bR_{\tau i}\\
c\neq1-a_{i,k}}}
\left(
1-a_{i,k}-c
-
Y_{u_\ba(i,k)}
\bigl(\nu_\ba,v_\ba(i,k)\bigr)
\right)v
\right)
\delta_{\ba^-},
&
a_{i,k}>\dfrac14,
\\[2em]
\displaystyle
\rho_{u_\ba(i,k),v_\ba(i,k)}
 ^{u_{\ba^-}(i,k),v_{\ba^-}(i,k)}
(\nu_\ba,i)
\\
\displaystyle\hspace{1.5cm}\cdot
\left(
\prod_{\substack{c\in\bR_{\tau i}\\
c\neq\frac34}}
\left(
\frac34-c
-
Y_{u_\ba(i,k)}
\bigl(\nu_\ba,v_\ba(i,k)\bigr)
\right)v
\right)
\delta_{\ba^-},
&
a_{i,k}=\dfrac14,
\\[2em]
\displaystyle
C_{u_\ba(i,k),v_\ba(i,k)}
 ^{u_{\ba^-}(i,k),v_{\ba^-}(i,k)}
(\nu_\ba,\tau i)
\\
\displaystyle\hspace{1.5cm}\cdot
\left(
\prod_{\substack{c\in\bR_{\tau i}\\
c\neq1-a_{i,k}}}
\left(
1-a_{i,k}-c
+
Y_{u_\ba(i,k)}
\bigl(\nu_\ba,v_\ba(i,k)\bigr)
\right)v
\right)
\delta_{\ba^-},
&
a_{i,k}<0.
\end{cases}
\\[2em]
\label{taco4}
\Gamma_{\tau i,k}(v\delta_\ba)
=
\begin{cases}
\displaystyle
(-1)^{
\#\{c\in\bR_i\mid c=a_{i,k}+1\}
}
C_{u_\ba(i,k),v_\ba(i,k)}
 ^{u_{\ba^+}(i,k),v_{\ba^+}(i,k)}
(\nu_\ba,\tau i)
\\
\displaystyle\hspace{1.5cm}\cdot
\left(
\prod_{\substack{c\in\bR_i\\
c\neq a_{i,k}+1}}
\left(
a_{i,k}+1-c
-
Y_{u_\ba(i,k)}
\bigl(\nu_\ba,v_\ba(i,k)\bigr)
\right)v
\right)
\delta_{\ba^+},
&
a_{i,k}<-\dfrac34,
\\[2em]
\displaystyle
\rho_{u_\ba(i,k),v_\ba(i,k)}
 ^{u_{\ba^+}(i,k),v_{\ba^+}(i,k)}
(\nu_\ba,\tau i)
\\
\displaystyle\hspace{1.5cm}\cdot
\left(
\prod_{\substack{c\in\bR_i\\
c\neq\frac14}}
\left(
\frac14-c
-
Y_{u_\ba(i,k)}
\bigl(\nu_\ba,v_\ba(i,k)\bigr)
\right)v
\right)
\delta_{\ba^+},
&
a_{i,k}=-\dfrac34,
\\[2em]
\displaystyle
C_{u_\ba(i,k),v_\ba(i,k)}
 ^{u_{\ba^+}(i,k),v_{\ba^+}(i,k)}
(\nu_\ba,i)
\\
\displaystyle\hspace{1.5cm}\cdot
\left(
\prod_{\substack{c\in\bR_i\\
c\neq a_{i,k}+1}}
\left(
a_{i,k}+1-c
+
Y_{u_\ba(i,k)}
\bigl(\nu_\ba,v_\ba(i,k)\bigr)
\right)v
\right)
\delta_{\ba^+},
&
a_{i,k}>0.
\end{cases}
\end{gather}
In \cref{taco3,taco4}, the operator \(\Gamma_{i,k}\) changes the
 coordinate from \(a_{i,k}\) to \(a_{i,k}-1\), whereas
\(\Gamma_{\tau i,k}\) changes it from \(a_{i,k}\) to
\(a_{i,k}+1\). If the initial and final coordinates have the same
sign, the corresponding black strand remains on the same side of the
mirror, and the operator is represented by a long crossing. If their
signs are different, the strand meets the mirror once, and the
operator is represented by a long reflection.

The black--black crossings and the reflection at the mirror reproduce
the black--black factors in the polynomial actions of
\(\Gamma_{i,k}\) and \(\Gamma_{\tau i,k}\). The red--black crossings
in the corresponding long diagram account for the framing factors
whose constant terms vanish on the relevant generalized weight space through the parity polynomial representation \cref{eq:paritypol}.

Consequently, in each case of \cref{taco3,taco4}, all framing factors
with zero constant term are already incorporated into the
corresponding long crossing or long reflection. The products displayed
in \cref{taco3,taco4} contain exactly the remaining framing factors.
Each of these factors has nonzero constant term on the relevant
generalized weight space and is therefore invertible in its
completion. This invertibility will be used in the construction of
the inverse functor in \cref{prop:Gamma}.
    \end{itemize}

It remains to verify that the operators defined above give a
well-defined action of \(\ftrust\). For an integral weight \(\ba\), let
\[
\widehat P_{\ba}
:=
\varprojlim_N
P/
\bigl(x_{i,k}-a_{i,k}\mid
i\in Q_0^{(+)},\ 1\leq k\leq v_i\bigr)^N
\]
be the completion of \(P\) at the maximal ideal corresponding to
\(\ba\). Similarly, for \(\nu\in\Omega\), let
\(\widehat{\Pol}(\nu)\) denote the completion of the
\(\nu\)-summand of the parity polynomial representation of
\(\intklrw\) at the ideal generated by all its dot variables.

For every integral weight \(\ba\), the change of variables
\[
x_{i,k}-a_{i,k}
\longmapsto
\begin{cases}
\displaystyle
Y_{u_\ba(i,k)}
\bigl(\nu_\ba,v_\ba(i,k)\bigr),
& a_{i,k}>0,
\\[1em]
\displaystyle
-
Y_{u_\ba(i,k)}
\bigl(\nu_\ba,v_\ba(i,k)\bigr),
& a_{i,k}<0,
\end{cases}
\]
defines an isomorphism
\[
\Phi_\ba\colon
\widehat P_\ba
\xrightarrow{\sim}
\widehat{\Pol}(\nu_\ba).
\]
Taking the direct sum over all integral weights gives an isomorphism
\[
\Phi
:=
\bigoplus_{\ba\ \mathrm{integral}}\Phi_\ba
\colon
\bigoplus_{\ba\ \mathrm{integral}}\widehat P_\ba
\xrightarrow{\sim}
\bigoplus_{\ba\ \mathrm{integral}}
\widehat{\Pol}(\nu_\ba)\delta_\ba.
\]

We now compare the polynomial actions under \(\Phi\).
Multiplication by \(x_{i,k}\) is carried to the operator in
\cref{taco1}. If \(a_{i,k}=a_{i,k+1}\), the ordering convention in the
corresponding subinterval identifies the divided difference
\(\sigma_{i,k}\) with the admissible crossing appearing in
\cref{taco2}. If \(a_{i,k}\neq a_{i,k+1}\), the usual divided
difference formula becomes the second case of \cref{taco2}. The
displayed denominator is invertible in the completion, since its
constant term is \(a_{i,k}-a_{i,k+1}\neq0\).

It remains to compare the operators \(\Gamma_{i,k}\) and
\(\Gamma_{\tau i,k}\). By their polynomial actions computed in
\cref{sec:diagiGKLO}, their shift parts change the \(k\)-th coordinate
by
\[
a_{i,k}\longmapsto a_{i,k}-1
\qquad\text{and}\qquad
a_{i,k}\longmapsto a_{i,k}+1,
\]
respectively. Moreover, the relations
\[
p_{\tau i}(-x_{i,k})d_{i,k}
=
d_{i,k}p_{\tau i}(1-x_{i,k}),
\qquad
p_i(x_{i,k})d_{i,k}^{-1}
=
d_{i,k}^{-1}p_i(x_{i,k}+1)
\]
express their framing coefficients in terms of the source variables.

Recall that the polynomial actions of
\(\Gamma_{i,k}\) and \(\Gamma_{\tau i,k}\) in
\cref{sec:diagiGKLO} were obtained by composing the local polynomial
actions of the black--black crossings and, when the strand meets the
mirror, the reflection occurring in the corresponding diagram.
The same local black--black crossings and reflection occur in the
long crossings and long reflections in \cref{taco3,taco4}.
Thus they reproduce the part of the polynomial action which is
independent of the framing parameters. It remains only to compare the
framing factors with the red--black crossings.

We explain this for \(\Gamma_{i,k}\); the argument for
\(\Gamma_{\tau i,k}\) is analogous. Set \(a=a_{i,k}\). A framing
factor
\[
1-x_{i,k}-c,
\qquad c\in\bR_{\tau i},
\]
has zero constant term on the generalized \(\ba\)-weight space
precisely when \(c=1-a\). For example, suppose that
\(a=n+\frac14>\frac14\). Then
\[
1-a=\frac34-n,
\]
so the multiplicity of \(c=1-a\) in \(\bR_{\tau i}\) is \(z_{i,n}\),
which, by \cref{eq:lamseq}, is the coefficient of \(\varpi_i\) in
\(\rlam^{(2n)}\). As the \(i\)-strand moves from the subinterval
corresponding to \(a\) to that corresponding to \(a-1\), it crosses
two red strands; by \cref{paritycross}, only its right-to-left
crossing with the even red strand \(\rlam^{(2n)}\) contributes, and
this contribution is
\(
Y_{u_\ba(i,k)}
\bigl(\nu_\ba,v_\ba(i,k)\bigr)^{z_{i,n}}.
\)
Under \(\Phi_\ba\), \(x_{i,k}-a=Y_{u_\ba(i,k)}
\bigl(\nu_\ba,v_\ba(i,k)\bigr)\), and hence the product of the
framing factors with \(c=1-a\) is
\(
\left(
-
Y_{u_\ba(i,k)}
\bigl(\nu_\ba,v_\ba(i,k)\bigr)
\right)^{z_{i,n}}.
\)
This is exactly the dot contribution above together with the sign
\(
(-1)^{\#\{c\in\bR_{\tau i}\mid c=1-a\}}
\)
appearing in the first case of \cref{taco3}. The cases
\(a=\frac14\) and \(a<0\) are checked in the same way, using the
corresponding long reflection and the left-to-right crossing with an
odd red strand, respectively. Thus all framing factors with zero
constant term are already contained in the corresponding long
crossing or long reflection. The products displayed explicitly in
\cref{taco3,taco4} are exactly the remaining framing factors.
Therefore, after applying \(\Phi\), the polynomial actions of
\(\Gamma_{i,k}\) and \(\Gamma_{\tau i,k}\) agree with the operators
defined in \cref{taco3,taco4}.

Consequently, conjugation by \(\Phi\) identifies the polynomial action
of every generator of \(\ftrust\) with the operator defined above on
the completed polynomial representation of \(\intklrw\). Since the parity
polynomial representation of \(\intklrw\) is faithful, its dot-adic
completion is faithful as well. It follows that all relations
satisfied by the generators of \(\ftrust\) are satisfied by the
operators in \cref{taco1,taco2,taco3,taco4}.

Now let \(M\in\intmodnil\). On each summand \(e(\nu)M\), the commuting
dot operators are nilpotent. Hence every power series with nonzero
constant term, and in particular every inverse appearing in
\cref{taco2}, acts by a finite sum. Therefore the above completed
operators induce well-defined operators on
\[
\Theta(M)
=
\bigoplus_{\ba\ \mathrm{integral}}
e(\nu_\ba)M\delta_\ba,
\]
and hence define a \(\ftrust\)-module structure on \(\Theta(M)\).

By \cref{taco1}, the generalized \(\ba\)-weight space is exactly
\[
W_\ba(\Theta(M))
=
e(\nu_\ba)M\delta_\ba
\cong e(\nu_\ba)M.
\]
These generalized weight spaces are finite-dimensional. Indeed, the
spaces of diagrams between two fixed interval idempotents are finite
rank over the polynomial algebra generated by the dots, while the dots
act nilpotently on \(M\).

For a morphism \(f\colon M\to M'\) in \(\intmodnil\), define
\[
\Theta(f)(v\delta_\ba):=f(v)\delta_\ba.
\]
The formulas \cref{taco1}--\cref{taco4} show that \(\Theta(f)\) commutes with the action of
all generators of \(\ftrust\). Thus \(\Theta\) is a functor
\[
\Theta\colon\intmodnil\longrightarrow\fywtmod,
\]
and \cref{eq:Theta} follows.
\end{proof}

For an integral weight \(\ba\), we say that \(\ba\) is \emph{weakly increasing} if
\[
   \aleph_{i,k}=\aleph_{i,k+1}
   \quad\Longrightarrow\quad
   a_{i,k}\leq a_{i,k+1}.
\]

For \(\nu\in\Omega\), we define an integral weight \(\ba(\nu)\) as follows.
Each black strand of \(\nu=(\bi,\redkappa)\) is naturally indexed by a pair
\((r,p)\), where \(r\geq 1\) is the index of the subinterval and
\(1\leq p\leq \ell_r\) records its position from left to right inside
\(\bi^{(r)}\). Thus the corresponding label is \(\bi^{(r)}_p\). By \cref{Omega}, for each \(i\in Q_0^{(+)}\), there are precisely \(v_i\)
pairs \((r,p)\) such that \({\bi^{(r)}_p}=i\) or \(\tau i
\). We assign to each pair \((r,p)\) a number \(a'_{r,p}(\nu)\in \frac14+\mathbb Z\)
by
\[
   a'_{r,p}(\nu)=
   \begin{cases}
   n+\frac14, & r=2n+1,\\[2mm]
   -n-\frac34, & r=2n+2.
   \end{cases}
\]
For each fixed \(i\in Q_0^{(+)}\), consider the set
\[
   I_i(\nu):=\{(r,p)\mid \bi^{(r)}_p\in\{i,\tau i\}\}.
\]
We order \(I_i(\nu)\) as follows. First, we arrange the pairs in nondecreasing
order of the values \(a'_{r,p}(\nu)\). If two pairs lie in the same subinterval
\(\bi^{(t)}\), then we use the left-to-right order inside
\(\bi^{(t)}\) to break ties when \(t\) is odd, and the reverse
left-to-right order to break ties when \(t\) is even. Let
\[
   (r_{i,1},p_{i,1}),\ldots,(r_{i,v_i},p_{i,v_i})
\]
be the resulting ordered list. We define
\[
   a_{i,k}(\nu):=a'_{r_{i,k},p_{i,k}}(\nu).
\]
Doing this for all \(i\in Q_0^{(+)}\), we obtain a uniquely determined weakly
increasing integral weight, denoted by
\[
   \ba(\nu)=\bigl(a_{i,k}(\nu)\bigr)_{i\in Q_0^{(+)},\,1\leq k\leq v_i}.
\]
We also record this rearrangement by the assignment
\[
   \eta_\nu(r_{i,k},p_{i,k})=(i,k).
\]
Equivalently, \(\eta_\nu(r,p)=(i,k)\) means that the black strand originally
lying in position \(p\) of the subinterval \(\bi^{(r)}\) contributes to the
coordinate \(a_{i,k}(\nu)\). For later use, we also define 
\[
\eta_\nu^\tau(r,p):=(\tau i,k)
\]
when \(r,p\) is the pair such that \(\eta_\nu(r,p)=(i,k)\).
\begin{example}
    Suppose we have $Q_0^{(+)}=\{i,j,k\}$ with $i<j<k$ and hence $\tau k<\tau j<\tau i$. Moreover, suppose that $v_i=3$ and $v_j=v_k=4$. Hence we have
    \[
    \aleph=(i,i,i,j,j,j,j,k,k,k,k).
    \]
    Suppose $\nu\in \Omega$ such that
    \[
    e(\nu)=
    \begin{tikzpicture}
        [anchorbase,scale=1.4]
        \mirror{0}{0}{1};
        \draw[red] (.3,0)\botlabel{} -- (.3,1)\toplabel{\rlam^{\textcolor{black}{(0)}}};
        \draw[-] (.7,0)\botlabel{i} -- (.7,1)
                 (1,0)\botlabel{i} -- (1,1)
                  (1.3,0)\botlabel{k} -- (1.3,1);
        \draw[red] (1.7,0)\botlabel{} -- (1.7,1)\toplabel{\rlam^{\textcolor{black}{(1)}}};
        \draw[-] (2.1,0)\botlabel{\tau j} -- (2.1,1)
                 (2.4,0)\botlabel{\tau i} -- (2.4,1);
        \draw[red] (2.8,0)\botlabel{} -- (2.8,1)\toplabel{\rlam^{\textcolor{black}{(2)}}};
        \draw[red] (3.3,0)\botlabel{} -- (3.3,1)\toplabel{\rlam^{\textcolor{black}{(3)}}};
        \draw[-] (3.7,0)\botlabel{\tau k} -- (3.7,1);
        \draw[red] (4.1,0)\botlabel{} -- (4.1,1)\toplabel{\rlam^{\textcolor{black}{(4)}}};
         \draw[-] (4.5,0)\botlabel{j} -- (4.5,1)
                 (4.8,0)\botlabel{k} -- (4.8,1);
        \draw[red] (5.2,0)\botlabel{} -- (5.2,1)\toplabel{\rlam^{\textcolor{black}{(5)}}};
        \draw[red] (5.7,0)\botlabel{} -- (5.7,1)\toplabel{\rlam^{\textcolor{black}{(6)}}};
        \draw[-] (6.1,0)\botlabel{k} -- (6.1,1);
        \draw[red] (6.5,0)\botlabel{} -- (6.5,1)\toplabel{\rlam^{\textcolor{black}{(7)}}};
        \draw[-] (6.9,0)\botlabel{\tau j} -- (6.9,1);
        \draw[-] (7.2,0)\botlabel{\tau j} -- (7.2,1);
        \draw[red] (7.6,0)\botlabel{} -- (7.6,1)\toplabel{\rlam^{\textcolor{black}{(8)}}};
        \draw[red] (8.1,0)\botlabel{} -- (8.1,1)\toplabel{\rlam^{\textcolor{black}{(9)}}};
        \node at (8.4,.5){$\cdots$};
    \end{tikzpicture}
    \]
    In this case, we have
    \begin{gather*}
        a'_{1,1}(\nu)=a'_{1,2}(\nu)=a'_{1,3}(\nu)=\frac{1}{4},\quad a'_{2,1}(\nu)=a'_{2,2}(\nu)=-\frac{3}{4},\quad a'_{4,1}(\nu)=-\frac{7}{4},\\
        \quad a'_{5,1}(\nu)=a'_{5,2}(\nu)=\frac{9}{4},\quad a'_{7,1}(\nu)=\frac{13}{4},\quad a'_{8,1}(\nu)=a'_{8,2}(\nu)=-\frac{15}{4}.
    \end{gather*}
    Hence the weakly increasing integral $\ba(\nu)$ is 
    \begin{align*}
    \ba(\nu)=&(\overbrace{a'_{2,2}(\nu),a'_{1,1}(\nu), a'_{1,2}(\nu)}^i, \overbrace{
a'_{8,2}(\nu),a'_{8,1}(\nu),
a'_{2,1}(\nu),a'_{5,1}(\nu)
}^{j},\overbrace{a'_{4,1}(\nu),a'_{1,3}(\nu),a'_{5,2}(\nu),a'_{7,1}(\nu)}^k)\\
        =&\Big(\overbrace{-\frac{3}{4},\frac{1}{4},\frac{1}{4}}^i,\overbrace{-\frac{15}{4},-\frac{15}{4},-\frac{3}{4},\frac{9}{4}}^j,\overbrace{-\frac{7}{4},\frac{1}{4},\frac{9}{4},\frac{13}{4}}^k\Big)
    \end{align*}
    Then the assignment $\eta_\nu$ is clear from the above equation. 
\end{example}

By construction, we have
\[
   \nu_{\ba(\nu)}=\nu
   \qquad (\nu\in\Omega),
\]
and
\[
   \ba(\nu_\ba)=\ba
\]
for every weakly increasing integral weight \(\ba\). Hence
\(\nu\mapsto\ba(\nu)\) is a bijection from \(\Omega\) to the set of
weakly increasing integral weights.

\begin{prop}\label{prop:Gamma}
There is a functor
\[
   \Gamma\colon\fywtmod\longrightarrow\intmodnil
\]
such that, for every \(N\in\fywtmod\) and every \(\nu\in\Omega\),
\begin{equation}
\label{eq:Gamma-idem}
   e(\nu)\Gamma(N)=W_{\ba(\nu)}(N).
\end{equation}
Equivalently,
\begin{equation}
\label{eq:Gamma}
   \Gamma(N)
   =
   \bigoplus_{\nu\in\Omega}W_{\ba(\nu)}(N)
   =
   \bigoplus_{\substack{\ba\ \mathrm{integral}\\
                        \ba\ \mathrm{weakly\ increasing}}}
   W_\ba(N).
\end{equation}
\end{prop}

\begin{proof}
The proof is inspired by the one of \cite[Lem. 5.18]{KTWWY19}.
The idea of the constructions is to inverse the constructions in \cref{prop:Theta}. Let
\(N\in\fywtmod\), and define
\[
   \Gamma(N)
   :=
   \bigoplus_{\nu\in\Omega}W_{\ba(\nu)}(N).
\]
For each \(\nu\in\Omega\), the idempotent \(e(\nu)\) acts as the
projection onto the summand \(W_{\ba(\nu)}(N)\).

We first define the action of the dots and admissible crossings. Suppose that
\[
   \eta_\nu(r,p)=(i,k),
   \qquad
   a=a_{i,k}(\nu).
\]
Then the dot \(Y_r(\nu,p)\) acts on \(W_{\ba(\nu)}(N)\) by
\begin{equation}
\label{eq:Gamma-dot}
   Y_r(\nu,p)v
   =
   \begin{cases}
      (y_{i,k}-a)v, & r\text{ is odd},\\[1mm]
      (a-y_{i,k})v, & r\text{ is even}.
   \end{cases}
\end{equation}
These operators are nilpotent, since \(W_{\ba(\nu)}(N)\) is a generalized
weight space.

Now suppose that
\(
\bi^{(r)}_p=\bi^{(r)}_{p+1}.
\)
If \(r\) is odd, then there are \(i\in Q_0^{(+)}\) and \(k\) such that
\[
   \eta_\nu(r,p)=(i,k),
   \qquad
   \eta_\nu(r,p+1)=(i,k+1).
\]
If \(r\) is even, then the reverse ordering convention gives
\[
   \eta_\nu(r,p)=(i,k+1),
   \qquad
   \eta_\nu(r,p+1)=(i,k).
\]
In either case, the admissible crossing \(\sigma_r(\nu,p)\) acts by
\begin{equation}
\label{eq:Gamma-short-cross}
   \sigma_r(\nu,p)v=\sigma_{i,k}v.
\end{equation}

It remains to define the long crossings and long reflections. By
diagrammatic concatenation, every long crossing or long reflection is a
product of admissible crossings and the following elementary diagrams:
long crossings between the \(r\)-th and \((r+2)\)-nd subintervals, long
crossings between the \(r\)-th and \((r-2)\)-nd subintervals, and long
reflections between the first and second subintervals. Thus it suffices to
define these elementary operators.

Suppose first that a long crossing moves a strand from the \(r\)-th
subinterval to the \((r+2)\)-nd subinterval. Let its bottom and top labels be
\(\nu\) and \(\nu'\), respectively. We take the moving strand to be the
rightmost occurrence of its label in the source subinterval and the leftmost
occurrence of its label in the target subinterval. Suppose that
\[
\eta_\nu(r,p)
=
\eta_{\nu'}(r+2,p')
=
(i,k),
\qquad
a=a_{i,k}(\nu).
\]
If \(r\) is odd, then the moving strand is labeled by \(i\), and we define
\begin{equation}
\label{eq:Gamma-cross-right-odd}
\begin{aligned}
C_{r,p}^{r+2,p'}(\nu,i)v
:={}&
\Gamma_{\tau i,k}
\left(
\prod_{\substack{c\in\bR_i\\
c\neq a+1}}
(y_{i,k}+1-c)^{-1}v
\right).
\end{aligned}
\end{equation}
If \(r\) is even, then the moving strand is labeled by \(\tau i\), and we
define
\begin{equation}
\label{eq:Gamma-cross-right-even}
\begin{aligned}
C_{r,p}^{r+2,p'}(\nu,\tau i)v
:={}&
\Gamma_{i,k}
\left(
\prod_{\substack{c\in\bR_{\tau i}\\
c\neq1-a}}
(1-c-y_{i,k})^{-1}v
\right).
\end{aligned}
\end{equation}

Next suppose that a long crossing moves a strand from the \(r\)-th
subinterval to the \((r-2)\)-nd subinterval. We take the moving strand to be
the leftmost occurrence of its label in the source subinterval and the
rightmost occurrence of its label in the target subinterval. Suppose that
\[
   \eta_\nu(r,p)
   =
   \eta_{\nu'}(r-2,p')
   =
   (i,k),
   \qquad
   a=a_{i,k}(\nu).
\]
If \(r\) is odd, then \(a>\frac14\), the moving strand is labeled by \(i\),
and we define
\begin{equation}
\label{eq:Gamma-cross-left-odd}
\begin{aligned}
C_{r,p}^{r-2,p'}(\nu,i)v
:={}&
(-1)^{
\#\{c\in\bR_{\tau i}\mid c=1-a\}
}
\Gamma_{i,k}
\left(
   \prod_{\substack{c\in\bR_{\tau i}\\c\neq1-a}}
   (1-c-y_{i,k})^{-1}v
\right).
\end{aligned}
\end{equation}
If \(r\) is even, then \(a<-\frac34\), the moving strand is labeled by
\(\tau i\), and we define
\begin{equation}
\label{eq:Gamma-cross-left-even}
\begin{aligned}
C_{r,p}^{r-2,p'}(\nu,\tau i)v
:={}&
(-1)^{
\#\{c\in\bR_i\mid c=a+1\}
}
\Gamma_{\tau i,k}
\left(
   \prod_{\substack{c\in\bR_i\\c\neq a+1}}
   (y_{i,k}+1-c)^{-1}v
\right).
\end{aligned}
\end{equation}

Finally, consider the elementary long reflections. Suppose first that a
strand is moved from the first subinterval to the second. Let the bottom and
top labels be \(\nu\) and \(\nu'\), and suppose that
\[
   \eta_\nu(1,p)
   =
   \eta_{\nu'}(2,p')
   =
   (i,k).
\]
Thus the corresponding coordinate changes from \(\frac14\) to
\(-\frac34\). We define
\begin{equation}
\label{eq:Gamma-reflection-right}
\begin{aligned}
\rho_{1,p}^{2,p'}(\nu,i)v
:={}&
\Gamma_{i,k}
\left(
   \prod_{\substack{c\in\bR_{\tau i}\\c\neq\frac34}}
   (1-c-y_{i,k})^{-1}v
\right).
\end{aligned}
\end{equation}
Similarly, suppose that a strand is moved from the second subinterval to the
first and that
\[
   \eta_\nu(2,p)
   =
   \eta_{\nu'}(1,p')
   =
   (i,k).
\]
Thus the corresponding coordinate changes from \(-\frac34\) to
\(\frac14\). We define
\begin{equation}
\label{eq:Gamma-reflection-left}
\begin{aligned}
\rho_{2,p}^{1,p'}(\nu,\tau i)v
:={}&
\Gamma_{\tau i,k}
\left(
   \prod_{\substack{c\in\bR_i\\c\neq\frac14}}
   (y_{i,k}+1-c)^{-1}v
\right).
\end{aligned}
\end{equation}

All inverses appearing in
\cref{eq:Gamma-cross-right-odd,eq:Gamma-cross-right-even,
eq:Gamma-cross-left-odd,eq:Gamma-cross-left-even,
eq:Gamma-reflection-right,eq:Gamma-reflection-left}
are taken on the source generalized weight space. They are well-defined:
each denominator has nonzero constant term there, and its remaining part is
nilpotent. 

We now verify that these operators define an action of \(\intklrw\). For an
integral weight \(\ba\), recall the completion
\[
   \widehat P_\ba
   =
   \varprojlim_L
   P/
   \bigl(
      x_{i,k}-a_{i,k}
      \mid
      i\in Q_0^{(+)},\ 1\leq k\leq v_i
   \bigr)^L
\]
and the isomorphism from the proof of \cref{prop:Theta}
\[
   \Phi_\ba\colon
   \widehat P_\ba
   \xrightarrow{\sim}
   \widehat{\Pol}(\nu_\ba).
\]

Under \(\Phi_\ba\), the shifted variable \(x_{i,k}-a_{i,k}\) is identified
with the dot on the corresponding strand when \(a_{i,k}>0\), and with the
negative of that dot when \(a_{i,k}<0\). Consequently,
\cref{eq:Gamma-dot,eq:Gamma-short-cross} are the inverse change-of-variable
formulas corresponding to \cref{taco1,taco2}.

Likewise, the formulas
\cref{eq:Gamma-cross-right-odd,eq:Gamma-cross-right-even,
eq:Gamma-cross-left-odd,eq:Gamma-cross-left-even,
eq:Gamma-reflection-right,eq:Gamma-reflection-left}
are obtained by solving \cref{taco3,taco4} for the corresponding long
crossing or long reflection. For example, when \(r\) is odd and the strand
moves from the \(r\)-th to the \((r+2)\)-nd subinterval, the relevant case of
\cref{taco4} is
\[
\Gamma_{\tau i,k}(v)
=
C_{r,p}^{r+2,p'}(\nu,i)
\left(
\prod_{\substack{c\in\bR_i\\
c\neq a+1}}
(y_{i,k}+1-c)v
\right).
\]
The omitted factor with \(c=a+1\) is precisely the dot factor
contributed by the odd red strand crossed by the moving black strand
under the parity polynomial representation. Solving the displayed
identity for the long crossing gives
\cref{eq:Gamma-cross-right-odd}. The other five cases are obtained in the
same way.

Thus, after conjugation by the maps \(\Phi_\ba\), the operators defined above
agree with the polynomial action of the corresponding generators of
\(\intklrw\) on the completed parity polynomial representation. Since the parity
polynomial representation of \(\intklrw\) is faithful, its dot-adic
completion is faithful as well. Hence these operators satisfy all defining
relations of \(\intklrw\). On every generalized weight space of \(N\), the operators
\(y_{i,k}-a_{i,k}\) are nilpotent. Therefore every formal inverse appearing
above acts by a finite sum, so the completed formulas induce well-defined
operators on \(\Gamma(N)\). It follows that \(\Gamma(N)\) is an
\(\intklrw\)-module. 

Finally, if \(f\colon N\to N'\) is a morphism in \(\fywtmod\), define
\[
   \Gamma(f)
   :=
   \bigoplus_{\nu\in\Omega}
   f\big|_{W_{\ba(\nu)}(N)}.
\]
Since \(f\) commutes with the actions of \(y_{i,k}\), \(\sigma_{i,k}\),
\(\Gamma_{i,k}\), and \(\Gamma_{\tau i,k}\), it commutes with all the
operators defined above. Therefore \(\Gamma(f)\) is
\(\intklrw\)-linear, and \(\Gamma\) is a functor. This proves
\cref{eq:Gamma-idem,eq:Gamma}.
\end{proof}

\begin{theorem}
\label{thm:equivalence}
    The functors \(\Theta\) and $\Gamma$ give mutually inverse equivalences
    \[
        \fywtmod \cong \intmodnil.
    \]
\end{theorem}

\begin{proof}
Recall the intertwiners defined in \cref{eq:intertw}.
For \(i\in Q_0^{(+)}\) and \(1\leq k<v_i\), let
\(\chi_{i,k}\in\NH\) be the intertwiner corresponding to the adjacent
transposition \(s_{i,k}\) of the \(k\)-th and \((k+1)\)-st copies of
\(i\) in \(\aleph\). Thus \(\chi_{i,k}\) is the corresponding element
\(r_l\), where
\[
   l=v_1+\cdots+v_{i-1}+k.
\]
The intertwiner relations imply
\[
   \chi_{i,k}^2=1,
   \qquad
   \chi_{i,k}y_{j,l}
   =
   y_{s_{i,k}(j,l)}\chi_{i,k},
\]
and the elements \(\chi_{i,k}\) satisfy the usual braid and commutation
relations. Hence, for every \(\pi\in\Sigma\), the element
\(\chi_\pi\) obtained from a reduced expression of \(\pi\) is
well-defined and gives an isomorphism
\[
   \chi_\pi\colon
   W_\ba(N)\xrightarrow{\sim}W_{\pi\ba}(N)
\]
for every integral weight \(\ba\) and every \(N\in\fywtmod\).

We first consider the composition \(\Gamma\circ\Theta\). Let
\(M\in\intmodnil\). By \cref{eq:Theta,eq:Gamma-idem}, for every
\(\nu\in\Omega\), we have
\[
\begin{aligned}
   e(\nu)\Gamma(\Theta(M))
   &=
   W_{\ba(\nu)}(\Theta(M))                                      
   =
   e(\nu_{\ba(\nu)})M                                           
   =
   e(\nu)M.
\end{aligned}
\]
Consequently,
\[
   \Gamma(\Theta(M))
   =
   \bigoplus_{\nu\in\Omega}e(\nu)M
   =
   M
\]
as vector spaces.

Under this identification, the action of every generator of
\(\intklrw\) is unchanged. Indeed, substituting \cref{taco1} into
\cref{eq:Gamma-dot} recovers the original dot operator. Similarly,
substituting the equal-weight cases of \cref{taco2} into
\cref{eq:Gamma-short-cross} recovers the original admissible
crossing. Finally, the formulas
\cref{eq:Gamma-cross-right-odd,eq:Gamma-cross-right-even,
eq:Gamma-cross-left-odd,eq:Gamma-cross-left-even,
eq:Gamma-reflection-right,eq:Gamma-reflection-left}
were obtained by solving the corresponding cases of
\cref{taco3,taco4} for the elementary long crossing or long
reflection. Hence their substitution into \cref{taco3,taco4}
recovers the original elementary long diagrams. Since all long
crossings and long reflections are obtained from the elementary ones
and admissible crossings by concatenation, all generators of
\(\intklrw\) act as they did on \(M\). Thus the above identification
is an isomorphism of \(\intklrw\)-modules, natural in \(M\), and
\[
   \Gamma\circ\Theta
   \cong
   \operatorname{Id}_{\intmodnil}.
\]

We now consider the other composition. Let \(N\in\fywtmod\). For an
integral weight \(\ba\), set
\(
   \ba^\uparrow:=\ba(\nu_\ba).
\)
Thus \(\ba^\uparrow\) is the weakly increasing rearrangement of
\(\ba\), and
\[
   \nu_{\ba^\uparrow}=\nu_\ba.
\]
By \cref{eq:Theta,eq:Gamma-idem}, we have
\[
\begin{aligned}
   \Theta(\Gamma(N))
   =
   \bigoplus_{\ba\ \mathrm{integral}}
   e(\nu_\ba)\Gamma(N)\delta_\ba                              
   =
   \bigoplus_{\ba\ \mathrm{integral}}
   W_{\ba^\uparrow}(N)\delta_\ba.
\end{aligned}
\]

For every integral weight \(\ba\), let
\(\pi_\ba\in\Sigma\) be the 
unique permutation of minimal length such that
\[
   \pi_\ba\ba=\ba^\uparrow.
\]
Define
\[
\begin{aligned}
   \chi_N\colon N&\longrightarrow\Theta(\Gamma(N)),\\
   v&\longmapsto
   \chi_{\pi_\ba}(v)\delta_\ba,
   \qquad v\in W_\ba(N).
 \end{aligned}
\]
Since every \(\chi_{\pi_\ba}\) is an isomorphism, \(\chi_N\) is a
vector-space isomorphism. We need to show that it is
\(\ftrust\)-linear.

First consider \(y_{i,k}\). The occurrence of the strand associated
with the coordinate \((i,k)\) in \(\nu_\ba\) is precisely the
occurrence associated, under \(\eta_{\nu_\ba}\), with the coordinate
to which \((i,k)\) is sent by \(\pi_\ba\). Therefore,
using \cref{eq:Gamma-dot,taco1} and the intertwining relation for
\(\chi_{\pi_\ba}\), we obtain
\[
   y_{i,k}
   \bigl(
      \chi_{\pi_\ba}(v)\delta_\ba
   \bigr)
   =
   \chi_{\pi_\ba}(y_{i,k}v)\delta_\ba.
\]

Next consider \(\sigma_{i,k}\). If
\(a_{i,k}=a_{i,k+1}\), then $\pi_\ba$ preserves
the relative order of these two coordinates. Hence the corresponding two
strands are adjacent in the same subinterval. By
\cref{eq:Gamma-short-cross,taco2}, together with the braid relations
for the intertwiners, we obtain
\[
   \sigma_{i,k}
   \bigl(
      \chi_{\pi_\ba}(v)\delta_\ba
   \bigr)
   =
   \chi_{\pi_\ba}(\sigma_{i,k}v)\delta_\ba.
\]

Suppose now that \(a_{i,k}\neq a_{i,k+1}\), and set
\(
\bb=s_{i,k}\ba.
\)
The vector \(\sigma_{i,k}v\) has one component in \(W_\ba(N)\) and
one component in \(W_\bb(N)\). If \(a_{i,k}<a_{i,k+1}\), then
\[
\pi_\bb=\pi_\ba s_{i,k},
\]
whereas, if \(a_{i,k}>a_{i,k+1}\), then
\[
\pi_\ba=\pi_\bb s_{i,k}.
\]
These identities follow from the minimality of \(\pi_\ba\) and
\(\pi_\bb\).

We compare separately the two weight-space components of
\(
\sigma_{i,k}
\bigl(
\chi_{\pi_\ba}(v)\delta_\ba
\bigr)
\)
using the unequal-weight case of \cref{taco2}. For the component in
\(W_\ba(N)\), the desired equality follows from
\(
\chi_{\pi_\ba}f
=
\pi_\ba(f)\chi_{\pi_\ba}
\)
for every polynomial \(f\). For the component in \(W_\bb(N)\), it
follows from the preceding relation between \(\pi_\ba\) and
\(\pi_\bb\), together with the braid relations for the intertwiners.
Thus the two components agree with the corresponding components of
\(
\chi_N(\sigma_{i,k}v),
\)
and hence
\[
\sigma_{i,k}
\bigl(
\chi_{\pi_\ba}(v)\delta_\ba
\bigr)
=
\chi_N(\sigma_{i,k}v).
\]

It remains to check the generators
\(\Gamma_{i,k}\) and \(\Gamma_{\tau i,k}\). We give the argument for
\(\Gamma_{i,k}\); the verification for \(\Gamma_{\tau i,k}\) is similar. We first record their
compatibility with the intertwiners. For every \(\pi\in\Sigma\), if
\(
\pi(i,k)=(i,l),
\)
then it is easy to show that
\begin{gather}
\label{bala1}
\chi_\pi\Gamma_{i,k}
=
\Gamma_{i,l}\chi_\pi,
\qquad
\chi_\pi\Gamma_{\tau i,k}
=
\Gamma_{\tau i,l}\chi_\pi.
\end{gather}
We now check compatibility of $\chi_N$ with \(\Gamma_{i,k}\). Set
\(
\bb=\ba^-,
\)
and suppose that
\(
\pi_\ba(i,k)=(i,l).
\)
Then, by \cref{bala1}, we have
\(
\chi_{\pi_\ba}\Gamma_{i,k}
=
\Gamma_{i,l}\chi_{\pi_\ba}.
\)

Since every long crossing or long reflection occurring in
\cref{taco3} is obtained from an elementary long crossing or long
reflection by composition with admissible crossings, and compatibility
with admissible crossings has already been verified, it is enough to
check the elementary ones. We first treat the case
\(
a_{i,k}>\frac14.
\)
Set
\[
a=a_{i,k},
\qquad
m=
\#\{c\in\bR_{\tau i}\mid c=1-a\}.
\]
In this case, \cref{taco3} is given by an elementary long crossing
from the odd subinterval containing \(a\) to the odd subinterval
containing \(a-1\). Since we are in the elementary case, the moving
strand is the leftmost \(i\)-strand in the source subinterval and the
rightmost \(i\)-strand in the target subinterval. Hence lowering the
corresponding coordinate of \(\ba^\uparrow\) by \(1\) preserves the
ordering. Thus
\(
\pi_\ba\bb=\bb^\uparrow,
\)
and therefore
\(
\pi_\bb=\pi_\ba.
\)

On \(W_{\ba^\uparrow}(N)\), the dot on the moving strand acts by
\[
Y_{u_\ba(i,k)}
\bigl(
\nu_\ba,v_\ba(i,k)
\bigr)
=
y_{i,l}-a.
\]
Therefore the polynomial factor occurring in the first case of
\cref{taco3} becomes
\[
\prod_{\substack{c\in\bR_{\tau i}\\c\neq1-a}}
\left(
1-a-c-(y_{i,l}-a)
\right)
=
\prod_{\substack{c\in\bR_{\tau i}\\c\neq1-a}}
(1-c-y_{i,l}).
\]
Set
\(
F
=
\prod_{\substack{c\in\bR_{\tau i}\\c\neq1-a}}
(1-c-y_{i,l}).
\)
By \cref{eq:Gamma-cross-left-odd}, the elementary long crossing
\(C\) occurring here acts on \(W_{\ba^\uparrow}(N)\) by
\[
C(w)
=
(-1)^m
\Gamma_{i,l}
\bigl(
F^{-1}w
\bigr).
\]
Consequently, for \(v\in W_\ba(N)\),
\[
\begin{aligned}
(-1)^m
C\bigl(
F\chi_{\pi_\ba}(v)
\bigr)
=
\Gamma_{i,l}\chi_{\pi_\ba}(v)
=
\chi_{\pi_\ba}\Gamma_{i,k}(v)
=
\chi_{\pi_\bb}\Gamma_{i,k}(v).
\end{aligned}
\]
The left-hand side is precisely the action prescribed by the first
case of \cref{taco3} on
\(\chi_{\pi_\ba}(v)\delta_\ba\). Hence
\[
\Gamma_{i,k}
\bigl(
\chi_{\pi_\ba}(v)\delta_\ba
\bigr)
=
\chi_{\pi_\bb}
\bigl(
\Gamma_{i,k}v
\bigr)
\delta_\bb.
\]
The cases
\(
a_{i,k}=\frac14
\)
and
\(
a_{i,k}<0
\)
can be proved in the same way, using
\cref{eq:Gamma-reflection-right} and
\cref{eq:Gamma-cross-right-even}, respectively. 

Thus \(\chi_N\) commutes with all generators of \(\ftrust\), and hence
is an isomorphism of \(\ftrust\)-modules. The construction is natural
in \(N\), so
\[
   \Theta\circ\Gamma
   \cong
   \operatorname{Id}_{\fywtmod}.
\]

Therefore \(\Theta\) and \(\Gamma\) are mutually inverse
equivalences.
\end{proof}

\begin{remark}
Suppose that \((\Qui,\tau)\) is of diagonal type as in \cref{sec:diagnoal}. In this case,
the shifted iYangian construction reduces to the ordinary
shifted Yangian construction associated with \(\Qui_{+}\). On the
diagrammatic side, the Morita equivalence in
\cref{thm:diagonal} reduces the corresponding oKLRW algebra to the
ordinary KLRW algebra associated with \(\Qui_{+}\), and a similar
reduction can be applied to the relevant interval truncations.

We expect that, after applying this Morita reduction and unfolding the diagrammatic data paired by \(\tau\), the folded interval description used here can be translated into the metric/parity description used in \cite{KTWWY19}. Under these identifications, up to the change of normalization of \(\hbar\), \cref{thm:equivalence} essentially recovers the equivalence in \cite[Thm.~5.19]{KTWWY19}. 
\end{remark}

\begin{remark}
In the untwisted setting, the equivalence between integral weight modules over
truncated shifted Yangians and nilpotent modules over the corresponding
diagrammatic algebras restricts further to equivalences involving categories
\(\mathcal O^{\pm}\); see \cite[Def.~5.20 and Thm.~5.21]{KTWWY19}. On the
Yangian side, these categories are defined by imposing suitable one-sided
boundedness conditions on the weights. On the KLRW side, they correspond to
module categories over appropriate one-sided quotients of the diagrammatic
algebra.

At present, we do not have a satisfactory intrinsic definition of category
\(\mathcal O\) for shifted iYangians due to the lack of triangular decomposition. For this reason, \cref{thm:equivalence} is formulated only at the level of integral weight modules.
Nevertheless, the equivalence above suggests a possible diagrammatic approach
to such a definition. Namely, one expects that suitable one-sided quotients of
the interval oKLRW algebra \(\intklrw\), analogous to the corresponding
quotients in the untwisted setting, should give rise under the above
equivalence to natural candidates for categories
\(\mathcal O^{\pm}\) on the twisted-Yangian side.
\end{remark}

\bibliographystyle{alpha}
\bibliography{KLRTY}

@article{BFN18,
    AUTHOR = {Braverman, A. and Finkelberg, M. and Nakajima,
              H.},
     TITLE = {Towards a mathematical definition of {C}oulomb branches of
              3-dimensional {$\mathcal{N}=4$} gauge theories, {II}},
   JOURNAL = {Adv. Theor. Math. Phys.},
  FJOURNAL = {Advances in Theoretical and Mathematical Physics},
    VOLUME = {22},
      YEAR = {2018},
    NUMBER = {5},
     PAGES = {1071--1147},
      ISSN = {1095-0761,1095-0753},
   MRCLASS = {57R57 (14J33 14N35 16G20 17B67 81T13)},
  MRNUMBER = {3952347},
MRREVIEWER = {Dave\ Auckly},
       DOI = {10.4310/ATMP.2018.v22.n5.a1},
       URL = {https://doi.org/10.4310/ATMP.2018.v22.n5.a1},
}

@article {KTWWY19,
    AUTHOR = {Kamnitzer, J. and Tingley, P. and Webster, B. and
              Weekes, A. and Yacobi, O.},
     TITLE = {On category {$\mathcal{O}$} for affine {G}rassmannian slices and
              categorified tensor products},
   JOURNAL = {Proc. Lond. Math. Soc. (3)},
  FJOURNAL = {Proceedings of the London Mathematical Society. Third Series},
    VOLUME = {119},
      YEAR = {2019},
    NUMBER = {5},
     PAGES = {1179--1233},
      ISSN = {0024-6115,1460-244X},
   MRCLASS = {14M15 (17B37 20G42)},
  MRNUMBER = {3968721},
MRREVIEWER = {Huafeng\ Zhang},
       DOI = {10.1112/plms.12254},
       URL = {https://doi-org.proxy.bib.uottawa.ca/10.1112/plms.12254},
}

@article {Ara62,
    AUTHOR = {Araki, S.},
     TITLE = {On root systems and an infinitesimal classification of
              irreducible symmetric spaces},
   JOURNAL = {J. Math. Osaka City Univ.},
  FJOURNAL = {Journal of Mathematics. Osaka City University},
    VOLUME = {13},
      YEAR = {1962},
     PAGES = {1--34},
      ISSN = {0449-2773},
   MRCLASS = {17.30 (53.73)},
  MRNUMBER = {153782},
MRREVIEWER = {S.\ Helgason},
}

@article{En09,
    AUTHOR = {Enomoto, N.},
     TITLE = {A quiver construction of symmetric crystals},
   JOURNAL = {Int. Math. Res. Not. IMRN},
  FJOURNAL = {International Mathematics Research Notices. IMRN},
      YEAR = {2009},
    NUMBER = {12},
     PAGES = {2200--2247},
      ISSN = {1073-7928,1687-0247},
   MRCLASS = {17B37 (17B67)},
  MRNUMBER = {2511909},
MRREVIEWER = {Alistair\ Savage},
       DOI = {10.1093/imrn/rnp014},
       URL = {https://doi.org/10.1093/imrn/rnp014},
}

@article{EK07,
    AUTHOR = {Enomoto, N. and Kashiwara, M.},
     TITLE = {Symmetric crystals and affine {H}ecke algebras of type {B}},
   JOURNAL = {Proc. Japan Acad. Ser. A Math. Sci.},
  FJOURNAL = {Japan Academy. Proceedings. Series A. Mathematical Sciences},
    VOLUME = {82},
      YEAR = {2006},
    NUMBER = {8},
     PAGES = {131--136},
      ISSN = {0386-2194},
   MRCLASS = {20G05 (20C08)},
  MRNUMBER = {2279279},
MRREVIEWER = {Eric\ C.\ Rowell},
       URL = {http://projecteuclid.org/euclid.pja/1162820093},
}

@book {ENGO15,
    AUTHOR = {Etingof, P. and Gelaki, S. and Nikshych, D. and
              Ostrik, V.},
     TITLE = {Tensor categories},
    SERIES = {Mathematical Surveys and Monographs},
    VOLUME = {205},
 PUBLISHER = {American Mathematical Society, Providence, RI},
      YEAR = {2015},
     PAGES = {xvi+343},
      ISBN = {978-1-4704-2024-6},
   MRCLASS = {18D10 (16T05)},
  MRNUMBER = {3242743},
MRREVIEWER = {Julien\ Bichon},
       DOI = {10.1090/surv/205},
       URL = {https://doi.org/10.1090/surv/205},
}

@article {GKLO05,
    AUTHOR = {Gerasimov, A. and Kharchev, S. and Lebedev, D. and Oblezin,
              S.},
     TITLE = {On a class of representations of the {Y}angian and moduli
              space of monopoles},
   JOURNAL = {Comm. Math. Phys.},
  FJOURNAL = {Communications in Mathematical Physics},
    VOLUME = {260},
      YEAR = {2005},
    NUMBER = {3},
     PAGES = {511--525},
      ISSN = {0010-3616,1432-0916},
   MRCLASS = {53C07 (17B37 53D17 53D30)},
  MRNUMBER = {2182434},
MRREVIEWER = {Olivier\ G.\ Schiffmann},
       DOI = {10.1007/s00220-005-1417-3},
       URL = {https://doi.org/10.1007/s00220-005-1417-3},
}

@article {KWWY14,
    AUTHOR = {Kamnitzer, J. and Webster, B. and Weekes, A. and Yacobi,
              O.},
     TITLE = {Yangians and quantizations of slices in the affine
              {G}rassmannian},
   JOURNAL = {Algebra Number Theory},
  FJOURNAL = {Algebra \& Number Theory},
    VOLUME = {8},
      YEAR = {2014},
    NUMBER = {4},
     PAGES = {857--893},
      ISSN = {1937-0652,1944-7833},
   MRCLASS = {17B37 (14D24 14M15 20G15 53D55)},
  MRNUMBER = {3248988},
MRREVIEWER = {Christian\ Ohn},
       DOI = {10.2140/ant.2014.8.857},
       URL = {https://doi.org/10.2140/ant.2014.8.857},
}

@article{LWW,
    author = {Lu, K. and Wang, W. and Weekes, A.},
    title = {Shifted twisted {Y}angians and affine {G}rassmannian islices},
    year={2025},
JOURNAL = {preprint},
note    = {\href{https://arxiv.org/abs/2510.10652}{arXiv:2510.10652}},
URL = {https://arxiv.org/abs/2510.10652},
    EPRINT = {2510.10652},
}

@article{LWW2,
    author = {Lu, K. and Wang, W. and Weekes, A.},
    title = {Shifted twisted {Y}angians of quasi-split {ADE} types},
    year={2025},
JOURNAL = {preprint},
note    = {\href{https://arxiv.org/abs/2512.19998}{arXiv:2512.19998}},
URL = {https://arxiv.org/pdf/2512.19998},
    EPRINT = {2512.19998},
}

@article{Web16,
    author = {Webster, B.},
    title = {Koszul duality between {H}iggs and {C}oulomb categories $\mathcal{O}$},
    year={2016},
JOURNAL = {preprint},
note    = {\href{https://arxiv.org/pdf/1611.06541}{arXiv:1611.06541}},
URL = {https://arxiv.org/abs/1611.06541},
    EPRINT = {1611.06541},
}

@article {KWWY24,
    AUTHOR = {Kamnitzer, J. and Webster, B. and Weekes, A. and Yacobi,
              O.},
     TITLE = {Lie algebra actions on module categories for truncated shifted
              {Y}angians},
   JOURNAL = {Forum Math. Sigma},
  FJOURNAL = {Forum of Mathematics. Sigma},
    VOLUME = {12},
      YEAR = {2024},
     PAGES = {Paper No. e18, 69},
      ISSN = {2050-5094},
   MRCLASS = {20G05 (14M15 17B67)},
  MRNUMBER = {4699878},
MRREVIEWER = {Linliang\ Song},
       DOI = {10.1017/fms.2024.3},
       URL = {https://doi-org.proxy.bib.uottawa.ca/10.1017/fms.2024.3},
}

@incollection {BD17,
    AUTHOR = {Brundan, J. and Davidson, N.},
     TITLE = {Categorical actions and crystals},
 BOOKTITLE = {Categorification and higher representation theory},
    SERIES = {Contemp. Math.},
    VOLUME = {683},
     PAGES = {105--147},
 PUBLISHER = {Amer. Math. Soc., Providence, RI},
      YEAR = {2017},
      ISBN = {978-1-4704-2460-2},
   MRCLASS = {17B67 (16G99 17B10 18D10)},
  MRNUMBER = {3611712},
MRREVIEWER = {Sorin\ D\u asc\u alescu},
       DOI = {10.1090/conm/683},
       URL = {https://doi-org.proxy.bib.uottawa.ca/10.1090/conm/683},
}

@article{LWW26,
    author = {Lu, K. and Wang, W. and Weekes, A.},
    title = {Shifted affine iquantum groups of quasi-split {ADE} types},
    year={2026},
JOURNAL = {preprint},
note    = {\href{https://arxiv.org/abs/2603.28446}{arXiv:2603.28446}},
URL = {https://arxiv.org/pdf/2603.28446},
    EPRINT = {2603.28446},
}

@article{LP26,
    author = {Li, J. and Przezdziecki, T.},
    title = {{GKLO} representations for shifted quantum affine symmetric pairs},
    year={2026},
JOURNAL = {preprint},
note    = {\href{https://arxiv.org/abs/2603.07250}{arXiv:2603.07250}},
URL = {https://arxiv.org/pdf/2603.07250},
    EPRINT = {2603.07250},
}

@article{Rou,
    author = {Rouquier, R.},
    title = {2-{K}ac-{M}oody algebras},
    year={2008},
JOURNAL = {preprint},
note    = {\href{https://arxiv.org/abs/0812.5023}{arXiv:0812.5023}},
URL = {https://arxiv.org/abs/0812.5023},
    EPRINT = {0812.5023},
}

@article {SSX25,
    AUTHOR = {Shen, Y. and Su, C. and Xiong, R.},
     TITLE = {Quivers with involutions and shifted twisted {Y}angians via Coulomb Branches},
      Journal = {to appear on Comm. Math. Phys.},
ARCHIVEPREFIX = {arXiv},
note    = {\href{https://arxiv.org/abs/2510.12118}{arXiv:2510.12118}},
URL = {https://arxiv.org/abs/2510.12118},
    EPRINT = {2510.12118},
}

@article {KL09,
    AUTHOR = {Khovanov, M. and Lauda, A.},
     TITLE = {A diagrammatic approach to categorification of quantum groups.
              {I}},
   JOURNAL = {Represent. Theory},
  FJOURNAL = {Representation Theory. An Electronic Journal of the American
              Mathematical Society},
    VOLUME = {13},
      YEAR = {2009},
     PAGES = {309--347},
      ISSN = {1088-4165},
   MRCLASS = {17B37},
  MRNUMBER = {2525917},
MRREVIEWER = {Fan\ Xu},
       DOI = {10.1090/S1088-4165-09-00346-X},
       URL = {https://doi-org.proxy.bib.uottawa.ca/10.1090/S1088-4165-09-00346-X},
}

@article {Web15,
    AUTHOR = {Webster, B.},
     TITLE = {Canonical bases and higher representation theory},
   JOURNAL = {Compos. Math.},
  FJOURNAL = {Compositio Mathematica},
    VOLUME = {151},
      YEAR = {2015},
    NUMBER = {1},
     PAGES = {121--166},
      ISSN = {0010-437X,1570-5846},
   MRCLASS = {17B37 (17B10 18D05)},
  MRNUMBER = {3305310},
MRREVIEWER = {Kevin\ D.\ Coulembier},
       DOI = {10.1112/S0010437X1400760X},
       URL = {https://doi-org.proxy.bib.uottawa.ca/10.1112/S0010437X1400760X},
}

@article {Web17a,
    AUTHOR = {Webster, B.},
     TITLE = {Knot invariants and higher representation theory},
   JOURNAL = {Mem. Amer. Math. Soc.},
  FJOURNAL = {Memoirs of the American Mathematical Society},
    VOLUME = {250},
      YEAR = {2017},
    NUMBER = {1191},
     PAGES = {v+141},
      ISSN = {0065-9266,1947-6221},
      ISBN = {978-1-4704-2650-7; 978-1-4704-4206-4},
   MRCLASS = {57M27 (17B10 18D05 57M25)},
  MRNUMBER = {3709726},
MRREVIEWER = {Stefan\ K.\ Friedl},
       DOI = {10.1090/memo/1191},
       URL = {https://doi-org.proxy.bib.uottawa.ca/10.1090/memo/1191},
}

@article {Prz23,
    AUTHOR = {Prze\'zdziecki, T.},
     TITLE = {Representations of orientifold {K}hovanov-{L}auda-{R}ouquier
              algebras and the {E}nomoto-{K}ashiwara algebra},
   JOURNAL = {Pacific J. Math.},
  FJOURNAL = {Pacific Journal of Mathematics},
    VOLUME = {322},
      YEAR = {2023},
    NUMBER = {2},
     PAGES = {407--441},
      ISSN = {0030-8730,1945-5844},
   MRCLASS = {81R50 (17B37 18N25 20C08)},
  MRNUMBER = {4592237},
MRREVIEWER = {Youjun\ Tan},
       DOI = {10.2140/pjm.2023.322.407},
       URL = {https://doi-org.proxy.bib.uottawa.ca/10.2140/pjm.2023.322.407},
}

@article {LWZ25,
    AUTHOR = {Lu, K. and Wang, W. and Zhang, W.},
     TITLE = {Affine {$\imath$}quantum groups and twisted {Y}angians in
              {D}rinfeld presentations},
   JOURNAL = {Comm. Math. Phys.},
  FJOURNAL = {Communications in Mathematical Physics},
    VOLUME = {406},
      YEAR = {2025},
    NUMBER = {5},
     PAGES = {Paper No. 98, 36},
      ISSN = {0010-3616,1432-0916},
   MRCLASS = {17B37 (81R50)},
  MRNUMBER = {4887617},
       DOI = {10.1007/s00220-025-05263-z},
       URL = {https://doi.org/10.1007/s00220-025-05263-z},
}

@article{LZ24,
  title={A {D}rinfeld type presentation of twisted {Y}angians of quasi-split type},
  author={Lu, K. and Zhang, W.},
  journal={to appear on Commun. Contemp. Math.},
note    = {\href{https://arxiv.org/pdf/2408.06981}{arXiv:2408.06981}},
  year={2026}
}

@article{Wang25,
  title={Quivers with Involutions and Shifted Twisted {Y}angians via {C}oulomb Branches II},
  author={Wang, Z.},
  journal={preprint},
note    = {\href{https://arxiv.org/abs/2601.00039}{arXiv:2601.00039}},
  year={2025}
}

@article{LPTTW25,
  title={Shifted twisted Yangians and finite W-algebras of classical type},
  author={Lu, Kang and Peng, Yung-Ning and Tappeiner, Lukas and Topley, Lewis and Wang, Weiqiang},
  year={2025},
JOURNAL = {preprint},
note    = {\href{https://arxiv.org/pdf/2505.03316}{arXiv:2505.03316}},
URL = {https://arxiv.org/pdf/2505.03316},
    EPRINT = {2505.03316},
}

@article {SSS25,
    AUTHOR = {Salmasian, H. and Savage, A. and Shen, Y.},
     TITLE = {The disoriented skein and iquantum {B}rauer categories},
   JOURNAL = {Forum Math. Sigma},
  FJOURNAL = {Forum of Mathematics. Sigma},
    VOLUME = {14},
      YEAR = {2026},
     PAGES = {Paper No. e11, 41},
      ISSN = {2050-5094},
   MRCLASS = {18M30 (17B37)},
  MRNUMBER = {5019638},
       DOI = {10.1017/fms.2025.10148},
       URL = {https://doi-org.proxy.bib.uottawa.ca/10.1017/fms.2025.10148},
}

@article{VV11,
    AUTHOR = {Varagnolo, M. and Vasserot, E.},
     TITLE = {Canonical bases and affine {H}ecke algebras of type {B}},
   JOURNAL = {Invent. Math.},
  FJOURNAL = {Inventiones Mathematicae},
    VOLUME = {185},
      YEAR = {2011},
    NUMBER = {3},
     PAGES = {593--693},
      ISSN = {0020-9910,1432-1297},
   MRCLASS = {20C08 (16G20 17B10 17B37)},
  MRNUMBER = {2827096},
MRREVIEWER = {Jonathan\ Brundan},
       DOI = {10.1007/s00222-011-0314-y},
       URL = {https://doi.org/10.1007/s00222-011-0314-y},
}

@article{We19,
  title={Generators for {C}oulomb branches of quiver gauge theories},
  author={Weekes, A.},
  journal={preprint},
    note={\href{https://arxiv.org/abs/1903.07734}{arXiv:1903.0773}},
  year={2019}
}

\end{document}